\documentclass[11pt,letterpaper,reqno]{amsart}

\numberwithin{equation}{section}

\usepackage{floatrow}

\usepackage{amssymb}
\usepackage{amsmath}
\usepackage{mathrsfs}
\usepackage{dsfont}
\usepackage{verbatim}
\usepackage{stmaryrd, MnSymbol}
\usepackage{enumerate, xspace}
\usepackage{changebar}
\usepackage{hyperref}
\usepackage{comment}
\usepackage{mathtools}
\mathtoolsset{showonlyrefs,showmanualtags}
\usepackage{tikz}
\usepackage{tikz-cd}
\usepackage{youngtab}
\usetikzlibrary{graphs,shapes}
\usepackage[edges]{forest}
\usepackage{algorithm2e}
\usepackage{cancel}
\usepackage{here}
\usepackage{amsthm}
\usepackage{caption}
\usepackage{ulem}
\usepackage[foot]{amsaddr}

\newcommand{\xc}[2]{ \xcancel{\color{#1}{#2}} }

\newtheorem{theorem}{Theorem}[section]
\newtheorem*{theorem*}{Theorem}
\newtheorem{lemma}[theorem]{Lemma}
\newtheorem{corollary}[theorem]{Corollary}
\newtheorem{proposition}[theorem]{Proposition}
\newtheorem*{proposition*}{Proposition}

\newtheorem{remark}[theorem]{Remark}

\numberwithin{equation}{section}

\newtheorem{claim}{Claim}

\renewcommand{\a}{\alpha}
\renewcommand{\b}{\beta}
\newcommand{\e}{\varepsilon}

\newcommand{\q}{\mathbf{q}}

\def\i{{\mathrm{i}}}
\def\R{{\mathbb{R}}}
\def\N{{\mathbb{N}}}
\def\Z{{\mathbb{Z}}}

\def\E{{\mathbb{E}}}

\def\E{{\mathbb{E}}}

\title{Scaling limits of solitons in the Box-Ball system}

\author{Stefano Olla${}^{1}$}
\address{${}^{1}$ CEREMADE, Universit\'e Paris Dauphine - PSL Research University
\emph{and}  Institut Universitaire de France \\
\emph{and} GSSI, L'Aquila}

\author{Makiko Sasada${}^{2}$}
\address{${}^{2}$Graduate School of Mathematical Sciences, University of Tokyo}

\author{Hayate Suda${}^{3}$}
\address{${}^{3}$Department of Physics, Institute of Science Tokyo}


\begin{document}

\begin{abstract}

    We study the space-time scaling limits of solitons in the box-ball system with random initial distribution. In particular, we show that any recentered tagged soliton converges to a Brownian motion in the diffusive space-time scale, and also prove the large deviation principle for the tagged soliton under certain shift-ergodic invariant distributions, including Bernoulli product measures and two-sided Markov distributions. Furthermore, in the diffusive space-time scaling, we show that two tagged solitons converge to the same Brownian motion even if they are macroscopically far apart.
     
\end{abstract}

\maketitle

\section{Introduction}

    An integrable many-body system is a deterministic dynamical system consisting of an infinite type of quasi-local conserved quantities that behave like particles interacting with each other. These quasi-local conserved quantities are called quasi-particles.  Solitary waves (solitons) in solitonic systems are examples of quasi-particles.
    Recently, integrable many-body systems have attracted much attention from the viewpoint of non-equilibrium statistical mechanics, and in particular, generalized hydrodynamics, which describes the macroscopic behavior of quasi-particles, see the reviews \cite{D,Sp} and references therein. 
    In the Euler space-time scale, it is expected that the hydrodynamics is described by the following generalized hydrodynamic equation (GHD equation) for $\mathbf{y}(u,t) = \left( y_{a}\left(u,t\right)\right)_{a}$ of the universal form, regardless of models : 
        \begin{align}
            \partial_t y_{a}\left(u,t\right) + \partial_u \left( v^{\mathrm{eff}}_{a}\left(\mathbf{y}\left(u,t\right)\right)y_{a}\left(u,t\right) \right) = 0,
        \end{align}
    where $y_{a}\left(u,t\right)$ is the macroscopic density of quasi-particles of type $a$ at macroscopic coordinate $(u,t), u \in \R, t \ge 0$, and $v^{\mathrm{eff}}_{a}$ is called the effective velocity of quasi-particles of type $a$. The specific form of the effective velocity depends on the scattering rule between quasi-particles, and this is where the differences among models arise. 
    Although such studies have been rapidly developed in the physics literature, and the GHD theory is expected to be applicable for a wide class of integrable systems including classical and quantum gases, chains and field theory models, very few rigorous results are obtained till now, which are for the hard rods dynamics and its generalization \cite{BDS,FFGS}, and the box-ball system (BBS) \cite{CS}. 

    While studies at the Euler scale have been progressing, there was no clear physical prediction on the behavior in a longer time scale. Therefore, it is important to obtain mathematically rigorous results on specific models in order to derive the universality for integrable many-body systems in the diffusive scale. 
    Recently, by \cite{FO}, the fluctuations for hard-rods in diffusive scale has been proved rigorously. 
    The difference from diffusive fluctuations in chaotic systems is the strong correlations between quasi-particles of the same type, i.e., { quasi-particles of the same type starting at macroscopic distance converge to the same Brownian motion. }
    However, the scattering rule in the hard-rods does not depend on the velocity of quasi-particles. This is a different feature from general integrable models, and no results have been known for the case where the scattering rule depends on velocities of the quasi-particles. 

    In this paper, we consider the BBS, which is a solitonic system with a scattering rule depending on velocities of quasi-particles (solitons). We rigorously show that any tagged soliton converges to a Brownian motion in diffusive space-time scale, and also prove the large deviation principle for the tagged soliton. This is the first mathematical result { for the central limit theorem and the large deviation principle for quasi-particles} in integrable systems with scattering rules depending on velocities of quasi-particles, unlike the hard-rods. Furthermore, {we} rigorously prove that solitons of the same type converge to the same Brownian motion, i.e., strong correlations between quasi-particles as observed in the hard-rods. In order to roughly describe the results, we first introduce the BBS below. 

    The BBS is a one-dimensional cellular automaton introduced by \cite{TS}, whose integrable structure has been extensively studied in the past, see the review \cite{IKT} for details. The BBS exhibits solitonic behavior and is understood as a discrete counterpart of the KdV equation, {which is a central example of an integrable system having solitary wave solutions}. The configuration space is $\left\{ 0, 1 \right\}^{\Z}$, where for $\eta \in \{0,1\}^{\Z}$ and $x \in \Z$, $\eta(x) = 1$ means that there is a ball at $x$, and $\eta(x) = 0$ means that $x$ is empty.
    When the total number of $1$s in $\eta \in \left\{ 0, 1 \right\}^{\Z}$ is finite, the one-step time evolution $\eta \mapsto T\eta \in \left\{ 0, 1 \right\}^{\Z}$ is described by the following rules : 
            \begin{itemize}
                \item An empty carrier enters the system from the left end (i.e. $-\infty$) and moves to the right {end (i.e. $\infty$)};
                \item If there is a ball at site $x$, then the carrier picks up the ball;
                \item If the site $x$ is empty and the carrier is not empty, then the carrier puts down a ball;
                \item Otherwise, the carrier just {passes} through.
            \end{itemize}
    Clearly, the total number of $1$s are conserved. If we denote by $W(x)$ the number of balls on the carrier at $x$, then $W(\cdot)$ satisfies $W(x) = 0$ for any $|x| > L$ with sufficiently large $L > 0$, and
        \begin{align}\label{def:carrier}
            W(x) - W(x-1) = 
            \begin{dcases}
                1 & \text{ if } \eta(x) = 1, \\
                -1 & \text{ if } \eta(x) = 0 \text{ and } W(x - 1) > 0, \\
                0 & \text{otherwise.}
            \end{dcases}
        \end{align}
    In addition, $T\eta$ can be represented by using $W(\cdot)$ as 
        \begin{align}\label{def:T_carrier}
            T\eta(x) = \eta(x) - W(x) + W(x-1). 
        \end{align}
    Figure \ref{ex:def_half} shows an example how $T\eta$ can be obtained from $\eta$. 
            \begin{figure}[H]
                    \footnotesize
                    \centering
                    \begin{equation}
                        \begin{matrix}
                         \eta : & \dots & 1 & 1 & 0 & 0 & 0 & 1 & 1 &  1 & 0 & 0 & 1 &  1 & 0 & 0 &  1& 0 & 1 &  1 & 0 & 0 & 0\dots   \\ 
                         W : & 0 & 1 & 2 & 1 & 0 & 0 & 1 & 2 & 3 & 2 & 1 & 2 & 3 & 2 & 1 & 2 & 1 & 2 & 3 & 2 & 1 & 0\dots   \\
                         T\eta : & \dots & 0 & 0 & 1 & 1 & 0 & 0 & 0 & 0 & 1 & 1 & 0 & 0 & 1 & 1 & 0 & 1 & 0 & 0 & 1 & 1 & 1\dots 
                        \end{matrix}
                    \end{equation}
                    \caption{$W$ and $T\eta$ obtained from $\eta =\dots 1100011100110010110000\dots$, where $\dots$ represents the consecutive $0$s.}\label{ex:def_half}
            \end{figure}
    \noindent It is known that the above rule can be extended to $\eta \in \Omega \subset \left\{ 0, 1 \right\}^{\Z}$, where 
        \begin{align}
            \Omega := \left\{ \eta \in \left\{0,1 \right\}^{\Z} \ ; \ \exists \lim_{x \to \infty} \frac{1}{x} \sum_{y = 1}^{x} \eta\left(y\right) < \frac{1}{2}, \ \exists \lim_{x \to \infty} \frac{1}{x} \sum_{y = 1}^{x} \eta\left(-y\right) < \frac{1}{2}  \right\},
        \end{align}
    see Section \ref{sec:BBS} for details. We note that by \cite{CKST}, more detailed results are obtained for the configuration space in which the dynamics of the BBS can be defined via the Pitman transform. 
    
    In recent years, the BBS started from random initial configurations, called the randomized BBS, has been studied in terms of its statistical aspects; characterizations of classes of invariant measures for the randomized BBS \cite{CS2,CKST,FG}, limit theorems under invariant measures \cite{CKST, FNRW, KL, KLO18, LLP, LLPS, S}. Also, the randomized BBS has been studied from the viewpoint of hydrodynamics for integrable systems \cite{CS, KMP, KMP2, KMP3}. Currently, only the BBS and hard-rods are known to be mathematically tractable models for deriving hydrodynamics from integrable systems, and thus the BBS is recognized as an important model in statistical mechanics. 

    In this article, we consider the BBS on the state space $\{0,1\}^{\Z}$ {under invariant measures}, and derive the scaling limits of the tagged soliton. 
    To give an overview of our results, we introduce the law of large numbers for the tagged soliton proved in \cite{FNRW}. 
    As mentioned at the beginning, the BBS is a soliton system and there are infinite types of solitons in the BBS labeled by positive integers $k \in \N$ respectively, called $k$-soliton. A $k$-soliton consists of $k$ $1$s and $k$ $0$s and is identified by a certain algorithm, see Section \ref{subsec:soliton} for details.  If there are only $k$-solitons {in the configuration}, then they move to right with velocity $k$ at each time evolution. When solitons of different {size}s exist and a $k$-soliton is to the left of an $\ell$-soliton with $\ell < k$, such solitons will meet at some time and {\it phase shift} will occur between them. In particular, during the interaction, the smaller soliton is stranded and cannot move. For example, in the following figure showing the time evolution of BBS, red \textcolor{red}{$1$}s and \textcolor{red}{$0$}s constitute a 2-soliton and \textcolor{blue}{blue} \textcolor{blue}{$1$} and \textcolor{blue}{$0$} constitute a 1-soliton, and while the 2-soliton tries to overtake the 1-soliton, the 1-soliton cannot move from its position. 
                    \begin{equation}
                        \begin{matrix}
                        \eta & \dots & \color{red}{1} & \color{red}{1} & \color{red}{0} & \color{red}{0} & 0 & \color{blue}{1} & \color{blue}{0} & 0 & 0 & 0 & 0 & 0 & 0 & 0 & 0 & 0 & \dots \\
                        T\eta & \dots & 0 & 0 & \color{red}{1} & \color{red}{1} & \color{red}{0} & \color{red}{0} & \color{blue}{1} & \color{blue}{0} & 0 & 0 & 0 & 0 & 0 & 0 & 0 & 0 & \dots \\
                        T^2\eta & \dots & 0 & 0 & 0 & 0 & \color{red}{1} & \color{red}{1} & \color{blue}{0} & \color{blue}{1} & \color{red}{0} & \color{red}{0} & 0 & 0 & 0 & 0 & 0 & 0 & \dots \\
                        T^3\eta & \dots & 0 & 0 & 0 & 0 & 0 & 0 & \color{blue}{1} & \color{blue}{0} & \color{red}{1} & \color{red}{1} & \color{red}{0} & \color{red}{0} & 0 & 0 & 0 & 0 & \dots \\
                        T^4\eta & \dots & 0 & 0 & 0 & 0 & 0 & 0 & 0 & \color{blue}{1} & \color{blue}{0} & 0 & \color{red}{1} & \color{red}{1} & \color{red}{0} & \color{red}{0} & 0 & 0 & \dots 
                        \end{matrix}
                    \end{equation}
    {After the interaction, 1-soliton is shifted backward $2$ sites and 2-soliton is shifted forward $2$ sites from where they should have come without interaction. This is the phase shift.} Thus, given a random initial configuration, even if the time evolution rule of the BBS is deterministic, the position of the tagged $k$-soliton at time $n$ is randomized by the random presence of other solitons of different sizes{, which is a random environment for the tagged soliton}. 
    In \cite{FNRW}, the authors show that the tagged soliton satisfies the law of large numbers (LLN) when the initial distribution ${\mu}$ is invariant for $T$ and {shift-ergodic}. Bernoulli product measures {of uniform density ${\mu}\left(\eta\left(x\right) = 1\right) = {\mu}\left(\eta\left(0\right) = 1\right) < 1/2$, $x \in \Z$,} and two-sided {space-homogeneous} Markov distributions supported on $\Omega$ are important examples of ${\mu}$ satisfying the assumptions below.
        \begin{theorem*}[\cite{FNRW}]
            Let $X_{k}(n)$ be the position of the leftmost component of a tagged $k$-soliton at time $n$ and ${\mu}$ be a probability measure on $\{0,1\}^{\Z}$ satisfying the following. 
                \begin{itemize}
                    \item ${\mu}(\Omega) = 1$. 
                    \item ${\mu}$ is an invariant measure of the BBS, {namely $T{\mu}={\mu}$.}
                    \item ${\mu}$ is a shift-ergodic measure. 
                \end{itemize}
            Assume that for some $k \in \N$, $k$-solitons exist with positive probability under ${\mu}$. Then, there exists some $v^{\mathrm{eff}}_{k} =v^{\mathrm{eff}}_k({\mu}) > 0$ such that 
                \begin{align}
                    \lim_{n \to \infty} \frac{X_{k}(n)}{n} = v^{\mathrm{eff}}_{k} \quad {\mu}\text{-a.s.}
                \end{align}
        \end{theorem*}
    {The constant $v^{\mathrm{eff}}_{k}$ is called the effective velocity for $k$-solitons. A characterization formula for $v^{\mathrm{eff}}_{k}$, $k \in \N$ has been obtained by \cite[(1.12)]{FNRW}. We will present an alternative formula for $v^{\mathrm{eff}}_{k}$, see Proposition \ref{prop:char_velo} for details. Also, we give a different proof for \cite[(1.12)]{FNRW}, see Remark \ref{rem:eff_v}}.

    Our main results are the central limit theorem (CLT) and the large deviation principle (LDP) {for the increment $Y_{k}\left(n\right) := X_{k}\left(n\right) - X_{k}\left(0\right)$,} corresponding to the {LLN}. As a by-product, we also show {the LLN in $\mathbb{L}^{p}$, $p \ge 1$.} 
        \begin{claim}[Limit theorems for a tagged soliton]\label{thm:intro_LT}
            Assume that ${\mu}$ is a {space-homogeneous} Bernoulli product measure or two-sided Markov distribution supported on $\Omega$ {and that $\mu\left(\eta\left(0\right) = \eta\left(1\right) = 1\right) > 0$.}
            Then, for any $k \in \N$, we have the following.
                \begin{enumerate}

                    \item \label{item:claim1} Under {$\mu$}, the diffusive space-time scaling, the step-interpolation of the discrete-time process $n \mapsto {Y_{k}\left(n\right)} -  v^{\mathrm{eff}}_{k} n$ converges weakly to a centered Brownian motion $B_{k}(t)$ with variance $D_{k} = D_{k}\left({\mu}\right) > 0$. 

                    \item \label{item:claim2} {Under $\mu$}, the sequence $\left({Y_{k}\left(n\right)} / n\right)_{n \in \N}$ satisfies the LDP with a smooth convex rate function.  

                    \item \label{item:claim3} { For any $p \ge 1$, we have 
                    \begin{align}
                        \lim_{n \to \infty} \E_{{\mu}}\left[ \left| \frac{{Y_{k}\left(n\right)}}{n} - v^{\mathrm{eff}}_{k} \right|^{p} \right] = 0. 
                    \end{align}}
                \end{enumerate}
        \end{claim}
        {The condition $\mu\left(\eta\left(0\right) = \eta\left(1\right) = 1\right) > 0$ guarantees the existence of $k$-solitons with positive probability, as we will see later in Section \ref{subsec:qstat}. If $\mu$ is a space-homogeneous two-sided Markov distribution with $\mu\left(\eta\left(0\right) = \eta\left(1\right) = 1\right) = 0$, then there are no solitons at all or only $1$-solitons, which is an obvious situation with no interaction between solitons and is not considered in this paper. We note that if $\mu$ is the Bernoulli product measure with the marginal density $0 < \rho < 1/2$, then the variance $D_{k}$ can be computed explicitly as a function of $\rho$, see Remark \ref{rem:diff_rho}. }
        {In Claim \ref{thm:intro_LT}, (\ref{item:claim1}) and (\ref{item:claim3}) are still valid if $Y_{k}(n)$ is replaced by $X_{k}(n)$, see Remark \ref{rem:nutomu}. For Claim \ref{thm:intro_LT}, (\ref{item:claim2}) $X_k(n)$ and $Y_k(n)$ may have different rate functions due to the deviations of $X_k(0)$, and in this paper we only focus on $Y_k(n)$.}
        
        Claim \ref{thm:intro_LT} will be stated precisely as Theorem \ref{thm:IPLDP}, and it will be proven {via Theorems \ref{thm:lln_lp}, \ref{thm:main} and \ref{thm:Markov}, which are limiting theorems for $Y_{k}\left(n\right)$ under the conditional probability measure $\nu := \mu\left( \ \cdot \ | \Omega_{0} \right)$, $\Omega_{0} \subset \Omega$, where 
        \begin{align}\label{def:omega_0}
            \Omega_{0} := \left\{ \eta \in \Omega \ ; \ \text{there is no soliton crossing the origin at time } 0\right\}.
        \end{align}
    By combining the limiting theorems under $\nu$, the exponential integrability of an excursion under $\mu$ and Proposition \ref{prop:main}, we will obtain the limiting theorems under $\mu$. We note that with additional assumptions on $\mu$ and $k$, the statement of Claim \ref{thm:intro_LT} holds under more general initial distributions, see Figure \ref{fig:table_limit}.
    
    We actually prove Theorems \ref{thm:lln_lp}, \ref{thm:main} and \ref{thm:Markov} under more general initial distributions $\mu$ conditioned on $\Omega_{0}$.
    These are given by invariant measures introduced in \cite{FG} with a further condition for probability of the existence of large solitons.
    }
   These invariant measures introduced in \cite{FG} are called ${\q}$-statistics, {see Section \ref{subsec:qstat} for the precise definition}. 

    \begin{remark}
            For the tagged ball (with a certain rule to identify the positions of distinguished balls), instead of the tagged soliton, the LLN, the CLT and the LDP are shown in \cite{CKST} under the Bernoulli product measures. However, in the case of {the} BBS, balls are local quantities, whereas solitons are quasi-local quantities, so a more sophisticated mathematical treatment is needed to show scaling limits for solitons.
        \end{remark}
    
    Furthermore, we will show that two $k$-solitons are strongly correlated in the diffusive space-time scaling even when they are far apart at the macroscopic level.  
        \begin{claim}[Strong correlations between $k$-solitons] \label{thm:intro_cor}
            Assume that {the initial distribution} $\nu$ is a ${\q}$-statistics {conditioned on $\Omega_{0}$} with a certain second moment condition. Then, even if two $k$-solitons are far apart after taking the space scaling, those fluctuations converge to the same Brownian motion obtained in Claim \ref{thm:intro_LT} (\ref{item:claim1}) in the diffusive scaling.
        \end{claim}
    In Section \ref{sec:results}, we will restate Claim \ref{thm:intro_cor} as Theorem \ref{thm:st_cor} with precise assumptions. 
    
    From the above results, it is expected that the macroscopic fluctuations of the density of $k$-solitons at diffusive time scale $t$ can be obtained by shifting the initial fluctuation field by $B_{k}(t)$ obtained in Claim \ref{thm:intro_LT} (\ref{item:claim1}). In other words, the macroscopic fluctuation {field $\mathcal{Y}_k(u,t)$} should follow the following stochastic partial differential equations for each $k \in \N$ : 
        \begin{align}\label{eq:rigid_de}
            d \mathcal{Y}_{k}(u,t) = \frac{D_{k}}{2} \Delta_{u} \mathcal{Y}_{k}(u,t) dt + \partial_{u}  \mathcal{Y}_{k}(u,t) dB_{k}(t),
        \end{align}
    It is noteworthy that the noise driving \eqref{eq:rigid_de} {does not depend on the spatial variables,} which is in contrast to typical diffusive fluctuations for chaotic systems {where an additive space-time white noise drives the macroscopic equation. We expect that \eqref{eq:rigid_de} is a universal equation in completely integrable many-body systems, and was recently derived rigorously for the first time from hard-rods dynamics \cite{FO}.}  We also note that diffusive corrections to Euler scale hydrodynamics for integrable many-body systems have been studied in physics literature \cite{DBD, DBD2, DDMP, Sp}. It would be an interesting problem to derive \eqref{eq:rigid_de} rigorously from the BBS, to prove that $\{B_{k}(t) ; t \ge 0, k \in \N\}$ is a centered Gaussian field, and to specify the correlation between $B_{k}(t)$ and $B_{\ell}(t)$ with $k \neq \ell$.

    Our approach is based on two different linearization methods for the BBS. One is called the seat number configuration, which is recently introduced by \cite{MSSS, S}, and it is a generalization of the slot configuration developed by \cite{FNRW}. We note that the slot configuration has played an important role in the study of the dynamical aspects of BBS \cite{CS, FG, FNRW}.  The other is the $k$-skip map, which is a generalization of the $10$-elimination introduced by \cite{MIT} to solve the initial value problem for the BBS with periodic boundary condition. In \cite{S}, the $k$-skip map is considered in terms of the seat number configuration, and the relation between ${\q}$-statistics and the $k$-skip map is studied. The results and computations in \cite{S} are essential to obtain Lemma \ref{lem:NM_kl} and a decomposition formula \eqref{eq:orthogonal} for the position of the tagged soliton, see Section \ref{subsec:kskip} for details. 
    By using Lemma \ref{lem:NM_kl}, \eqref{eq:orthogonal} and the property of ${\q}$-statistics, the CLT/LDP for the tagged soliton can be reduced to the CLT/LDP for $M(n)$, respectively, where $M(n)$ is the number of times that the tagged soliton interacts with {solitons larger than itself} until time $n$. {By the same idea, we can show the LLN for the tagged soliton in $\mathbb{L}^p$ by that for $M(n)$.}
    Furthermore, Lemma \ref{lem:NM_kl} and \eqref{eq:orthogonal} are also useful for showing the strong correlations between solitons of the same {size}. To the best of our knowledge, this is the first time that the $10$-elimination is applied to the dynamical problem of the randomized BBS. A version of the $10$-elimination was used in \cite{LLP}, but they considered static problems. We note that our proof strategy can be applied even if the initial distribution ${\mu}$ is not a ${\q}$-statistics as long as ${\mu}$ has some nice property, see Remark \ref{rem:othernu} for details. 

    The rest of the paper is organized as follows. In Section \ref{sec:BBS} we briefly recall the basics of the BBS on $\Z$ and introduce some terminologies used in this paper, then  we present our main result, Theorem \ref{thm:IPLDP}, where the initial distribution  ${\mu}$ is a Bernoulli product measure or two-sided Markov distribution. {To prove Theorem \ref{thm:IPLDP}, we need combinatorial tools.} In Section \ref{sec:seat}, we introduce {such tools and} some notations that are essential for the proof as well as for describing the results under more general invariant distributions. In Section \ref{sec:general}, we present some technical results and our second main result, Theorem \ref{thm:st_cor}, when the initial distribution is a more general $\q$-statistics. 
    In Section \ref{subsec:kskip}, we introduce the notion of the $k$-skip map, $k \in \N$, and we prepare some lemmas for the proofs of main results. In the subsequent sections, we give proofs of results in Section \ref{sec:general}. In particular, in Section \ref{sec:st_proof}, we give a proof of Theorem \ref{thm:st_cor}.
    Finally, in Section \ref{sec:final}, we show Theorem \ref{thm:IPLDP} as a straightforward consequence of results in Section \ref{sec:general}.
    

\section{Box-Ball System}\label{sec:BBS}

    \subsection{Dynamics of the Box-Ball system}

        First we recall the definition of the one-step time evolution $\eta \mapsto T\eta$ when the total number of $1$s in $\eta \in \Omega$ is finite, presented in Introduction. A site $x \in \Z$ is called a record if $\eta(x) = T\eta(x) = 0$. Clearly, $T$ can be considered as the flipping $1$s (resp. $0$s) to $0$s (resp. $1$s) except for records, i.e., we can write $T\eta$ as 
            \begin{align}\label{eq:char_dy}
                T\eta\left(x\right) = 
                \begin{dcases}
                    1 - \eta\left(x\right) \ & \ \text{if } x \text{ is not a record}, \\
                    0 \ & \ \text{if } x \text{ is a record}.
                \end{dcases}
            \end{align}
        We note that records can be characterized as follows : 
            \begin{align}\label{eq:char_rec}
                x \text{ is a record} \text{ if and only if } \max_{z \le x} \sum_{y = z}^{x} \left(2\eta\left(y\right) - 1\right) \le - 1.
            \end{align}

        Now we define the BBS on $\Omega$. 
        The one-time step evolution $T : \Omega \to \Omega$ can be also defined via the notion of records as follows. For $\eta \in \Omega$, we define a record in $\eta$ as a site $x \in \Z$ that satisfies \eqref{eq:char_rec}. We note that there are infinitely many records in $\eta$ because the asymptotic ball density as $x \to \pm\infty$ is strictly smaller than $1/2$. Then, we define $T : \Omega \to \Omega$ by \eqref{eq:char_dy}. 
        
        We number records in $\eta$ from left to right as follows. For any $\eta \in \Omega$, we define
            \begin{align}
                s_{\infty}\left(\eta,0\right) :=\max \left\{ x \le 0 \ ; \ \max_{z \le x} \sum_{y = z}^{x}  \left(2\eta\left(y\right) - 1\right) \le -1  \right\},
            \end{align}
        and then we recursively define $s_{\infty}\left(\eta,i\right)$ as 
            \begin{align}
                s_{\infty}\left(\eta,i\right) :=\min \left\{ x > s_{\infty}\left(\eta,i - 1\right) \ ; \ \max_{z \le x} \sum_{y = z}^{x}  \left(2\eta\left(y\right) - 1\right) \le -1 \right\}, \\
                s_{\infty}\left(\eta,-i\right) :=\max \left\{ x < s_{\infty}\left(\eta,-i + 1\right) \ ; \ \max_{z \le x} \sum_{y = z}^{x}  \left(2\eta\left(y\right) - 1\right) \le -1 \right\},
            \end{align}
        for any $i \in \N$. Notice that {$s_{\infty}(\eta,i) \in \Z$ for any $i \in \Z$} because $\eta \in \Omega$. 
 
        We note that the dynamics of the BBS on $\Omega$ can also be described via the carrier process $W\left(\eta,x\right) : \Z \to \Z_{\ge 0}$ recursively defined by \eqref{def:carrier} and $W\left(\eta,s_{\infty}\left(\eta,i\right)\right) := 0$ for any $i \in \Z$. Then, by using $W$, $T : \Omega \to \Omega$ is written as \eqref{def:T_carrier}.

\subsection{Solitons in the Box-Ball configuration}\label{subsec:soliton}

        In this subsection, we explain how we can identify solitons in $\Omega$. 
        
        For given $\eta \in \Omega$, we consider the following decomposition : 
            \begin{align}\label{def:excursion}
                \eta = \cup_{i \in \Z} \mathbf{e}^{(i)}, \quad 
                \mathbf{e}^{(i)} := \left( \eta(x) \ ; \  s_{\infty}\left({\eta,}i\right) \le x < s_{\infty}({\eta,}i+1) \right).
            \end{align}
        The sequence $\mathbf{e}^{(i)}$ is called the $i$-th excursion of $\eta$. {In an abuse of notation, we write $\mathbf{e}^{(i)} \setminus \{ s_{\infty}\left(i\right) \}$ the sequence of $1$s and $0$s obtained by eliminating the leftmost $0$ from $\mathbf{e}^{(i)}$, i.e., $\mathbf{e}^{(i)} \setminus \{ s_{\infty}\left(i\right) \} := \left( \eta(x) \ ; \  s_{\infty}\left({\eta,}i\right) < x < s_{\infty}({\eta,}i+1) \right)$.}
        Then, for each $\mathbf{e}^{(i)} \setminus \{ s_{\infty}\left(i\right) \}$, we can find solitons via the Takahashi-Satsuma algorithm as follows : 
            \begin{itemize}
                
                \item Select the leftmost run of consecutive $0$s or $1$s such that {the length of the subsequent run is at least as long as the length of it.}
                \item Let $k$ be the length of the selected run. Group the $k$ element of the selected run and the first $k$ elements of the subsequent run. The grouped $2k$ elements are identified as a soliton with {size} $k$, or $k$-soliton.
                \item Remove the identified $k$-soliton, and repeat the above procedure until all $1$s are removed. 
            \end{itemize}
        {By the above algorithm, a $k$-soliton is defined as a subset of $\Z$, and its cardinality is $2k$.}
        From the definition of records, if $\mathbf{e}^{(i)} \setminus \{ s_{\infty}\left(i\right) \}$ is not empty, then all 1s and 0s in $\mathbf{e}^{(i)} \setminus \{ s_{\infty}\left(i\right) \}$ are grouped and become components of solitons. 
        We note that from the TS-algorithm, we see that solitons of the same {size} do not overlap, and a larger soliton can contain a smaller soliton inside, but not vice versa.
        An example of applying the above algorithm to $\eta = \dots01100011100110010110000\dots $ is shown in Figure \ref{ex:TSforex}. In this example, only two excursions have solitons, and there are one $3$-soliton, three $2$-solitons and one $1$-soliton in total. 
        
            \begin{figure}[H]
                    \footnotesize
                    \setlength{\tabcolsep}{4pt}
                    \centering
                    \begin{equation}
                        \begin{matrix}
                        \dots 0 & 1 & 1 & 0 & 0 & 0 & 1 & 1 & 1 & 0 & 0 & 1 & 1 & 0 & 0 & 1 & 0 & 1 & 1 & 0 & 0 & 0 & 0 \dots
                        \\
                        \dots \vline0 & \xc{red}{1} & \xc{red}{1} & \xc{red}{0} & \xc{red}{0} & \vline 0 & 1 & 1 & 1 & \xc{red}{0} & \xc{red}{0} & \xc{red}{1} & \xc{red}{1} & 0 & 0 & 1 & 0 & 1 & 1 & 0 & 0 & 0 & \vline 0  \dots
                        \\
                        \dots \vline0 & \xc{red}{1} & \xc{red}{1} & \xc{red}{0} & \xc{red}{0} & \vline0 & 1 & 1 & 1 & \xc{red}{0} & \xc{red}{0} & \xc{red}{1} & \xc{red}{1} & 0 & 0 & \xc{blue}{1} & \xc{blue}{0} & 1 & 1 & 0 & 0 & 0 & \vline0  \dots
                        \\
                        \dots \vline0 & \xc{red}{1} & \xc{red}{1} & \xc{red}{0} & \xc{red}{0} & \vline0 & 1 & 1 & 1 & \xc{red}{0} & \xc{red}{0} & \xc{red}{1} & \xc{red}{1} & \xc{red}{0} & \xc{red}{0} & \xc{blue}{1} & \xc{blue}{0} & \xc{red}{1} & \xc{red}{1} & 0 & 0 & 0 & \vline0  \dots
                        \\
                        \dots  \vline0 & \xc{red}{1} & \xc{red}{1} & \xc{red}{0} & \xc{red}{0} & \vline0 & \xc{brown}{1} & \xc{brown}{1} & \xc{brown}{1} & \xc{red}{0} & \xc{red}{0} & \xc{red}{1} & \xc{red}{1} & \xc{red}{0} & \xc{red}{0} & \xc{blue}{1} & \xc{blue}{0} & \xc{red}{1} & \xc{red}{1} & \xc{brown}{0} & \xc{brown}{0} & \xc{brown}{0} & \vline0  \dots \\
                        \dots \vline0 & \color{red}{1} & \color{red}{1} & \color{red}{0} & \color{red}{0} & \vline0 & \color{brown}{1} & \color{brown}{1} & \color{brown}{1} & \color{red}{0} & \color{red}{0} & \color{red}{1} & \color{red}{1} & \color{red}{0} & \color{red}{0} & \color{blue}{1} & \color{blue}{0} & \color{red}{1} & \color{red}{1} & \color{brown}{0} & \color{brown}{0} & \color{brown}{0} & \vline0  \dots 
                        \end{matrix}
                    \end{equation}
                    \caption{Identifying solitons in $\eta$ by the TS Algorithm. $1$-soliton is colored by blue, $2$-solitons are colored by red, and $3$-soliton is colored by brown.}\label{ex:TSforex}
                \end{figure}
        It was discovered by \cite{TS} that total number of $k$-solitons is conserved in time for each $k \in \N$, i.e., for any $\eta \in \Omega$ with the condition $\sum_{x \in \Z} \eta(x) < \infty$, we have 
            \begin{align}
                \left| \left\{ k\text{-solitons in } T\eta \right\} \right| = \left| \left\{ k\text{-solitons in } \eta \right\} \right|.
            \end{align}
        Now we observe the behaviors of solitons in time evolution. 
        If there are only solitons of the same {size}, they move to the right by $k$ :
            \begin{align}
                \begin{matrix}
                    \eta & = & \dots & \color{brown}{1} & \color{brown}{1} & \color{brown}{1} & \color{brown}{0} & \color{brown}{0} & \color{brown}{0} & 0 & 0 & \color{brown}{1} & \color{brown}{1} & \color{brown}{1} & \color{brown}{0} & \color{brown}{0} & \color{brown}{0} & 0 & 0 & 0 & \dots \\
                    T\eta & = & \dots & 0 & 0 & 0 & \color{brown}{1} & \color{brown}{1} & \color{brown}{1} & \color{brown}{0} & \color{brown}{0} & \color{brown}{0} & 0 & 0 & \color{brown}{1} & \color{brown}{1} & \color{brown}{1} & \color{brown}{0} & \color{brown}{0} & \color{brown}{0} & \dots
                \end{matrix}
            \end{align}
        If there are two solitons of different {size}s and the larger soliton is to the left of the smaller soliton, an interaction will occur between them at some time. {During the interaction,} the solitons overlap each other and the shapes of the solitons are collapsed, but they return to their original shapes after the interaction is over. Furthermore, the larger soliton accelerates when overtaking the smaller soliton, while the smaller soliton stays where it is. For example, see Figure \ref{ex:phase_shift}. 
            \begin{figure}[H]
                \footnotesize
                \centering
                    \begin{equation}
                        \begin{matrix}
                            \eta & = & \dots  & \color{brown}{1} & \color{brown}{1} & \color{brown}{1} & \color{brown}{0} & \color{brown}{0} & \color{brown}{0} & 0 & \color{blue}{1} & \color{blue}{0} & 0 & 0 & 0 & 0 & 0 & 0 & 0 & 0 & 0 & \dots \\
                            T\eta & = & \dots  & 0 & 0 & 0 & \color{brown}{1} & \color{brown}{1} & \color{brown}{1} & \color{brown}{0} & \color{brown}{0} & \color{blue}{1} & \color{blue}{0} & \color{brown}{0} & 0 & 0 & 0 & 0 & 0 & 0 & 0 &  \dots \\
                            T^2\eta & = & \dots  & 0 & 0 & 0 & 0 & 0 & 0 & \color{brown}{1} & \color{brown}{1} & \color{blue}{0} & \color{blue}{1} & \color{brown}{1} & \color{brown}{0} & \color{brown}{0} & \color{brown}{0} & 0 & 0 & 0 & 0 & \dots \\
                            T^3\eta & = & \dots  & 0 & 0 & 0 & 0 & 0 & 0 & 0 & 0 & \color{blue}{1} & \color{blue}{0} & 0 & \color{brown}{1} & \color{brown}{1} & \color{brown}{1} & \color{brown}{0} & \color{brown}{0} & \color{brown}{0} & 0 & \dots
                        \end{matrix}
                \end{equation}
                \caption{The $3$-soliton accelerates from time $2$ to $3$. On the other hand, the $1$-soliton does not move from time $1$ to $3$.}
                \label{ex:phase_shift}
            \end{figure}

    In the rest of this subsection, we introduce some notions for later use. 

\subsubsection{Set of all $k$-soltions}
        
        We denote by $ \Gamma_{k}{\left(\eta\right)}$ the set of all $k$-solitons in $\eta$. For any $\gamma \in \Gamma_{k}{\left(\eta\right)}$, we define $X\left(\gamma\right) := (\inf \gamma) - 1$, and call $X(\gamma)$ the position of $\gamma$. {To obtain our key results, which are Lemma \ref{lem:NM_kl} and the decomposition formula \eqref{eq:orthogonal}, it is important to define $X(\gamma)$ as $ (\inf \gamma) - 1$ instead of $\inf \gamma$.} 
\subsubsection{Natural numbering for solitons}

        In this paper, since we focus on a single soliton and consider its scaling limit, it is necessary to label each soliton. 
        For the BBS on $\Z$, it is convenient to use a record as a reference site for the detailed analysis. In particular, we are interested in the case where the origin is a record, i.e., $s_{\infty}\left(\eta,0\right) = 0$. However, for later use, considering the case where $s_{\infty}\left(\eta,0\right) = 0$ is not $0$, we order solitons as follows. For each $k \in \N$, a $k$-soliton to the left of $s_{\infty}\left({\eta,}0\right)$ is the 0th soliton, and $k$-solitons are numbered in order from left to right from there. More precisely, for any $k \in \N$, we denote by $\gamma^{0}_{k}$ the $k$-soliton such that 
            \begin{align}
                X\left(\gamma^{0}_{k}\right) = \max_{\gamma \in \Gamma_{k}, X(\gamma) < s_{\infty}\left({\eta,}0\right)} X\left(\gamma\right). 
            \end{align}
        Then, we recursively define $\gamma^{i}_{k}$ as the $k$-soliton such that 
            \begin{align}
                X\left(\gamma^{i}_{k}\right) = \min_{\gamma \in \Gamma_{k}, X(\gamma) > X\left(\gamma^{i-1}_{k}\right)} X\left(\gamma\right)
            \end{align}
        for $i \ge 1$, and 
            \begin{align}
                X\left(\gamma^{i}_{k}\right) = \max_{\gamma \in \Gamma_{k}, X(\gamma) < X\left(\gamma^{i+1}_{k}\right)} X\left(\gamma\right)
            \end{align}
        for $i \le -1$. 
        We call $\gamma^{i}_{k}$ the $i$-th $k$-soliton.

        In this paper, the above numbering is called the natural numbering for solitons. 

\subsubsection{Position of a soliton at time $n$}

        We can track each soliton in time evolution. First, for any $\gamma \in \Gamma_{k}$ we define heads $H(\gamma)$ and tails $T\left(\gamma\right)$ as
            \begin{align}
            H\left(\gamma\right) &:= \left\{ x \in \gamma \ ; \ \eta\left(x\right) = 1  \right\} = \left\{ H_{1}\left(\gamma\right) <  \dots <  H_{k}\left(\gamma\right)  \right\}, \\
            T\left(\gamma\right) &:= \left\{ x \in \gamma \ ; \ \eta\left(x\right) = 0  \right\} = \left\{ T_{1}\left(\gamma\right) <  \dots <  T_{k}\left(\gamma\right)  \right\}. 
        \end{align}
        From \cite[Proposition 1.3]{FNRW}, for any $\gamma \in \Gamma_{k}\left(\eta\right)$, there exists unique $\gamma' \in \Gamma_{k}\left(T\eta\right)$ such that 
        $T\left(\gamma\right) = H\left(\gamma'\right)$, and we write $\gamma'$ as $\gamma(1)$, i.e., $X\left(\gamma(1)\right)$ is the position of $\gamma$ at time $1$. By repeating the above, for any $n \in \N$, we can find $\gamma\left(n\right) \in \Gamma_{k}\left(T^n \eta\right)$ such that $T\left(\gamma\left(n - 1\right)\right) = H\left(\gamma\left(n\right)\right)$, and we call $X\left(\gamma\left(n\right)\right)$ the position of $\gamma$ at time $n$. We note that since there may be a $k$-soliton passing through the origin in time evolution,  $\gamma^{i}_{k}\left(n\right)$ is not always the $i$-th $k$-soliton in $T^n \eta$. 

        In the following, the position at time $n$ of $i$-th $k$-soliton is denoted by $X^{i}_{k}({\eta,}n)$.

\subsection{Scaling limits for solitons}

    Now we state our main results on the fluctuations of $k$-solitons when the initial distribution $\mu$ is given by a {space-homogeneous} Bernoulli product measure or two-sided Markov distribution supported on $\Omega$ {and conditioned on $\Omega_{0}$, where $\Omega_{0} $ is defined at \eqref{def:omega_0}. This is defined more precisely as $\Omega_0=\{ \eta \in \Omega \ ;  \ 0 \ \text{is a record} \}$}. {Since we are interested in the increment of the position of a fixed $k$-soliton from time $0$ to $n$, for any $k \in \N$, $i \in \Z$ and $n \in \Z_{\ge 0}$ we define 
        \begin{align}\label{def:Y}
            Y^{{i}}_{k}\left(\eta,n\right) := X^{{i}}_{k}\left( \eta,n \right) - X^{{i}}_{k}\left(\eta,0\right).
        \end{align}}
    
    \begin{theorem}\label{thm:IPLDP}
        Assume that the initial distribution ${\mu}$ is a {space-homogeneous} Bernoulli product measure or two-sided  Markov distribution supported on $\Omega$ {and that $\mu\left(\eta\left(0\right) = \eta\left(1\right) = 1\right) > 0$}. {Let $\nu$ be the conditional probability measure such that ${\mu}$ is conditioned on $\Omega_{0}$.} Then, for any $k \in \N$, we have the following. 
                \begin{enumerate}
                
                    \item \label{item:IPLDP1} {For any $i \in \Z$ and  $\mathbf{T} > 0$, under $\mu$ or $\nu$}, the following step-interpolation process,
                    \begin{align}\label{def:step_X}
                        t \mapsto \frac{1}{n} {Y}^{i}_{k}\left(\left\lfloor n^2 t \right\rfloor  \right) - nt v^{\mathrm{eff}}_{k}\left( {\mu} \right),
                    \end{align}
                    converges weakly in $D[0,\mathbf{T}]$ to a Brownian motion with variance {$D_{k}$}, defined in \eqref{eq:dif_co} below. 

                    \item \label{item:IPLDP2} {For any $i \in \Z$, under $\mu$ and $\nu$,} the sequence $\left({Y}^{i}_{k}\left(n\right)/ n\right)_{n \in {\N}}$ satisfies the LDP with a smooth convex rate function defined in \eqref{def:rate_I} below. 

                    \item \label{item:IPLDP0} {For any {$i \in \Z$} and $p \ge 1$, we have 
                        \begin{align}
                            \lim_{n \to \infty} \E_{{\mu}}\left[ \left| \frac{{Y}^{i}_{k}\left(n\right)}{n} - v^{\mathrm{eff}}_{k}\left({\mu}\right) \right|^{p} \right] = \lim_{n \to \infty} \E_{\nu}\left[ \left| \frac{{Y}^{i}_{k}\left(n\right)}{n} - v^{\mathrm{eff}}_{k}\left({\mu}\right) \right|^{p} \right] =  0.
                        \end{align}}
                \end{enumerate}
    \end{theorem}
    {Theorem \ref{thm:IPLDP} will be proved via Theorems \ref{thm:lln_lp}, \ref{thm:main}, \ref{thm:Markov} and Proposition \ref{prop:main} described in Section \ref{sec:general}, and the proof of Theorem \ref{thm:IPLDP} will be given in Section \ref{sec:final}.  
    \begin{remark}\label{rem:nutomu}
        If $\mu$ is a space-homogeneous Bernoulli product measure or two-sided  Markov distribution supported on $\Omega$, then we can show that the $\mathbb{L}^{p}$ norm of $X^{i}_{k}\left(0\right)$ with respect to $\mu$ or $\nu$ is finite, see Appendix \ref{app:expbound_ex}. Hence, we can replace $Y^{i}_{k}(n)$ in the statement of Theorem \ref{thm:IPLDP} (\ref{item:IPLDP1}), (\ref{item:IPLDP0}) by $X^{i}_{k}(n)$.
    \end{remark}}

    From Theorem \ref{thm:IPLDP} (\ref{item:IPLDP1}), if we focus on a single soliton, it converges to a Brownian motion whose variance depends only on the initial distribution ${\mu}$ and the size of the soliton $k$, and not on the number $i$. In Section \ref{sec:results}, we focus on two solitons of the same size under more general initial conditions and prove their strong correlations in diffusive space time scaling, see Theorem \ref{thm:st_cor} and Corollary \ref{cor:2BM}.

\section{Seat number configuration}\label{sec:seat}

    In order to prove the results described in Section \ref{sec:BBS}, we need the linearization method of BBS. In this section, first we introduce the notion of a soliton with volume. This is a useful notion when looking at the correspondence between the linearization method and the position of solitons, which will be introduced later. 
    and then recall the definition of the seat number configuration. Next, we recall the linearization method called “seat number configuration” introduced by \cite{MSSS,S}.

\subsection{A soliton with volume}\label{subsec:volume}

    This subsection introduces the notion of volume for solitons. We then introduce a way to number solitons with volume. Finally, we explain how the increment of a soliton from time $0$ to $n$ is described. 
    
    We fix $\eta \in \Omega$ and recall that $X\left(\gamma\right)$ is the position of $\gamma$. We denote by $\bar{X}\left(\gamma\right) := \sup \gamma$ the rightmost site of $\gamma$. 
    For any $\gamma, \gamma' \in \Gamma_{k}$ with $X\left(\gamma\right) < X\left(\gamma'\right)$, we say that $\gamma$ and $\gamma'$ are connected if there are no $\ell$-solitons with $\ell \ge k + 1$ and records in $\left[\bar{X}\left(\gamma\right), X\left(\gamma'\right)\right]$. In equation form, $\gamma$ and $\gamma'$ are connected if the followings hold : 
        \begin{itemize}
                \item $\left[\bar{X}\left(\gamma\right), X\left(\gamma'\right)  \right] \cap \left\{ s_{\infty}\left(i\right) ; i \in \Z \right\} = \emptyset$,
                \item for any $\gamma'' \in \cup_{\ell \ge k + 1} \Gamma_{\ell}$, $\left[\bar{X}\left(\gamma\right), X\left(\gamma'\right)  \right] \cap  \gamma'' = \emptyset$.
        \end{itemize}
    Then, for any $\gamma \in \Gamma_{k}$, we define  
            \begin{align}
                Con\left(\gamma\right) := \left\{ \gamma' \in \Gamma_{k} \ ; \ \gamma \text{ and } \gamma' \text{ are connected} \right\}.
            \end{align}
    Note that $\gamma' \in Con\left(\gamma\right)$ then $Con\left(\gamma\right) = Con\left(\gamma'\right)$. For later use, we denote by $\Gamma_{k}^{*}$ the set of $k$-solitons such that 
            \begin{itemize}
                \item for any $\gamma, \gamma' \in \Gamma_{k}^{*}$, $Con\left(\gamma\right) \cap Con\left(\gamma'\right) = \emptyset$,
                \item for any $\gamma \in \Gamma_{k}^{*}$, $X\left(\gamma\right) \le X\left(\gamma''\right)$ for any $\gamma'' \in Con\left(\gamma\right)$.
            \end{itemize}
    In other words, the leftmost one among the connected ones is chosen as the representative and $\Gamma_{k}^{*} = \Gamma_{k}^{*}(\eta)$ is the set of such representatives. 
    Clearly, we have $\Gamma_{k} = \cup_{\gamma \in \Gamma_{k}^{*}}Con\left(\gamma\right)$. For any $\gamma \in \Gamma_{k}^{*}$, we say that the number of solitons in $Con\left(\gamma\right)$ is the volume of $\gamma$, and write $\operatorname{Vol}\left(\gamma\right) := \left|Con\left(\gamma\right) \right|$. Also, we say that for each $k \in \N$, an element $\gamma \in  \Gamma_{k}^{*}$ is a $k$-soliton with volume. For example, in the configuration used in Figure \ref{ex:TSforex}, there are three $2$-solitons colored by red. The volume of leftmost $2$-soliton is $1$, and that of the middle $2$-soliton is $2$.

\subsubsection{Truncated numbering of solitons with volume}

        We consider the truncated numbering for solitons with volume. 
        We denote by $\gamma^{\left( 0 \right)}_{k}$ the $k$-soliton with volume such that 
            \begin{align}
                X\left(\gamma^{\left( 0 \right)}_{k}\right) = \max_{\gamma \in \Gamma^{*}_{k}, X(\gamma) < s_{\infty}\left({\eta,}0\right)} X\left(\gamma\right). 
            \end{align}
        Then, we recursively define $\gamma^{{\left(i\right)}}_{k}$ as the $k$-soliton with volume such that 
            \begin{align}
                X\left(\gamma^{(i)}_{k}\right) = \min_{\gamma \in \Gamma^{*}_{k}, X(\gamma) > X\left(\gamma^{(i-1)}_{k}\right)} X\left(\gamma\right)
            \end{align}
        for $i \ge 1$, and 
            \begin{align}
                X\left(\gamma^{(i)}_{k}\right) = \max_{\gamma \in \Gamma^{*}_{k}, X(\gamma) < X\left(\gamma^{i+1}_{k}\right)} X\left(\gamma\right)
            \end{align}
        for $i \le -1$. 
        We call $\gamma^{(i)}_{k}$ the $i$-th $k$-soliton with volume. The difference from natural numbering is that the order is assigned only to the representatives in $\Gamma^{*}_{k}$. We note that $\gamma^{(1)}_{k} = \gamma^{1}_{k}$ from the rules of numbering. 

        It is an abuse of notation, but for any $k \in \N$, $i \in \Z$ and $n \in \Z_{\ge 0}$, we denote by $X^{(i)}_{k}\left(\eta,n\right)$ the position of the $i$-th $k$-soliton with volume at time $n$, i.e., $X^{(i)}_{k}\left(\eta,n\right) = X\left(\gamma^{(i)}_{k}({\eta,}n)\right)$. {Also, we will write $Y^{\left(i\right)}_{k}\left(\eta,n\right) := X^{(i)}_{k}\left(\eta,n\right) - X^{(i)}_{k}\left(\eta,0\right)$.}

\subsubsection{Interactions between solitons}

    Recall that the points $\left( s_{\infty}\left(\eta,i\right) \right)_{i \in \Z}$ separate groups of solitons. 
    For any $\gamma \in \Gamma_{k}, \gamma' \in \Gamma_{\ell}$ with $k < \ell$, we say that $\gamma$ and $\gamma'$ are interacting if $s_{\infty}\left(i\right) < X\left(\gamma'\right) < X\left(\gamma\right) < s_{\infty}\left(i + 1\right)$ for some $i \in \Z$. We say that $\gamma \in \Gamma_{k}$ is free if $\gamma$ does not interact with any $\ell$-soliton with $\ell > k$. 

    For any $\gamma \in \Gamma_{k}, \gamma' \in \Gamma_{\ell}$ with $k > \ell$, we say that $\gamma$ overtakes $\gamma'$ ( or $\gamma'$ is overtaken by $\gamma$) at time $n$ if $X\left(\gamma\left(n - 1\right)\right) < X\left(\gamma'\left(n - 1\right)\right)$ and $X\left(\gamma\left(n\right)\right) > X\left(\gamma'\left(n\right)\right)$. We denote by $N_{\ell}\left(\gamma,n\right)$ the number of $\ell$-solitons overtaken by $\gamma$ at time $n$. It is shown by \cite[Proposition 6.4]{FNRW} that for any $\gamma\in \Gamma_{k}$, the increment $X\left(\gamma(n)\right) - X\left(\gamma(n - 1)\right)$ can be represented as 
            \begin{align}
                &X\left(\gamma(n)\right) - X\left(\gamma(n - 1)\right) \\ 
                &= 
                \begin{dcases}
                    k + 2 \sum_{\ell = 1}^{k - 1} \ell N_{\ell}\left(\gamma,n\right) \ & \ \text{if $\gamma(n - 1)$ is free,} \\
                    0 \ & \ \text{otherwise.}
                \end{dcases}\label{eq:inc_sol}
            \end{align}
        In particular, $X\left(\gamma(n)\right) - X\left(\gamma(n - 1)\right) > 0$ if and only if $\gamma(n - 1)$ is free. For later use, {for any $k, \ell \in \N$, $i \in \Z$ and $n \in \Z_{\ge 0}$,} we define 
            \begin{align} 
                N^{{i}}_{k,\ell}\left(\eta,n\right) &:= 
                    \begin{dcases}
                        N_{\ell}\left(\gamma^{{i}}_{k},n\right) \ & \ k > \ell, \\ 
                        0 \ & \ \text{otherwise}, 
                    \end{dcases} \\
                M^{{i}}_{k}\left(\eta,n\right) &:= \left| \left\{ 1 \le m \le n \ ; \ \gamma^{{i}}_{k}(m) \text{ is not free} \right\} \right|, 
            \end{align}
        and 
            \begin{align}
                &M^{{i}}_{k,\ell}\left(\eta,n\right) \\ 
                &:= 
                    \begin{dcases}
                        \left| \left\{ \gamma \in \Gamma_{\ell} \ ; \ \gamma \text{ overtakes } \gamma^{{i}}_{k} \text{ at time } m, \ 1 \le m \le n \right\} \right| \ & \ k < \ell,  \\
                        0 \ & \ \text{otherwise.}
                    \end{dcases}
            \end{align}
        {If $\gamma^{i}_{k} = \gamma^{\left(j\right)}_{k}$ for some $j \in \Z$, then we write $N^{(j)}_{k,\ell}\left(\eta,n\right) := N^{i}_{k,\ell}\left(\eta,n\right)$, $M^{(j)}_{k}\left(\eta,n\right) := M^{i}_{k}\left(\eta,n\right)$, and $M^{(j)}_{k,\ell}\left(\eta,n\right) = M^{i}_{k,\ell}\left(\eta,n\right)$.}
        {From} \eqref{eq:inc_sol}, we have 
            \begin{align}
                &X^{{i}}_{k}\left(\eta,n\right) \\ 
                &= X^{{i}}_{k}\left(\eta,0\right) + k\left(n - M^{{i}}_{k}\left(\eta,n\right) \right) + 2 \sum_{m = 1}^{n} \sum_{\ell = 1}^{k - 1} \ell N^{{i}}_{k,\ell}\left(\eta,m\right). \label{eq:pos_n}
            \end{align}
        We now observe the interactions between solitons. When a soliton $\gamma$ is free and catches up with a smaller soliton $\gamma'$, $\gamma$ overtakes all of $Con(\gamma')$ simultaneously. In addition, when $\gamma$ catches up with $Con(\gamma')$, $Con(\gamma')$ are involved in the right half of $\gamma$, and if $\gamma(1)$ is free, then in the next step $Con(\gamma')$ are involved in the left half of $\gamma$. If $\gamma(1)$ is not free, then both $\gamma$ and $Con(\gamma)$ do not move. For example, see Figure \ref{ex:int_2}, \ref{ex:int_1} and \ref{ex:int_3}. 
            \begin{figure}[H]
                    \footnotesize
                    \centering
                    \begin{equation}
                        \begin{matrix}
                        \dots & \color{red}{1} & \color{red}{1} & \color{red}{0} & \color{red}{0} & 0 & \color{blue}{1} & \color{blue}{0} & \color{blue}{1} & \color{blue}{0} & \color{blue}{1} & \color{blue}{0} & 0 & 0 & 0 & 0 & 0 & 0 & 0 & 0 & 0 & \dots \\
                        \dots & 0 & 0 & \color{red}{1} & \color{red}{1} & \color{red}{0} & \color{red}{0} & \color{blue}{1} & \color{blue}{0} & \color{blue}{1} & \color{blue}{0} & \color{blue}{1} & \color{blue}{0} & 0 & 0 & 0 & 0 & 0 & 0 & 0 & 0 & \dots \\
                        \dots & 0 & 0 & 0 & 0 & \color{red}{1} & \color{red}{1} & \color{blue}{0} & \color{blue}{1} & \color{blue}{0} & \color{blue}{1} & \color{blue}{0} & \color{blue}{1} & \color{red}{0} & \color{red}{0} & 0 & 0 & 0 & 0 & 0 & 0 & \dots \\
                        \dots & 0 & 0 & 0 & 0 & 0 & 0 & \color{blue}{1} & \color{blue}{0} & \color{blue}{1} & \color{blue}{0} & \color{blue}{1} & \color{blue}{0} & \color{red}{1} & \color{red}{1} & \color{red}{0} & \color{red}{0} & 0 & 0 & 0 & 0 & \dots \\
                        \dots & 0 & 0 & 0 & 0 & 0 & 0 & 0 & \color{blue}{1} & \color{blue}{0} & \color{blue}{1} & \color{blue}{0} & \color{blue}{1} & \color{blue}{0} & 0 & \color{red}{1} & \color{red}{1} & \color{red}{0} & \color{red}{0} & 0 & 0 & \dots
                        \end{matrix}
                    \end{equation}
                    \caption{One $2$-soliton with volume $1$ and one $1$-soliton with volume $3$ are included in this figure. These solitons are interacting from the second line to fourth line. The $2$-soliton overtakes the group of $1$-solitons simultaneously.}\label{ex:int_2}
            \end{figure} 
        \begin{figure}[H]
                    \footnotesize
                    \centering
                    \begin{equation}
                        \begin{matrix}
                        \dots & \color{red}{1} & \color{red}{1} & \color{red}{0} & \color{red}{0} & \color{red}{1} & \color{red}{1} & \color{red}{0} & \color{red}{0} & 0 & \color{blue}{1} & \color{blue}{0} & 0 & 0 & 0 & 0 & 0 & 0 & 0 & 0 & 0 & \dots \\
                        \dots & 0 & 0 & \color{red}{1} & \color{red}{1} & \color{red}{0} & \color{red}{0} & \color{red}{1} & \color{red}{1} & \color{red}{0} & \color{red}{0} & \color{blue}{1} & \color{blue}{0} & 0 & 0 & 0 & 0 & 0 & 0 & 0 & 0 & \dots \\
                        \dots & 0 & 0 & 0 & 0 & \color{red}{1} & \color{red}{1} & \color{red}{0} & \color{red}{0} & \color{red}{1} & \color{red}{1} & \color{blue}{0} & \color{blue}{1} & \color{red}{0} & \color{red}{0} & 0 & 0 & 0 & 0 & 0 & 0 & \dots \\
                        \dots & 0 & 0 & 0 & 0 & 0 & 0 & \color{red}{1} & \color{red}{1} & \color{red}{0} & \color{red}{0} & \color{blue}{1} & \color{blue}{0} & \color{red}{1} & \color{red}{1} & \color{red}{0} & \color{red}{0} & 0 & 0 & 0 & 0 & \dots \\
                        \dots & 0 & 0 & 0 & 0 & 0 & 0 & 0 & 0 & \color{red}{1} & \color{red}{1} & \color{blue}{0} & \color{blue}{1} & \color{red}{0} & \color{red}{0} & \color{red}{1} & \color{red}{1} & \color{red}{0} & \color{red}{0} & 0 & 0 & \dots \\
                        \dots & 0 & 0 & 0 & 0 & 0 & 0 & 0 & 0 & 0 & 0 & \color{blue}{1} & \color{blue}{0} & \color{red}{1} & \color{red}{1} & \color{red}{0} & \color{red}{0} & \color{red}{1} & \color{red}{1} & \color{red}{0} & \color{red}{0} & \dots 
                        \end{matrix}
                    \end{equation}
                    \caption{One $2$-soliton with volume $2$ and one $1$-soliton with volume $1$ are included in this figure. These solitons are interacting from the second line to fifth line. Each $2$-solitons overtake the $1$-soliton step by step.}\label{ex:int_1}
            \end{figure}
            \begin{figure}[H]
                    \footnotesize
                    \centering
                    \begin{equation}
                        \begin{matrix}
                        \dots & \color{green}{1} & \color{green}{1} & \color{green}{1} & \color{green}{0} & \color{green}{0} & \color{green}{0} & \color{red}{1} & \color{red}{1} & \color{blue}{0} & \color{blue}{1} & \color{red}{0} & \color{red}{0} & 0 & 0 & 0 & 0 & 0 & 0 & 0 & 0 & 0 & \dots \\
                        \dots & 0 & 0 & 0 & \color{green}{1} & \color{green}{1} & \color{green}{1} & \color{red}{0} & \color{red}{0} & \color{blue}{1} & \color{blue}{0} & \color{red}{1} & \color{red}{1} & \color{green}{0} & \color{green}{0} & \color{green}{0} & 0 & 0 & 0 & 0 & 0 & 0 & \dots \\
                        \dots & 0 & 0 & 0 & 0 & 0 & 0 & \color{red}{1} & \color{red}{1} & \color{blue}{0} & \color{blue}{1} & \color{red}{0} & \color{red}{0} & \color{green}{1} & \color{green}{1} & \color{green}{1} & \color{green}{0} & \color{green}{0} & \color{green}{0} & 0 & 0 & 0 & \dots \\
                        \dots & 0 & 0 & 0 & 0 & 0 & 0 & 0 & 0 & \color{blue}{1} & \color{blue}{0} & \color{red}{1} & \color{red}{1} & \color{red}{0} & \color{red}{0} & 0 & \color{green}{1} & \color{green}{1} & \color{green}{1} & \color{green}{0} & \color{green}{0} & \color{green}{0} & \dots
                        \end{matrix}
                    \end{equation}
                    \caption{One $3$-soliton with volume $1$, one $2$-soliton with volume $1$ and one $1$-soliton with volume $1$ are included in this figure. The $2$-soliton overtakes the $1$-soliton after being overtaken by the $3$-soliton. }\label{ex:int_3}
            \end{figure}
        \noindent Hence, if the $i$-th $k$-soliton is free at time $n - 1$, then 
            \begin{align}
                &N^{{i}}_{k,\ell}(\eta,n) \\
                &= \left| \left\{ \gamma \in \Gamma_{\ell}\left(\eta\right) \ ; \ \gamma(n - 1) \subset \left[ H_{1}\left( \gamma^{{i}}_{k}(n - 1) \right), T_{1}\left( \gamma^{{i}}_{k}(n - 1) \right) \right] \right\} \right| \label{eq:overtake}.
            \end{align}
        {In addition, for any $k \in \N$, $n \in \Z_{\ge 0}$ and $i,j \in \Z$ such that $\gamma^{i}_k \in Con\left(\gamma^{(j)}_k\right)$, $M^{i}_{k}\left(\eta,n\right) = M^{(j)}_{k}\left(\eta,n\right)$. }
        On the other hand, we have the following inequality for $M^{i}_{k}(n)$ : 
            \begin{align}\label{ineq:overtaken}
                2\sum_{\ell \ge k + 1} M^{{i}}_{k, \ell}(\eta,n) \le M^{{i}}_{k}(\eta,n) \le 1 +  2\sum_{\ell \ge k + 1} M^{{i}}_{k, \ell}(\eta,n). 
            \end{align}
        We note that since the operator $T$ and spatial shift operators are commutative, the values of $N^{{i}}_{k,\ell}(n)$, $M^{{i}}_{k}(n)$, $M^{{i}}_{k,\ell}(n)$ are invariant under any spatial shift that does not change the numbering of solitons. For later use, we write this fact as a lemma. We define spatial shift operators $\tau_{y} : \Omega \to \Omega,  y \in \Z$ as 
            \begin{align}
                \tau_{y}\eta(x) := \eta(x + y),
            \end{align}
        for any $x \in \Z$. 
            \begin{lemma}\label{lem:sp_shift}
                Suppose that $\eta \in \Omega$. Then, for any $k,\ell \in \N$, $i \in \Z$ and $n \in \Z_{\ge 0}$, we have 
                    \begin{align}
                        X^{{i}}_{k}\left( \tau_{s_{\infty}(\eta,0)}\eta, n \right) = X^{{i}}_{k}\left( \eta, n \right) -  s_{\infty}(\eta,0),
                    \end{align}
                and
                    \begin{align}
                        N^{{i}}_{k,\ell}\left( \tau_{s_{\infty}(\eta,0)}\eta, n \right) &= N^{{i}}_{k,\ell}\left( \eta, n \right), \\
                        M^{{i}}_{k}\left( \tau_{s_{\infty}(\eta,0)}\eta, n \right) &= M^{{i}}_{k}\left( \eta, n \right), \\
                        M^{{i}}_{k,\ell}\left( \tau_{s_{\infty}(\eta,0)}\eta, n \right) &= M^{{i}}_{k,\ell}\left( \eta, n \right).
                    \end{align}
            \end{lemma}
    
    Thanks to Lemma \ref{lem:sp_shift}, we see that $Y^{{i}}_{k}$ is a function of $\tau_{s_{\infty}(\eta,0)}\eta$.

\subsection{Seat number configuration for the box-ball system}
    
        To derive the limiting behaviors of solitons, it is useful to consider seat number configuration space in which the dynamics of the BBS is linearized. In this section, we briefly recall the linearization method introduced by \cite{MSSS} and seat by \cite{S}. The main idea of this method is to assign a different parameter to each $0,1$ in $\eta \in \Omega$ based on the fact that $\eta$ contains many kinds of solitons. Then, we introduce a class of invariant measures for the BBS that are defined via the seat number configuration space. 

        Throughout this subsection, we fix an $\eta \in \Omega$ arbitrarily. 
        First, we introduce the notion of {\it carrier with seat numbers}. We consider a situation in which the {\it seats} of the carrier $W$ are indexed by $k \in \N$, i.e, $W$ is decomposed as 
            \begin{align}
                W\left(\eta,x\right) := \sum_{k \in \N} \mathcal{W}_{k}\left(\eta,x\right), \quad \mathcal{W}_{k}\left(\eta,x\right) \in \{0,1\},
            \end{align}
        where $\mathcal{W}_{k}\left(\eta,x\right) = 1$ means that the No.$k$ seat is occupied {when the carrier is at the site $x \in \mathbb{Z}$}. 
        Then, the refined update rule of such a carrier is given as follows:
            \begin{itemize}
                \item If there is a ball at site $x$, then the carrier picks up the ball and puts it at the empty seat with the smallest seat number;  
                \item If the site $x$ is empty, namely $\eta(x)=0$, and if there is at least one occupied seat, then the carrier puts down the ball at the occupied seat with the smallest seat number;
                \item Otherwise, the carrier just passes through. 
            \end{itemize}
        In other words, $\mathcal{W}_{k}, k \in \N$ are defined as $\mathcal{W}_{k}\left(\eta,s_{\infty}\left(\eta,i\right)\right) := 0$ for any $i \in \Z$ and 
            \begin{align}
                &\mathcal{W}_{k}\left(\eta,x\right) - \mathcal{W}_{k}\left(\eta,x-1\right) \\ 
                &:=
                \begin{dcases}
                    1 \ & \ \text{ if } \ \sum_{\ell = 1}^{k-1}\mathcal{W}_{\ell}\left(\eta,x-1\right) = 1,  \mathcal{W}_{k}\left(\eta,x-1\right) = 0 \text{ and } \eta(x) = 1, \\
                    -1 \ & \ \text{ if } \ \sum_{\ell = 1}^{k-1}\mathcal{W}_{\ell}\left(\eta,x-1\right) = 0,  \mathcal{W}_{k}\left(\eta,x-1\right) = 1 \text{ and } \eta(x) = 0, \\
                    0 \ & \ \text{otherwise},
                \end{dcases}
            \end{align}
        and we call $\mathcal{W}_{k}, k \in \N$ the carrier with seat numbers. 
            
        By using the carrier with seat numbers, we define the {\it seat number configuration} $\eta^{\sigma}_{k} \in \Omega$, $\sigma \in \left\{ \uparrow, \downarrow \right\}$, $k \in \N$ and $r \in \{0,1 \}^{\Z}$ as
            \begin{align}
                \eta^{\uparrow}_{k}\left(x\right) &:= 
                \begin{dcases}
                    1 \ & \ \mathcal{W}_{k}\left(\eta,x\right) - \mathcal{W}_{k}\left(\eta,x-1\right) = 1, \\
                    0 \ & \ \text{otherwise,}
                \end{dcases} \\
                \eta^{\downarrow}_{k}\left(x\right) &:= 
                \begin{dcases}
                    1 \ & \ \mathcal{W}_{k}\left(\eta,x\right) - \mathcal{W}_{k}\left(\eta,x-1\right) = -1, \\
                    0 \ & \ \text{otherwise,}
                \end{dcases}
            \end{align}
        and    
            \begin{align}
                r\left(\eta,x\right) := 
                \begin{dcases}
                    1 \ & \ x = s_{\infty}\left(\eta,i\right) \text{ for some } i \in \Z, \\
                    0 \ & \ \text{otherwise.}
                \end{dcases}
            \end{align}
        We note that by the seat number configuration, all $1, 0$ in $\eta$ are distinguished by the parameter $(k,\sigma)$ in the following sense : for any $x \in \Z$, 
            \begin{align}
                r(x) + \sum_{k \in \N} (\eta^{\uparrow}_{k}(x) + \eta^{\downarrow}_{k}(x)) = 1.
            \end{align}
        In the following, if a site $x \in \Z$ satisfies $\eta^{\sigma}_{k}\left(x\right) = 1$ for some $k \in \N$ and $\sigma \in \{\uparrow, \downarrow\}$, then we call $x$ a $\left(k,\sigma\right)$-seat.  
        \begin{remark}\label{rem:seat_soliton}
            We note that the seat number configuration can be described in terms of solitons as follows, see \cite{MSSS, S} for details. 
                \begin{align}
                    \eta^{\uparrow}_{k}\left(x\right) &= 
                    \begin{dcases}
                        1 \ & \ x = H_{k}\left(\gamma\right) \text{ for some } \gamma \in \bigcup_{\ell \ge k} \Gamma_{\ell}, \\
                        0 \ & \ \text{otherwise,}
                    \end{dcases} \\
                    \eta^{\downarrow}_{k}\left(x\right) &= 
                    \begin{dcases}
                        1 \ & \ x = T_{k}\left(\gamma\right) \text{ for some } \gamma \in \bigcup_{\ell \ge k} \Gamma_{\ell}, \\
                        0 \ & \ \text{otherwise.}
                    \end{dcases}
                \end{align}         
            In other words, a $k$-soliton consists of exactly one $\left(\ell,\sigma\right)$-seat for each $1 \le \ell \le k$ and $\sigma \in \{\uparrow, \downarrow\}$.
        \end{remark}
        Then, by using the above configurations, for each $k \in \N$, we define a non-decreasing function  $\xi_{k}\left(\ \cdot \ \right) : \Z \to \Z $ and its inverse {(for a certain sense)} $s_{k}\left(\ \cdot \ \right) : \Z \to \Z $ as
            \begin{align}
                \xi_{k}\left( \eta, x \right) - \xi_{k}\left( \eta, x - 1 \right) &:= r\left(\eta,x\right) + \sum_{\ell \in \N} \sum_{\sigma \in \{\uparrow, \downarrow\}} \eta^{\sigma}_{k + \ell}\left( x \right),  \\
                \xi_{k}\left(\eta, s_{\infty}\left(\eta,0\right) \right) &:= 0,
            \end{align}
        and
            \begin{align}\label{def:s_k}
                s_{k}\left(\eta, x \right) := \min \left\{ y \in \Z \ ; \ \xi_{k}\left(\eta,y \right) = x \right\}.
            \end{align}
        \begin{remark}
                The intuitive meaning of $\xi_{k}$ is that it is a function that counts the number of $1$s and $0$s from the reference point $s_{\infty}(0)$, ignoring solitons of {size} $k$ or less, and ignoring up to the k-th $1$s and $0$s constituting solitons, see Remark \ref{rem:seat_soliton}.  This counting method allows us to measure the effective distance between solitons, see Remark \ref{rem:eff_dis}. 
            \end{remark}
        Finally, for any $k \in \Z$, we define $\zeta_{k} : \Z \to \Z_{\ge 0}$ as 
            \begin{align}
                \zeta_{k}\left(\eta,i\right) := \sum_{y = s_{k}\left(\eta,i\right) + 1}^{s_{k}\left(\eta,i + 1\right)} \left( \eta^{\uparrow}_{k}\left(y\right) - \eta^{\uparrow}_{k+1}\left(y\right) \right).
            \end{align}
            
        We emphasize three important properties about $\zeta_{k}$ as follows. 
        The first is that the function $\zeta_{k}$ and $k$-soltions are related via the following formula,
            \begin{align}
                \zeta_{k}\left(\eta,i\right) &= \left| \left\{ \gamma \in \Gamma_{k}\left(\eta\right) \ ; \ \gamma \subset \left[ s_{k}\left(\eta,i\right), s_{k}\left(\eta,i+1\right) \right] \right\} \right|,
            \end{align}
        i.e., $\zeta_{k}$ represents the total number of $k$-solitons satisfying $\xi_{k}(\eta,X(\gamma)) = i$. In particular, our $\zeta$ coincides with the slot decomposition introduced in \cite{FNRW}, see \cite[Section 2.1, Proposition 2.3]{MSSS} and \cite[Section 4.1]{S} for details. From the same reasons as in the discussion just before Lemma \ref{lem:sp_shift}, we can see that for any $k \in \N$, $i \in \Z$, 
            \begin{align}\label{eq:zeta_shift}
                \zeta_{k}\left( \tau_{s_{\infty}(0)} \eta, i\right) = \zeta_{k}\left( \eta, i\right).
            \end{align}
        The second is the bijectivity {between $\zeta$ and $\eta$ satisfying $s_{\infty}(0)=0$, namely the configuration such that the origin is a record}. We define {the space of such configurations } $\Omega_{0} \subset \Omega$, {and also introduce}  $\bar{\Omega} \subset \Z_{\ge 0}^{\mathbb{N} \times \Z}$ as 
                    \begin{align}
                        \Omega_{0} &:= \Omega \cap \left\{ s_{\infty}\left(0\right) = 0 \right\} \\
                        \bar{\Omega} &:= \left\{ \zeta \in  \Z_{\ge 0}^{\mathbb{N} \times \Z} \ ; \ \sum_{k \in \N} \zeta_{k}\left(i\right) < \infty \ \text{for any } i \right\}.
                    \end{align}
        It is known that $\zeta : \Omega_{0} \to \bar{\Omega}$ is a bijection, {see} \cite[Section 3]{FNRW} for details. 
        We note that we can not reconstruct the original $\eta$ from $(\zeta_{k}(\eta, \ \cdot \ ))_{k \in \N}$ in general, because there is an arbitrariness in the choice of the position of $s_{\infty}(0)$ from \eqref{eq:zeta_shift}, see also Figure \ref{fig:diag_eta_zeta} for a summary of these properties, where for any $\eta \in \Omega$, we write 
            \begin{align}\label{def:palm_eta}
                \tilde{\eta} := \tau_{s_{\infty}(0)} \eta \in \Omega_{0}.
            \end{align} 
        \begin{figure}[H]
            \centering
            \begin{tikzcd}
                \eta \arrow[r] & \arrow[l] \left( \tilde{\eta},  s_{\infty}\left(0\right) \right) \arrow[r] & \tilde{\eta} \arrow[r, "\zeta"] & \arrow[l, "\zeta^{-1}"] \left(\zeta_{k}\left(i\right)\right)_{k \in \N, i \in \Z}
            \end{tikzcd}
            \caption{The relationship between $\eta$ and $\zeta$. The arrows $\leftrightarrow$ represent certain bijections, and the arrow $\rightarrow$ represents the first coordinate projection.}
            \label{fig:diag_eta_zeta}
        \end{figure}

        The third is that the dynamics of the BBS can be linearized via $\zeta_{k}$ with a certain offset \cite{FNRW, S}. Here, we cite the result by \cite{S} for later use. 
            \begin{theorem*}[Theorem 4.{5} in \cite{S}]
                Suppose that $\eta \in \Omega$. Then, for any $k \in \N$ and $i \in \Z$, we have
                    \begin{align}\label{eq:linear_seat}
                        \zeta_{k}\left(T\eta, i +  k + o_{k}\left(\eta\right) \right) = \zeta_{k}\left(\eta, i \right),
                    \end{align}
                where the offset $o_{k}\left(\eta\right)$ is given by 
                    \begin{align}
                        &o_{k}\left(\eta\right) \\
                        &:= s_{\infty}\left(\eta,0\right) - s_{\infty}\left( T\eta,0\right) + 2\sum_{y = s_{\infty}\left(0\right)+ 1}^{0} \sum_{\ell = 1}^{k}  \eta^{\downarrow}_{\ell}\left(y\right) - 2\sum_{y = Ts_{\infty}\left(0\right)+1}^{0}\sum_{\ell = 1}^{k}  T\eta^{\uparrow}_{\ell}\left(y\right).
                    \end{align}
            \end{theorem*}
        \begin{remark}
                We note that if $s_{\infty}\left(0\right) = 0$ and any soliton do not cross the origin $x = 0$ {in the evolution from $\eta$ to $T\eta$}, then $o_{k} = 0$ for any $k \in \N$. Hence, if $\eta \left(x\right) = 0$ for any $x \le 0$, then 
                    \begin{align}
                        \zeta_{k}\left(T^n\eta, i + nk\right) = \zeta_{k}\left(\eta,i\right),
                    \end{align}
                for any $k \in \N$, $n \in \N$ and $i \in \Z$. The above equation reproduces the linearization result for the BBS on $\left\{0,1 \right\}^{\N}$ shown in \cite{MSSS}. 
            \end{remark}
    Later in this paper we will use the following lemma, which can be considered as a version of \eqref{eq:linear_seat}. The proof will be given in Section \ref{app:lem_1}. 
        \begin{lemma}\label{lem:dif_xi}
            For any $k \ge \ell$, $i \in \mathbb{Z}$ and $n \in \N$, we have 
                \begin{align}
                    &\xi_{\ell}\left( T^{n}\eta, X^{(i)}_{k}\left(n\right) \right) - \xi_{\ell}\left( T^{n-1}\eta, X^{(i)}_{k}\left(n-1\right) \right) \\ 
                    &= 
                    \begin{dcases}
                        k + o_{\ell}\left(T^{n-1}\eta\right) + 2 \sum_{h = \ell + 1}^{k - 1}(h - \ell) N^{(i)}_{k,h}\left(\eta, n\right) \ & \  \text{if } \gamma^{(i)}_{k}(n-1) \text{ is free}, \\
                        \ell + o_{\ell}\left(T^{n-1}\eta\right) \ & \  \text{otherwise}.
                    \end{dcases}\label{eq:dif_xi}
                \end{align}
        \end{lemma}
        
        \begin{remark}\label{rem:eff_dis}
                For each $k \in \N$, we define the $k$-th effective distance between $\gamma, \gamma' \in \Gamma_{k}$ as 
                    \begin{align}
                        d_{\text{eff},k}\left(\eta,\gamma, \gamma'\right) := \left| \xi_{k}\left(\eta,X\left(\gamma\right)\right) - \xi_{k}\left(\eta,X\left(\gamma'\right)\right) \right|. 
                    \end{align} 
                Then, we see that 
                    \begin{align}
                        d_{\text{eff},k}\left(\eta,\gamma, \gamma'\right) = 0 \text{ if and only if } \gamma' \in Con\left(\gamma\right).
                    \end{align}
                In addition, from Lemma \ref{lem:dif_xi}, we see that the effective distance is conserved in time, i.e., for any $\gamma, \gamma' \in \Gamma_{k}$ and $n \in \Z_{\ge 0}$, we have 
                    \begin{align}
                        d_{\text{eff},k}\left(T^n\eta,\gamma\left(n\right), \gamma'\left(n\right) \right) &= \left| \xi_{k}\left(T^n\eta , X\left(\gamma\left(n\right)\right)\right) - \xi_{k}\left(T^n\eta,  X\left(\gamma'\left(n\right)\right)\right) \right| \\
                        &= d_{\text{eff},k}\left(\eta,\gamma, \gamma'\right),
                    \end{align}
                and thus we get
                    \begin{align}
                        \left|T^n Con\left( \gamma\left(n\right) \right)\right| = \left|Con\left( \gamma \right)\right|.
                    \end{align}
            \end{remark}
            
\section{General initial distributions}\label{sec:general}

    In this section we recall a class of invariant measures for the BBS, introduced by \cite{FG}. Then, we consider scaling limits for solitons starting from such invariant measures. 

\subsection{q-statistics}\label{subsec:qstat}

       We recall a class of translation-invariant stationary measures on $\left\{0,1 \right\}^{\Z}$ introduced by \cite{FG}. We define a set of infinite number of parameters as follows : 
            \begin{align}
                \mathcal{Q} := \left\{ \mathbf{q} = \left(q_k\right)_{k \in \N} \in [0,1)^{\N} \ ; \ \sum_{k \in \N} k q_{k} < \infty \right\}.
            \end{align}
        From \cite[Theorem 4.4, 4.5]{FG}, for given $\mathbf{q} \in  \mathcal{Q}$, there exists a translation-invariant stationary measure {$\mu_{\mathbf{q}}$} such that $\left(\zeta_{k}\left(i\right)\right)_{k \in \N, i \in \Z}$ are i.i.d. for each $k$ and independent over $k$ under {$\mu_{\mathbf{q}}\left( \ \cdot \ | \Omega_{0} \right) = \mu_{\mathbf{q}}\left( \ \cdot \ | s_{\infty}\left(0\right) = 0 \right)$}, and its distribution is characterized via $\zeta$ as 
            \begin{align}\label{def:nu_q}
               {\mu}_{\mathbf{q}}\left( \zeta_{k}\left(i\right) = m | s_{\infty}\left(0\right) = 0 \right) = q_{k}^{m}\left(1 - q_{k}\right), 
            \end{align}
        for any $k \in \N$, $i \in \Z$ and $m \ge 0$. {For notational simplicity, we will write 
            \begin{align}
                \nu_{\q}\left( \ \cdot \ \right) := \mu_{\mathbf{q}}\left( \ \cdot \ | s_{\infty}\left(0\right) = 0 \right).
            \end{align}
        The measure $\mu_{\q}$ can be reconstructed from $\nu_{\q}$ by the inverse Palm transform as follows : 
            \begin{align}\label{def:inv_Palm}
                \mathbb{E}_{\mu_{\q}}\left[f\right] = \frac{\mathbb{E}_{\nu_{\q}}\left[ \sum_{z = 0}^{s_{\infty}\left(1\right) - 1} \tau_{z}f  \right]}{\mathbb{E}_{\nu_{\q}}\left[ s_{\infty}\left(1\right) \right]},
            \end{align}
        for any local function $f : \{0,1\}^{\Z} \to \R$, where $\tau_{z} f\left(\eta\right) := f\left(\tau_{z} \eta\right)$, $z \in \Z$, see \cite[Section 5]{FNRW} or \cite[Section 4.2]{FG} for details.} 
        {We note that if $\q \in \mathcal{Q}$, then the mean size of excursion under $\nu_{\q}$ is finite, that is,  $\E_{\nu_{\q}}\left[\left| \mathbf{e}^{(i)} \right|\right] = \E_{\nu_{\q}}\left[s_{\infty}\left(i+1\right) - s_{\infty}\left(i\right)\right] < \infty$, see \cite[Section 3]{FG}.}
        In the following, we call $\mu_\mathbf{q}$ the ${\q}$-statistics. For later use, we recall {some properties of {${\q}$-statistics} in the following remark.
        \begin{remark}\label{lem:ex_iid}
            Recall that the $i$-th excursion $\mathbf{e}^{(i)}(\eta)$ in $\eta$ is defined in \eqref{def:excursion}. 
            These excursions are elements of the set $\mathcal{E}$ given by
            \begin{align}
                \mathcal{E} &:= \cup_{m \in \Z_{\ge 0}} \mathcal{E}(m), \\
                \mathcal{E}(m)
                &:= \left\{ \mathbf{e} \in \{ 0, 1 \}^{2m + 1} \ ; \ \sup_{1 \le y \le 2m+1}\sum_{x = 1}^{y} \left( 2\mathbf{e}(x) - 1 \right) \le -1, \ \sum_{x = 1}^{2m+1} \mathbf{e}(x)  = m  \right\}.
            \end{align}
            For each $i \in \Z$, $\mathbf{e}^{(i)}\left(\eta\right)$ can be considered as an $\mathcal{E}$-valued random variable under $\nu_{\mathbf{q}}$. Then, from the explicit construction of $\nu_{\mathbf{q}}$ in \cite[Section 4]{FG}, $(\mathbf{e}^{(i)})_{i \in \Z}$ are an i.i.d. sequence under the conditional probability measure {$\nu_{\q}$}. In particular, the centered configuration $\tilde{\eta} \in \Omega_0$ defined in \eqref{def:palm_eta} is record-shift invariant under $\nu_\q$, i.e., 
                \begin{align}
                    \nu_{\q}\left( {\eta} \in \cdot \right) = \nu_{\q}\left( \tau_{s_{\infty}\left(x\right)} {\eta} \in \cdot \right),
                \end{align}
            for any $x \in \Z$. 
        \end{remark}}
        
        {The Bernoulli product measures and stationary Markov distributions are two important classes of ${\q}$-statistics.} Let $\text{Ber}\left(\rho\right)$ be the Bernoulli product measure on $\left\{0,1 \right\}^{\Z}$ with intensity $0 < \rho < 1/2$. By choosing $\mathbf{q}$ as 
            \begin{align}\label{def:ber_q}
            q_1:=\rho(1-\rho), \quad 
                q_{k} := \frac{\rho^{k}\left(1 - \rho\right)^{k}}{\prod_{\ell = 1}^{k-1}\left(1 - q_{\ell}\right)^{2\left(k - \ell\right)}} \ \text{for} \ k \ge 2, 
            \end{align}
        from \cite[Theorem 3.1, Corollary 4.6]{FG}, we have ${\mu}_{\mathbf{q}} = \text{Ber}\left(\rho\right)$. {We denote the class of parameters $\mathbf{q}=\q(\rho)$ given by \eqref{def:ber_q} for $0 < \rho < \frac{1}{2}$ by $\mathcal{Q}_{\mathrm{Ber}} \subset \mathcal{Q}$. }
        
        Another important class of $\q$-statistic{s} is two-sided Markov distribution on $\left\{ 0, 1 \right\}^{\Z}$ with transition matrix $P = (p_{ij})_{i, j = 0, 1}$ on $\{0,1\}$ satisfying $ 0 < p_{01} + p_{11} < 1$. 
        In \cite{FG}, it is proved that such Markov distribution can be obtained by choosing $\mathbf{q}$ as
            \begin{align}\label{def:Mar_q}
                q_1:=a, \quad q_{k} := \frac{a b^{k-1}}{\prod_{\ell = 1}^{k-1}\left(1 - q_{\ell}\right)^{2\left(k - \ell\right)}} \ \text{for} \ k \ge 2, 
            \end{align}
        where 
            \begin{align}
                a := p_{01}p_{10}, \quad b := p_{00}p_{11}.
            \end{align}
        As shown in \cite[Section 5.3]{S}, the above map $P= (p_{ij})_{i, j = 0, 1} \to (a,b)$ induces a bijection between the set of transition matrix $\{P = (p_{ij})_{i, j = 0, 1} \ ; \ 0<p_{01}+p_{11}<1 \}$ and the set of the pair of parameters $\{(a,b) \ ; \ a >0, \ 0 \le b <1, \sqrt{a} + \sqrt{b} <1 \}$. We define 
            \begin{align}\label{def:q_Markov}
                \mathcal{Q}_{\mathrm{M}} := \left\{ \q \ ; \  {\mu}_{\q} \text{ is a two-sided Markov distribution} \right\}.
            \end{align}
        For each $\q \in \mathcal{Q}_{\mathrm{M}}$, we denote by $a(\q),b(\q)$ the pair of parameters giving $\q$ by \eqref{def:Mar_q}. Note that by taking $a=b=\rho(1-\rho)$ we have $\mathcal{Q}_{\mathrm{Ber}} \subset \mathcal{Q}_{\mathrm{M}}$, and by taking $b = 0$, we have $ \left\{ \q \in \mathcal{Q} \ ; \ q_{k} = 0 \ k \ge 2 \right\} \subset \mathcal{Q}_{\mathrm{M}}$. 
            
        In the following, we will introduce another class of $\q$-statistics. To do so, we define a shift operator $\theta : [0,1)^{\N} \to [0,1)^{\N}$ as $\theta \mathbf{q} = \left(q_{k + 1}\right)_{k \in \N}$ for any $\mathbf{q} = \left(q_{k}\right)_{k \in \N} \in [0,1)^{\N}$. We note that $\theta \mathcal{Q} = \mathcal{Q}$. Moreover, in \cite[Theorem 5.{16}]{S}, it is shown that for $\q \in \mathcal{Q}_{\mathrm{M}}$, $\theta \q \in \mathcal{Q}_{\mathrm{M}}$ with 
            \begin{align}\label{eq:theta_ab}
                a\left(\theta \q\right)=\frac{a\left(\q\right)b\left(\q\right)}{\left(1-a\left(\q\right)\right)^2}, \quad b\left(\theta \q\right)=\frac{b\left(\q\right)}{\left(1-a\left(\q\right)\right)^2}.
            \end{align}
        From this, we have $\theta\mathcal{Q}_{\mathrm{M}} \subset \mathcal{Q}_{\mathrm{M}}$, but $\theta\mathcal{Q}_{\mathrm{M}} \neq \mathcal{Q}_{\mathrm{M}}$. Actually, $\mathcal{Q}_{\mathrm{Ber}} \not\subset \theta\mathcal{Q}_{\mathrm{M}}$ since $a(\q)b(\q)<1$ for any $\q \in \mathcal{Q}_{\mathrm{M}}$. We also note that for any $\q \in \mathcal{Q}_{\mathrm{Ber}}$, $\theta \q \notin \mathcal{Q}_{\mathrm{Ber}}$. 
        
        
        We say that $\mathbf{q} \in \mathcal{Q}$ is asymptotically Markov if 
        there exists some $K \in \N$ such that $\theta^{K-1} \q \in \mathcal{Q}_{\mathrm{M}}$ with convention $\theta^0 \q =\q$. 
        We define 
            \begin{align}\label{def:q_AMarkov}
                \mathcal{Q}_{\mathrm{AM}} := \left\{ \q \in \mathcal{Q} \ ; \ \q \text{ is asymptotically Markov}  \right\}.
            \end{align}
        We note that $\theta \mathcal{Q}_{\mathrm{AM}} = \mathcal{Q}_{\mathrm{AM}}$ since for any $\q \in \mathcal{Q}_{\mathrm{AM}}$, $\theta\tilde{\q}=\q$ where $\tilde{q}_1=0, \tilde{q}_{k}=q_{k-1}$ for $k \ge 2$. For $\mathbf{q} \in \mathcal{Q}_{\mathrm{AM}}$, we define $K(\mathbf{q})$ as 
            \begin{align}\label{def:Kq}
                K\left( \mathbf{q} \right) := \min\left\{ \ell \in \N \ ; \ \theta^{\ell - 1}\q \in \mathcal{Q}_{\mathrm{M}} \right\}.
            \end{align}
        In particular, for $\q \in \mathcal{Q}_{\mathrm{AM}}$, $\q \in \mathcal{Q}_{\mathrm{M}}$ if and only if $K(\q)=1$. In summary, we have $\mathcal{Q}_{\mathrm{Ber}} \subsetneqq \mathcal{Q}_{\mathrm{M}} \subsetneqq \mathcal{Q}_{\mathrm{AM}} \subsetneqq \mathcal{Q}$ and $\theta \mathcal{Q}_{\mathrm{Ber}} \not\subset \mathcal{Q}_{\mathrm{Ber}}$, $\theta \mathcal{Q}_{\mathrm{M}} \subsetneqq \mathcal{Q}_{\mathrm{M}}$, $\theta \mathcal{Q}_{\mathrm{AM}} = \mathcal{Q}_{\mathrm{AM}}$ and $\theta \mathcal{Q} = \mathcal{Q}$. 

        {We will use the following conditions on the exponential integrability of $s_{\infty}\left(1\right)$ under $\nu_{\q}$. 
            \begin{lemma}\label{lem:expbound_ex_0}
            Suppose that $\q \in \mathcal{Q}$. If there exist some $k \in \N$ and $\lambda > 0$ such that $\E_{\nu_{\theta^{k}\q}} \left[ e^{\lambda {s_{\infty}\left(1\right)}} \right] < \infty$, then there exists some $\lambda' > 0$ such that $\E_{\nu_{\q}} \left[ e^{\lambda' {s_{\infty}\left(1\right)}} \right] < \infty$. 
        \end{lemma}
            \begin{lemma}\label{lem:expbound_ex}
            Suppose that $\q \in \mathcal{Q}_{{\mathrm{AM}}}$. Then, for sufficiently small $\lambda > 0$, we have $\E_{\nu_\q}\left[ e^{\lambda {s_{\infty}\left(1\right)}} \right] < \infty$.
        \end{lemma}
        The proofs of Lemmas \ref{lem:expbound_ex_0} and \ref{lem:expbound_ex} will be given in Section \ref{app:expbound_ex_0} and \ref{app:expbound_ex}.}

        For later use, we introduce some notations. For any $k \in \N$, we define $C_{k} : \mathcal{Q} \to \mathcal{Q}$ as 
            \begin{align}\label{def:cut_k}
                \left(C_{k}\q\right)_{\ell} := 
                \begin{dcases}
                    q_{\ell} \ & \ 1 \le \ell \le k, \\
                    0 \ & \ \ell \ge k + 1,
                \end{dcases}
            \end{align}
        for any $\q \in \mathcal{Q}$. We note that under $\nu_{C_{k}\q}$, there are no solitons larger than $k$ a.s. 
        Next, for any $k \in \N$ and $\q \in \mathcal{Q}$, we define $\a_{k}\left(\q\right)$, $\b_{k}\left(\q\right), \bar{r}_{k}\left(\q\right)$ as 
            \begin{align}
                \a_{k}\left(\q\right) &:= \E_{\nu_{\q}}\left[ \zeta_{k}\left(0\right) \right] = \frac{q_{k}}{1 - q_{k}}, \label{def:alpha} \\
                \b_{k}\left(\q\right) &:= \E_{\nu_{\q}}\left[ \left( \zeta_{k}\left(0\right) -  \a_{k}\left(\q\right) \right)^2 \right] = \frac{q_{k}}{\left(1 - q_{k}\right)^2}, \label{def:beta} \\
                \bar{r}_{k}\left(\q\right) &:= \E_{{\mu}_{\theta^{k}\q}}\left[r\left(0\right)\right] \label{def:r}.
            \end{align}
        We note that $\bar{r}_{k}\left(\q\right)$, $k \in \N$ satisfies the following system, 
            \begin{align}
                \frac{1}{\bar{r}_{k}\left(\q\right)} = 1 + 2\sum_{\ell = k+1}^{\infty} \frac{\left(\ell - k\right)\a_{\ell}\left(\q\right) }{\bar{r}_{\ell}\left(\q\right)} \label{eq:system_r},
            \end{align}
        see Section \ref{app:rec} for the derivation of \eqref{eq:system_r}. 

\subsection{Scaling limits for solitons under \texorpdfstring{$\q$}{q}-statistics}\label{sec:results} 

    In this subsection we will describe our main results on the fluctuations of $k$-solitons under the $\q$-statistics {conditioned on $\Omega_0$}. 
    
    First we recall that by \cite[Theorem 1.1, 1.5]{FNRW} and \cite[Theorem 4.5]{FG}, $Y^{{i}}_{k}\left(  \ \cdot  \  \right)$ satisfies the law of large numbers (LLN) in the hyperbolic scaling under ${\mu}_{\mathbf{q}}$. {Since $\mu_{\q}\left(s_{\infty}\left(0\right) = 0\right) > 0$, the same LLN holds under $\nu_{\mathbf{q}}$.} For later use we describe this fact as follows.
        \begin{theorem*}[Theorem 1.1, 1.5 in \cite{FNRW} + Theorem 4.5 in \cite{FG}]
            Suppose that $\mathbf{q} \in \mathcal{Q}$ and $q_{k} > 0$ for some $k \in \N$. Then, for any $i \in \Z$, we have 
                \begin{align}
                    \lim_{n \to \infty} \frac{1}{n} Y^{{i}}_{k}\left(\eta, n \right) &= v^{\mathrm{eff}}_{k}\left(\mathbf{q}\right) , \quad \mu_{\q} \text{ and } \nu_\q\text{-a.s.}
                    \label{eq:LLN_q}
                \end{align}
        \end{theorem*}
    The constant $v^{\mathrm{eff}}_{k}\left(\mathbf{q}\right)$, $k \in \N$ is called the effective velocity of $k$-solitons. 
        In this paper, we will show the $\mathbb{L}^{p}$ version of the above LLN for any $p \ge 1$.
        \begin{theorem}\label{thm:lln_lp}
            Suppose that $\mathbf{q} \in \mathcal{Q}$ and $q_{k} > 0$ for some $k \in \N$. Then for any $i \in \Z$ and $p \ge 1$, we have 
                \begin{align}
                            \lim_{n \to \infty} \E_{\nu_\q}\left[ \left| \frac{1}{n} Y^{{i}}_{k}\left( n \right) - v^{\mathrm{eff}}_{k}\left(\q\right) \right|^{p} \right] = 0.
                \end{align}
        \end{theorem}
    {We will show Theorem \ref{thm:lln_lp} in Section \ref{sec:lln_lp}.}
        \begin{remark}\label{rem:lln_lp}
            If $X^{{i}}_{k}(0)$ has the finite $p$-th moment, then one can show the $\mathbb{L}^{p}$ convergence for $X^{{i}}_{k}\left( n \right) / n$ instead of $Y^{{i}}_{k}\left( n \right) / n$. When {$s_{\infty}\left(1\right)$ has the exponential integrability under $\nu_{\q}$,} then $X^{{i}}_{k}(0)$ has the finite $p$-th moment for any $p \ge 1$, see Section \ref{app:Lp_X0}. 
        \end{remark}

    We will use the following relation between effective velocities. Recall that $\a_k, \bar{r}_k$ are defined in \eqref{def:alpha} and \eqref{def:r}.
        \begin{proposition}\label{prop:char_velo}
            Suppose that $\mathbf{q} \in \mathcal{Q}$ and $q_{k} > 0$ for some $k \in \N$. Then, we have
                \begin{align}\label{eq:v_eff}
                    v^{\mathrm{eff}}_{k}\left(\mathbf{q}\right) = k v^{\mathrm{eff}}_{1}\left(\theta^{k-1}\q\right) + 2 \sum_{\ell = 1}^{k - 1} \ell \a_{\ell}\left(\q \right) v^{\mathrm{eff}}_{k-\ell}\left(\theta^{\ell} \q \right), 
                \end{align}
            and 
                \begin{align}\label{eq:v_eff_1_r}
                    v^{\mathrm{eff}}_{1}\left(\theta^{k-1}\q\right) = \bar{r}_{k}\left(\q\right).
                \end{align}
        \end{proposition}
    The proof of Proposition \ref{prop:char_velo} will be given in Section \ref{app:prop_1}. 

    Our purpose in this paper is to consider the fluctuations of $Y^{{i}}_{k}\left(  \ \cdot  \  \right)$ corresponding to the law of large numbers mentioned above. {The following} result implies that the invariance principle(IP)/large deviations principle(LDP) for $Y^{{i}}_{k}\left(  \ \cdot  \  \right)$ can be reduced to the IP/LDP for $M^{{i}}_{k}\left(  \ \cdot  \  \right)$ {under $\nu_{\q}$}. 
        \begin{theorem}\label{thm:main}
        \mbox{}
            \begin{enumerate}
            
                \item \label{item:1} Suppose that there exist some $\q \in \mathcal{Q}$ {and} $k \in \N$ such that $q_{k} > 0$ and the following step-interpolation process, 
                    \begin{align}\label{def:step_M}
                        t \mapsto \frac{1}{n} M^{{\left(0\right)}}_{k}\left(\eta, \left\lfloor n^2 t \right\rfloor \right) - t  \mathbb{E}_{\nu_{\q}}\left[ M^{{\left(0\right)}}_{k}\left( n \right) \right],
                    \end{align} 
                converges weakly {in $D\left([0,\mathbf{T}]\right)$, $\mathbf{T} > 0$} to the centered Brownian motion with variance $G_{k}\left(\q\right)$ under $\nu_{\q}$. Then, {for any $i \in \Z$,} the following step-interpolation process 
                    \begin{align}\label{def:step_Y}
                        t \mapsto 
                        \frac{1}{n} Y^{{i}}_{k}\left(\eta, \left\lfloor n^2 t \right\rfloor \right) - nt v^{\mathrm{eff}}_{k}\left( \q \right),
                    \end{align}
                also converges weakly {in $D\left([0,\mathbf{T}]\right)$, $\mathbf{T} > 0$} to the centered Brownian motion with variance $D_{k}\left(\mathbf{q}\right)$ under $\nu_{\q}$, where $D_{k}\left(\mathbf{q}\right)$ is given by 
                    \begin{align}\label{eq:dif_co}
                        D_{k}\left(\mathbf{q}\right) := \frac{v^{\mathrm{eff}}_{k}\left(\q\right)^2 G_{k}\left(\mathbf{q} \right)}{v^{\mathrm{eff}}_{1}\left(\theta^{k-1}\q\right)^2}  + 4\sum_{\ell = 1}^{k-1} \frac{v^{\mathrm{eff}}_{\ell}\left(\q\right)^2 v^{\mathrm{eff}}_{k-\ell}\left(\theta^{\ell}\q\right) \b_{\ell}\left(\q\right)}{v^{\mathrm{eff}}_{1}\left(\theta^{\ell-1}\q\right)^2},
                    \end{align}
                {and $\b_{\ell}\left(\q\right)$ is defined in \eqref{def:beta}}.

                \item \label{item:2} Suppose that there exist some $\q \in \mathcal{Q}$, $k \in \N$ and $i \in \Z$ such that $q_{k} > 0$ and the following limit 
                    \begin{align}\label{def:Lambda_M}
                        \Lambda^{M {,i}}_{\q,k}\left(\lambda\right) &:= \lim_{n \to \infty} \frac{1}{n} \log\left(\E_{\nu_{\mathbf{q}}}\left[ \exp\left( \lambda \left(n - M^{{i}}_{k}\left( n \right) \right) \right) \right] \right) \in \R,
                    \end{align}
                exists for any $\lambda \in \R$, and $\Lambda^{M{,i}}_{\q,k}\left( \ \cdot \ \right)$ is essentially smooth in the sense of \cite[Definition 2.3.5]{DZ}. Then, the following limit 
                    \begin{align}
                        \Lambda^{Y{,i}}_{\q, k}\left(\lambda\right) &:= \lim_{n \to \infty} \frac{1}{n} \log\left(\E_{\nu_{\mathbf{q}}}\left[ \exp\left( \lambda  Y^{{i}}_{k}\left( n \right) \right)  \right] \right) \in \R \cup \{\infty\},
                    \end{align}
                exists for any $\lambda \in \R$, and $\Lambda^{Y{,i}}_{\q, k}\left(\lambda\right)$ satisfies \eqref{eq:Lambda_Y}.  In addition, we have   $\sup_{|\lambda| \le \delta} \left| \Lambda^{Y{,i}}_{\q,k}\left(\lambda\right) \right| < \infty$ for sufficiently small $\delta > 0$, and
                $\Lambda^{Y{,i}}_{\q,k}\left( \ \cdot \ \right)$ is also essentially smooth. Consequently, thanks to the G\"{a}rtner-Ellis theorem (cf. \cite[Theorem 2.3.6]{DZ}), {under $\nu_{\q}$,} the sequence {$\left(Y^{{i}}_{k}\left(n\right) / n\right)_{n \in \N}$} satisfies the LDP with the good rate function $I^{Y{,i}}_{\q,k}$, where 
                \begin{align}\label{def:rate_I}
                    I^{Y{,i}}_{\q,k}\left(u\right) := \sup_{\lambda \in \R} \left\{ \lambda u - \Lambda^{Y{,i}}_{\q,k}\left(\lambda\right) \right\}.
                \end{align}
            \end{enumerate}
        \end{theorem}

    {We will prove Theorem \ref{thm:main} in Section \ref{sec:main_proof}. 
    
    Theorems \ref{thm:lln_lp} and \ref{thm:main} are results under $\nu_\q$. For the IP, with the same assumption one can also show the same convergence under $\mu_{\q}$. For the $\mathbb{L}^p$ LLN and the LDP under $\mu_{\q}$, we need the exponential integrability of the size of an excursion as an additional assumption. We recall that $\mathbf{e}^{(i)}\left(\eta\right)$ is defined in \eqref{def:excursion}, and by definition, for any $i \in \Z$, we have the relation $\left| \mathbf{e}^{(i)}\left(\eta\right) \right| = s_{\infty}\left(i+1\right) - s_{\infty}\left(i\right)$.  In particular, if $\eta \in \Omega_{0}$, then $\left| \mathbf{e}^{(0)}\left(\eta\right) \right| = s_{\infty}\left(1\right) $.
        \begin{proposition}\label{prop:main}
        \mbox{}
            \begin{enumerate}
                \item \label{item:prop1} Assume that the assumptions of Theorem \ref{thm:main} (\ref{item:1}) holds for $\q \in \mathcal{Q}$, $k \in \N$ and $i \in \Z$. Then, under $\mu_{\q}$, the scaled process \eqref{def:step_Y} converges weakly in $D\left( \left[0,\mathbf{T}\right] \right)$, $\mathbf{T} > 0$ to the centered Brownian motion with variance $D_{k}\left(\q\right)$. 
                \item \label{item:prop2} Assume that the assumptions of Theorem \ref{thm:main} (\ref{item:2}) holds for $\q \in \mathcal{Q}$, $k \in \N$ and $i \in \Z$, and that there exists $\lambda > 0$ such that $\E_{\nu_{\q}}\left[ e^{\lambda s_{\infty}\left(1\right)} \right] < \infty$.
                Then, under $\mu_{\q}$, the sequence {$\left(Y^{{i}}_{k}\left(n\right) / n\right)_{n \in \N}$} satisfies the LDP with the good rate function $I^{Y{,i}}_{\q,k}$. 
                \item \label{item:prop3} Suppose that $\q \in \mathcal{Q}$ satisfies $\E_{\nu_{\q}}\left[ s_{\infty}\left(1\right)^{p} \right] < \infty$ with some $p > 1$. Then, for any $k \in \N$ with $q_{k} > 0$, $i \in \Z$ and $p \ge 1$, we have 
                    \begin{align}
                        \lim_{n \to \infty} \E_{\mu_\q}\left[ \left| \frac{1}{n} Y^{{i}}_{k}\left( n \right) - v^{\mathrm{eff}}_{k}\left(\q\right) \right|^{p} \right] = 0.
                    \end{align}
            \end{enumerate}
        \end{proposition}
        
    The proof of Proposition \ref{prop:main} will be presented in Section \ref{app:propmain}.}

    \begin{remark}\label{rem:eff_v}
        We note that $v^{\mathrm{eff}}_{k}\left(\q\right)$ can be given by 
            \begin{align}
                v^{\mathrm{eff}}_{k}\left(\q\right) &= \frac{d \Lambda^{Y{,i}}_{\q,k}\left(\lambda\right)}{d\lambda}|_{\lambda= 0} \\
                &= \bar{r}_{k}\left(\q\right) v^{\mathrm{eff}}_{k}\left( C_{k} \q \right). \label{eq:eff_lam}
            \end{align}
        In addition, \eqref{eq:eff_lam} gives the same formula for the effective velocity as the formula by \cite[(1.12)]{FNRW}, see Section \ref{app:eff_lam} for the proof of \eqref{eq:eff_lam} and the equivalence between the formulas. 
    \end{remark}
    \begin{remark}\label{rem:othernu}
        We note that Theorem \ref{thm:main} can be shown with initial distribution $\mu$ conditioned on $\Omega_0$, not necessarily $\q$-{statistics}, such that $\zeta_{k}\left(i\right)$, $k \in \N$, $i \in \Z$ are i.i.d. for each $k$ and independent over $k$ and satisfy an exponential moment condition{,} by the same argument in this paper. {We note that under the condition that $\left(\zeta_{k}\left(i\right)\right)_{k \in \N, i \in \Z}$ are i.i.d. for each $k$ and independent over $k$, the measure $\mu$ is stationary under the box-ball dynamics, which is proven in \cite{FNRW}.}
    \end{remark}
        
    In the next subsection we will give sufficient conditions for $\q,k$ such that the assumptions in Theorem \ref{thm:main} are satisfied.

    Next, we consider the correlations between two $k$-solitons. Our second result implies that even if two $k$-solitons are macroscopically far apart, they are strongly correlated in the diffusive space-time scaling. 
        \begin{theorem}\label{thm:st_cor} 
            {Suppose that $\q \in \mathcal{Q}$ and $\E_{\nu_\q}\left[ {s_{\infty}\left(1\right)}^2 \right] < \infty$. Then, for any $k \in \N$ with $q_k > 0$,}
            $u, v \in \R$ and $0 \le a \le 1$ we have 
                \begin{align}\label{eq:st_cor_trun}
                    &\varlimsup_{n \to \infty} \E_{\nu_{\mathbf{q}}}\left[ \left| \frac{1}{n} Y^{\left( \left\lfloor n^{a} u \right\rfloor \right)}_{k}\left(n^2 \right) - \frac{1}{n} Y^{\left( \left\lfloor n^{a} v \right\rfloor \right)}_{k}\left(n^2\right) \right|^{2} \right] = 0.
                \end{align}
        \end{theorem}
    {We will show Theorem \ref{thm:st_cor} in Section \ref{sec:st_proof}.}
    
    By combining Theorems \ref{thm:main} and \ref{thm:st_cor}, we have the following. 
        \begin{corollary}\label{cor:2BM}
            Suppose that $\q \in \mathcal{Q}$ and $k \in \N$ satisfy the assumption of Theorem \ref{thm:main} (\ref{item:1}) {and $\E_{\nu_\q}\left[ {s_{\infty}\left(1\right)}^2 \right] < \infty$}. Then, for any $u, v \in \R$ and $\mathbf{T} > 0$, we have the following weak convergence in $D\left([0,\mathbf{T}]\right)^2$ under $\nu_{\q}$.
                \begin{align}
                    &\lim_{\e \to 0}\left( \frac{1}{n} Y^{\left(\left\lfloor n u \right\rfloor\right)}_{k}\left( \left\lfloor n^2 t \right\rfloor \right) - nt v^{\mathrm{eff}}_{k}\left( \q \right), \frac{1}{n} Y^{\left(\left\lfloor n v \right\rfloor\right)}_{k}\left(\left\lfloor n^2 t \right\rfloor \right) - nt v^{\mathrm{eff}}_{k}\left( \q \right) \right) \\ &= \left(B_{k}\left(t\right), B_{k}\left(t\right)\right),
                \end{align}
            where $B_{k}\left(\ \cdot \ \right)$ is the centered Brownian motion with variance $D_{k}\left(\mathbf{q}\right)$.
        \end{corollary}
        
    Hence, under the assumption of Theorems \ref{thm:main} (\ref{item:1}) {and \ref{thm:st_cor}}, $k$-solitons with volume {starting at macroscopic distance} converge to the same Brownian motion.

\subsection{Scaling limits for \texorpdfstring{$M^{(i)}_{k}\left(  \ \cdot  \  \right)$}{M}}

    By Theorem \ref{thm:main}, we have found that for $\q \in \mathcal{Q}$ and $k \in \N$ such that IP/LDP for $M^{{i}}_{k}\left(  \ \cdot  \  \right)$ hold, IP/LDP for $k$-solitons also hold. In this subsection, we give some sufficient conditions of such $\q, k$. To describe the results, we define $\rho\left(\q\right)$ as the ball density under ${\mu}_{\q}$, i.e., 
        \begin{align}\label{def:ball_den}
            \rho\left(\q\right) := {\mu}_{\q}\left(\eta\left(0\right) = 1\right).
        \end{align}

    First we consider the case that $\q$ is asymptotically Markov. Recall that $\mathcal{Q}_{\mathrm{AM}}$, $K\left(\q\right)$ are defined in \eqref{def:q_AMarkov} and \eqref{def:Kq}. If $k$ is sufficiently large, we can show that $M^{{i}}_{k}\left(  \ \cdot  \  \right)$ satisfies the invariance principle, and the nice regularity property of \eqref{def:Lambda_M}. 
        \begin{theorem}\label{thm:Markov}
            If $\q \in \mathcal{Q}_{\mathrm{AM}}$ and $k \ge K\left(\q\right)$, then for any $i \in \Z$, \eqref{def:step_M} converges weakly to the Brownian motion with variance $G_{k}\left(\q\right)$ under {$\mu_{\q}$ and} $\nu_{\q}$, where $G_{k}\left(\q\right)$ is given by 
                \begin{align}\label{eq:dif_M_AM}
                    G_{k}\left(\q\right) &= 4\rho\left(\theta^{k}\q\right)\left(1 - \rho\left(\theta^{k}\q\right)\right)\left(1 - 2 \rho\left(\theta^{k}\q\right)\right).
                \end{align}
           For any $i \in \Z$ and $\lambda \in \R$, the limit $\Lambda^{M{,i}}_{\q,k}\left(\lambda\right)$ exists {and does not depend on $i$}. In addition, $\Lambda^{M}_{\q,k}\left(\lambda\right) {:= \Lambda^{M{,0}}_{\q,k}\left(\lambda\right)}$ is a smooth monotone convex function, which is explicitly given by 
                \begin{align}
                    \Lambda^{M}_{\q,k}\left(\lambda\right) &= \log\left( \frac{1 - 2\rho\left(\theta^{k}\q\right)}{2\left(1 - \rho\left(\theta^{k}\q\right)\right)} \left(e^{\lambda} + \sqrt{e^{2\lambda} - 1 + \frac{1}{\left(1 - 2\rho\left(\theta^{k}\q\right)\right)^2}}  \right)\right). \label{eq:log_gen}
                \end{align}
            In particular, the assumptions of Theorem \ref{thm:main} (\ref{item:1}) and (\ref{item:2}) are satisfied with $\q \in \mathcal{Q}_{\mathrm{AM}}$ and $k \ge K\left(\q\right)$. 
        \end{theorem}
    {The proof of Theorem \ref{thm:Markov} will be given in Section \ref{sec:Markov_proof}.}
    \begin{remark}
        Recall that if $\nu_{\q}$ is a Bernoulli product measure or two-sided Markov distribution, then $K(\q) = 1$. Hence, when the initial distribution is a Bernoulli product measure or two-sided Markov distribution supported on $\Omega$, then the statement of Theorem \ref{thm:Markov} holds for any $k \in \N$.  
    \end{remark}
    \begin{remark}
        If $\q \in \mathcal{Q}_{M}$, then the ball density $ \rho\left(\q\right)$ can be represented via $a\left(\q\right)$, $b\left(\q\right)$ as
            \begin{align}\label{eq:densitymarkov}
                \rho\left(\q\right) = \frac{1}{2}\left( 1 - \sqrt{1 - \frac{4a\left(\q\right)}{\left(1 + a\left(\q\right) - b\left(\q\right)\right)^2}} \right). 
            \end{align}
        We note that if $\q \in \mathcal{Q}_{M}$, then $\rho\left(\theta\q\right) < \rho\left(\q\right)$ holds. To show this, it suffices to show that  
            \begin{align}
                \frac{a\left(\theta\q\right)}{\left(1 + a\left(\theta\q\right) - b\left(\theta\q\right)\right)^2} < \frac{a\left(\q\right)}{\left(1 + a\left(\q\right) - b\left(\q\right)\right)^2},
            \end{align}
        and by using \eqref{eq:theta_ab}, we see that $0 < \sqrt{a\left(\q\right)} + \sqrt{b\left(\q\right)} < 1$ implies
            \begin{align}
                \frac{a\left(\theta\q\right)}{\left(1 + a\left(\theta\q\right) - b\left(\theta\q\right)\right)^2} {= \frac{a\left(\q\right)b\left(\q\right)}{\left(1 - a\left(\q\right) - b\left(\q\right)\right)^2}} 
                < \frac{a\left(\q\right)}{\left(1 + a\left(\q\right) - b\left(\q\right)\right)^2}.
            \end{align}
    \end{remark}
    \begin{remark}
        {We note that 
        \begin{align}
            \E_{{\mu}_\q}\left[ r\left(0\right) \right] &= {\mu}_{\q}\left(\eta\left(x\right) = T\eta\left(x\right) = 0\right) \\
            &= 1 -  {\mu}_{\q}\left(\eta\left(x\right) = 1\right) - \mu_{\q}\left(T\eta\left(x\right) = 1\right) \\
            &= 1 - 2\rho\left(\q\right) \label{eq:r_rho}.
        \end{align}}
        By \eqref{eq:v_eff_1_r} and \eqref{eq:r_rho}, under the assumption of Theorem \ref{thm:Markov}, $G_{k}\left(\q\right)$ and $\Lambda^{M}_{\q,k}\left(\lambda\right)$ can be represented as 
            \begin{align}
                G_{k}\left(\q\right) = v^{\mathrm{eff}}_{1}\left(\theta^{k-1} \q\right) \left(1 - v^{\mathrm{eff}}_{1}\left(\theta^{k-1} \q\right)^2 \right),
            \end{align}
        and 
            \begin{align}
                \Lambda^{M}_{\q,k}\left(\lambda\right) = \log\left( \frac{v^{\mathrm{eff}}_{1}\left(\theta^{k-1} \q\right)}{1 + v^{\mathrm{eff}}_{1}\left(\theta^{k-1} \q\right)} \left( e^{\lambda} + \sqrt{e^{2\lambda} + \frac{1 - v^{\mathrm{eff}}_{1}\left(\theta^{k-1} \q\right)^2 }{v^{\mathrm{eff}}_{1}\left(\theta^{k-1} \q\right)^2} } \right) \right).
            \end{align}
    \end{remark}
    {
    \begin{remark}\label{rem:diff_rho}
        If the initial distribution $\mu_{\q}$ is a space-homogeneous two-sided Markov distribution, then by \eqref{def:Mar_q}, \eqref{eq:theta_ab}, Proposition \ref{prop:char_velo}, \eqref{eq:densitymarkov} and \eqref{eq:r_rho}, one can compute $\rho\left(\theta^{k}\q\right)$ and $v^{\mathrm{eff}}_{k-\ell}\left(\theta^{\ell}\q\right)$, $0 \le \ell \le k - 1$ recursively in $k$ as functions of $a\left(\q\right), b\left(\q\right)$, and thus by \eqref{eq:dif_co} and \eqref{eq:dif_M_AM}, one can represent the diffusion coefficient $D_{k}\left(\q\right)$ as an explicit function of $a\left(\q\right), b\left(\q\right)$. For example, we compute $D_{1}\left(\q\right), D_{2}\left(\q\right)$ in the following if $\mu_{\q}$ is the Bernoulli product measure with marginal density $\rho = \rho\left(\q\right)$, and in this case, $D_{k}\left(\q\right)$ becomes a function of $\rho$. 
        For the case $k = 1$, we get 
            \begin{align}
                \rho\left(\theta \q\right) = \frac{\rho^2}{1 - 2\rho\left(1-\rho\right)},
            \end{align}
        and thus we obtain
            \begin{align}
                D_{1}\left(\q\right) = G_{1}\left(\q\right) = \frac{4\rho^2\left(1 - \rho\right)^2\left(1 - 2\rho\right)}{\left(1 - 2\rho\left(1-\rho\right)\right)^2}.
            \end{align}
        For $k = 2$, $\rho\left(\theta^2\q\right)$ and $v^{\mathrm{eff}}_{2}\left(\q\right)$ can be computed as 
            \begin{align}
                \rho\left(\theta^2\q\right) = \frac{\rho^3}{1 - 3\rho\left(1 - \rho\right)},
            \end{align}
        and 
            \begin{align}
                v^{\mathrm{eff}}_{2}\left(\q\right) = 2\left(1 + \a_{1}\left(\q\right)\right) v^{\mathrm{eff}}_{1}\left(\theta\q\right) = \frac{2\rho^3}{\left(1 - 3\rho\left(1 - \rho\right)\right)\left(1 - \rho\left(1 - \rho\right)\right)}.
            \end{align}
        By substituting the above to \eqref{eq:dif_co} and \eqref{eq:dif_M_AM}, we get  
            \begin{align}
                D_{2}\left(\q\right) = \frac{4\rho\left(1-\rho\right)\left(1-2\rho\right)}{\left(1 - 3\rho\left(1 - \rho\right)\right)\left(1 - \rho\left(1 - \rho\right)\right)}\left(\frac{\rho^2\left(1-\rho\right)^2}{\left(1 - 3\rho\left(1 - \rho\right)\right)^2} + 1\right).
            \end{align}
        By repeating the above calculations, $D_{k}\left(\q\right)$ can be computed.
    \end{remark}}
    
    Next, we consider the case where there are at most a finite number of nonzero elements in $\q$, i.e., there are at most a finite number of types of solitons under $\nu_\q$. If we denote by $q_\ell$ the largest nonzero element, then $M^{{i}}_{\ell} = 0$ $\nu_{\q}$-a.s., and thus $M^{{i}}_{\ell}$ trivially satisfies the assumptions in Theorem \ref{thm:main}. For the second largest solitons in $\q$, we can show the following.  
        \begin{theorem}\label{thm:finite}
            Suppose that $\q \in \mathcal{Q}$ satisfies $q_{\ell} > 0$, $q_{h} = 0$, $h \ge \ell + 1$ with some $\ell \ge 2$ and $|\q| \ge 2$. We denote by $k = k(\q)$ the second largest element in $\q$, i.e., 
                \begin{align}
                    k := \max\left\{1 \le h \le \ell -1 ; q_{h} > 0 \right\}.
                \end{align}
            Then, for any $i \in \Z$, \eqref{def:step_M} converges weakly to the Brownian motion with variance $G_{k}\left(\q\right)$ under $\nu_{\q}$, where $G_{k}\left(\q\right)$ is given by 
                \begin{align}
                    G_{k}\left(\q\right) &= \frac{4q_{\ell}}{\left(1 - q_{\ell}\right)^{2(\ell-k)}} \left(1 + \frac{4 q_\ell}{\left(1 - q_{\ell}\right)^{2(\ell-k)}}\right)^{-\frac{3}{2}}.
                \end{align}
            In addition, for any $i \in \Z$, the limit $\Lambda^{M{,i}}_{\q,k}\left(\lambda\right)$ exists {and does not depend on $i$}. In addition,  $\Lambda^{M}_{\q,k}\left(\lambda\right) {:= \Lambda^{M{,0}}_{\q,k}\left(\lambda\right)}$ is a smooth monotone convex function, which is explicitly given by 
                \begin{align}
                    \Lambda^{M}_{\q,k}\left(\lambda\right) &= \log\left(\frac{\left(1 - q_{\ell}\right)^{\ell-k}e^{\lambda} }{2} + \sqrt{\frac{\left(1 - q_{\ell}\right)^{2( \ell - k)}e^{2\lambda}}{4} + q_{\ell}}\right).
                \end{align}
            In particular, the assumptions of Theorem \ref{thm:main} (\ref{item:1}) and (\ref{item:2}) are satisfied with the above $\q$ and $k, \ell$. 
        \end{theorem}
    {We will prove Theorem \ref{thm:finite} in Section \ref{sec:finite_proof}.
    
    Finally, we summarize the limit theorems that can be proved for each initial distribution from Theorems \ref{thm:lln_lp}, \ref{thm:main} \ref{thm:st_cor}, \ref{thm:Markov} and \ref{thm:finite}. In Figure \ref{fig:table_limit}, we describe the condition on the soliton size $k$ such that the limit theorems hold, depending on the initial distribution. That is, if $k$ satisfies the condition under each initial distribution, then the limit theorems are obtained for any $k$-soliton. 
    In addition, the limit theorems hold under either uniform measure or conditioned measure on $\Omega_{0}$, thanks to Lemma \ref{lem:expbound_ex} and Proposition \ref{prop:main}.  Here, the abbreviation SC stands for strong correlations, and we say that SC holds if Theorem \ref{thm:st_cor} holds. }
    \begin{figure}[H]
        \begin{tabular}{|c|c|c|c|c|} 
            \hline Initial dist. \textbackslash \ results & LLN in $\mathbb{L}^p$ & IP & LDP & SC \\ 
            \hline Bernoulli or Markov & $k \in \N$ & $k \in \N$ & $k \in \N$ & $k \in \N$ \\ 
            \hline Asymptotically &  $q_{k} > 0$ & $k \ge K\left(\q\right)$, & $k \ge K\left(\q\right)$, & $q_{k} > 0$ \\
            Markov &   & $q_{k} > 0$ & $q_{k} > 0$ &  \\
            \hline $\max\{k \in \N ; q_{k} > 0\} < \infty$ & $q_{k} > 0$ & $1$st, $2$nd & $1$st, $2$nd & $q_{k} > 0$ \\ 
            \hline 
        \end{tabular}
        \centering
        \caption{Table of conditions on $k$ such that the limit theorem holds, depending on the initial distribution. Here,``1st" (resp. ``2nd") represents the largest $k$ (resp. second largest $k$) such that $q_{k} > 0$.}
        \label{fig:table_limit}
    \end{figure}

\section{\texorpdfstring{$k$}{k}-skip map for the BBS}\label{subsec:kskip}

    In this section we introduce the notion of {\it $k$-skip map}. The $k$-skip map is a natural generalization of the $10$-elimination introduced by \cite{MIT} in terms of the seat number configuration, and the results presented in this section are crucial for the proofs of main results. For the proofs of some known results on the $k$-skip map, we may refer to \cite{S}. 
    
    For any $k \in \N$, we define the $k$-skip map $\Psi_{k} : \Omega \to \Omega$ as
            \begin{align}
                \Psi_{k}\left( \eta \right)(x) := \eta\left(s_{k}\left(\eta, x + \xi_{k}\left(\eta, 0\right) \right) \right).
            \end{align}
    First we explain the intuitive meaning of the $k$-skip map when $k = 1$. Since $s_{1}(\eta, \ \cdot \ )$ is the inverse function of $\xi_{1}(\eta, \ \cdot \ )$, the subset $\{s_{1}(\eta,x) \ ; \ x \in \Z \} \subset \Z$ does not include the non-increasing points of $\xi_{1}(\eta, \ \cdot \ )$, i.e., 
        \begin{align}
            \left\{s_{1}\left(\eta,x\right) \ ; \ x \in \Z \right\} = \Z \setminus \left\{ x \in \Z \ ; \ \eta^{\uparrow}_{1}\left(x\right) + \eta^{\downarrow}_{1}\left(x\right) = 1 \right\}.
        \end{align}
    When $0 \in \{s_{1}(x) \ ; \ x \in \Z \}$, then $\Psi_{1}(\eta)$ is obtained by removing all $1,0$ with parameter $(1, \sigma)$, $\sigma \in \{\uparrow, \downarrow\}$ from $\eta$, and numbering the remaining $1$ and $0$ from left to right with respect to the origin $\eta(0)$. For the case $0 \nin \{s_{1}(\eta,x) \ ; \ x \in \Z \}$, we first translate $\eta$ by 
        \begin{align}
            \inf\left\{ s_{1}(\eta,x) \ ; \ s_{1}(\eta,x) < 0 \right\} = s_{1}\left(\eta,\xi_{1}\left(\eta,0\right)\right),
        \end{align}
    so that $0 \in \{s_{1}( \tau_{s_{1}(\eta,\xi_{1}(\eta,0))} \eta, x) \ ; \ x \in \Z \}$, and then we perform the same operation for $\tau_{s_{1}(\eta,\xi_{1}(\eta,0))} \eta$. 

    \begin{figure}[H]
                \footnotesize
                \setlength{\tabcolsep}{4pt}
                \begin{center}
                \renewcommand{\arraystretch}{2}
                \begin{tabular}{rccccccccccccccccccccc}
                    $x$ & -5 & -4 & -3 &  -2 &  -1 &  0 &  1 &  2 & 3 &  4 & 5 & 6 & 7 & 8 & 9 & 10 & 11 & 12 & 13 & 14 & 15  \\
                 \hline\hline 
                    $\eta(x)$ & \dots & 0 & 1 & 1 & 0 & 0 & 1 & 1 & 1 & 0 & 1 & 0 & 1 & 1 & 0 & 0 & 0 & 1 & 0 & 0 & \dots \\
                    \hline 
                    $\eta^{\uparrow}_{1}(x)$ & \dots & 0 & 1 & 0 & 0 & 0 & 1 & 0 & 0 & 0 & 1 & 0 & 1 & 0 & 0 & 0 & 0 & 1 & 0 & 0 & \dots \\
                    \hline 
                    $\eta^{\downarrow}_{1}(x)$ & \dots & 0 & 0 & 0 & 1 & 0 & 0 & 0 & 0 & 1 & 0 & 0 & 0 & 0 & 1 & 0 & 0 & 0 & 1 & 0 & \dots \\
                    \hline 
                    & \dots & 0 & \xcancel{1} & 1 & \xcancel{0} & 0 & \xcancel{1} & 1 & 1 & \xcancel{0} & \xcancel{1} & \xcancel{0} & \xcancel{1} & 1 & \xcancel{0} & 0 & 0 & \xcancel{1} & \xcancel{0} & 0 & \dots 
                    \\
                    $x$ & -3 & -2 &  &  -1 &   &  0 &   &  1 & 2 &   &  &  &  & 3 &  & 4 & 5 &  &  & 6 & 7 \\
                    \hline \hline 
                    $\Psi_{1}(\eta)$ & \dots & 0 &  & 1 &  & 0 &  & 1 & 1 &  &  &  &  & 1 &  & 0 & 0 &  &  & 0 & \dots 
                   \end{tabular}
                \end{center}
                \caption{How $\Psi_{1}(\eta)$ can be obtained from $\eta$ for the case $0 \in \{s_{1}(\eta,x) \ ; \ x \in \Z \}$, where $\dots$ represents the consequtive $0$s with infinite length. }
            \end{figure}
    \begin{figure}[H]
                \footnotesize
                \setlength{\tabcolsep}{4pt}
                \begin{center}
                \renewcommand{\arraystretch}{2}
                \begin{tabular}{rccccccccccccccccccccc}
                    $x$ & -6 & -5 & -4 & -3 &  -2 &  -1 &  0 &  1 &  2 & 3 &  4 & 5 & 6 & 7 & 8 & 9 & 10 & 11 & 12 & 13 & 14  \\
                 \hline\hline 
                    $\eta(x)$ & \dots & 0 & 1 & 1 & 0 & 0 & 1 & 1 & 1 & 0 & 1 & 0 & 1 & 1 & 0 & 0 & 0 & 1 & 0 & 0 & \dots \\
                    \hline 
                    $\tau_{- 1}\eta(x)$ & \dots & 0 & 0 & 1 & 1 & 0 & 0 & 1 & 1 & 1 & 0 & 1 & 0 & 1 & 1 & 0 & 0 & 0 & 1 & 0 & \dots \\
                    \hline 
                    $\tau_{- 1}\eta^{\uparrow}_{1}(x)$ & \dots & 0 & 0 & 1 & 0 & 0 & 0 & 1 & 0 & 0 & 0 & 1 & 0 & 1 & 0 & 0 & 0 & 0 & 1 & 0 & \dots \\
                    \hline 
                    $\tau_{- 1}\eta^{\downarrow}_{1}(x)$ & \dots & 0 & 0 & 0 & 0 & 1 & 0 & 0 & 0 & 0 & 1 & 0 & 0 & 0 & 0 & 1 & 0 & 0 & 0 & 1 & \dots \\
                    \hline 
                    & \dots & 0 & 0 & \xcancel{1} & 1 & \xcancel{0} & 0 & \xcancel{1} & 1 & 1 & \xcancel{0} & \xcancel{1} & \xcancel{0} & \xcancel{1} & 1 & \xcancel{0} & 0 & 0 & \xcancel{1} & \xcancel{0} & \dots  
                    \\ 
                    $x$ & -4 & -3 & -2 &  &  -1 &   &  0 &   &  1 & 2 &   &  &  &  & 3 &  & 4 & 5 &  &  & 6 \\
                    \hline \hline 
                    $\Psi_{1}(\eta)$ & \dots & 0 & 0 &  & 1 &  & 0 &  & 1 & 1 &  &  &  &  & 1 &  & 0 & 0 &  &  & \dots 
                   \end{tabular}
                \end{center}
                \caption{How $\Psi_{1}(\eta)$ can be obtained from $\eta$ for the case $0 \nin \{s_{1}(\eta,x) \ ; \ x \in \Z \}$. }
            \end{figure}
    The above observations can be made for any $k \in \N$ as well. 
    Now, we cite some results by \cite{S}. The following means that there is a one-to-one correspondence between cites in $\eta$ with parameter $(k + \ell, \sigma)$ and cites in $\Psi_{k}(\eta)$ with parameter $(\ell, \sigma)$, for any $k, \ell \in \N$ and $\sigma \in \{ \uparrow, \downarrow \}$. This property implies that $\Psi_{k}$ has the semi-group property and that $\Psi_{k}$ is a shift operator for $\zeta_{ \ \cdot  \ }$, see \cite{S} for the details and proofs. 
            \begin{proposition*}[Proposition 4.{11} in \cite{S}]
                Suppose that $\eta \in \Omega$. Then, for any $k, \ell \in \N$, $\sigma \in \{\uparrow, \downarrow\}$ and $x \in \Z$, we have 
                    \begin{align}\label{eq:seat_psi}
                        \Psi_{k}\left( \eta \right)^{\sigma}_{\ell}(x) = \eta^{\sigma}_{ k + \ell}\left(s_{k}\left(\eta, x + \xi_{k}\left(\eta, 0\right) \right) \right).
                    \end{align}
                In addition, we have 
                    \begin{align}\label{eq:semig_psi}
                        \Psi_{k }\left( \Psi_{\ell}\left(\eta\right) \right)\left( \ \cdot \ \right) =  \Psi_{k + \ell }\left( \eta\right)\left( \ \cdot \ \right),
                    \end{align}
                and 
                    \begin{align}
                        \zeta_{k}\left(\Psi_{\ell}\left( \eta \right), \ \cdot \ \right) = \zeta_{k+\ell}\left(\eta, \ \cdot \ \right) \label{eq:semig_zeta}.
                    \end{align}
            \end{proposition*}
    In \cite{S}, the following result has been proven. 
            \begin{theorem*}[Corollary 5.13 in \cite{S}]
                {Suppose that $\mathbf{q} \in \mathcal{Q}$.
                Then, for any $k \in \N$ and local function $f : \{0,1\}^{\Z} \to \R$, we have
                    \begin{align}\label{eq:shift_q}
                        \int_{\Omega} d\nu_{\mathbf{q}}\left(\eta\right) f\left( \Psi_{k}\left(\eta\right)\right) = \int_{\Omega} d\nu_{\theta^{k}\mathbf{q}}\left(\eta\right) f\left( \eta\right).
                    \end{align}}
            \end{theorem*}
        \begin{remark}
            Thanks to \eqref{eq:shift_q}, we have 
            \begin{align}
                \a_{k}\left(\theta^{\ell}\q\right) &= \a_{k+\ell}\left(\q\right), \quad
                \b_{k}\left(\theta^{\ell}\q\right) = \b_{k+\ell}\left(\q\right),
            \end{align}
        for any $k \in \N$, $\ell \in \Z_{\ge 0}$. Calculations similar to the above appear frequently in this paper. 
        \end{remark}

    From now on we prepare some lemmas for the proofs of main results. First we check the relation between $\Psi_{k}$ and $\tau_{s_{\infty}(0)}$. 
        \begin{lemma}\label{lem:palm_psi}
            Assume that $\eta \in \Omega$. Then, for any $k \in \N$ and $x \in \Z$, we have 
                \begin{align}
                    \Psi_{k}\left( \tilde{\eta} \right)\left(x\right) = \widetilde{\Psi_{k}\left(\eta\right)}\left(x\right).
                \end{align}
        \end{lemma}
        \begin{proof}[Proof of Lemma \ref{lem:palm_psi}]
            First we observe that for any $k \in \N$ and $x \in \Z$,
                \begin{align}
                    s_{k}\left( \tilde{\eta}, x \right) + s_{\infty}\left(0\right) = s_{k}\left( \eta, x \right).
                \end{align}    
            Next, from \cite[(4.1{7})]{S}, for any $k \in \N$, {$\ell \in \N \cup \{\infty\}$ and $x \in \Z$, we have 
                \begin{align}\label{eq:S_417}
                    s_{\ell}\left( \Psi_{k}\left(\eta\right), x \right) =\xi_{k}\left(\eta, s_{k+\ell}\left(\eta,x\right)\right) -\xi_{k}\left(\eta, 0\right).
                \end{align}
            In particular, we get}
            \begin{align}\label{eq:s_psi_xi}
                    s_{\infty}\left( \Psi_{k}\left(\eta\right), 0 \right) = -\xi_{k}\left(\eta, 0\right).
                \end{align}
            By using the above, we obtain
                \begin{align}
                    \Psi_{k}\left( \tilde{\eta} \right)\left(x\right) &= \eta\left( s_{k}\left( \tilde{\eta}, x \right) + s_{\infty}\left(0\right) \right) \\
                    &= \eta\left( s_{k}\left( \eta, x \right) \right) \\
                    &= \Psi_{k}\left(\eta\right)\left( x - \xi_{k}\left(\eta, 0\right) \right) \\
                    &= \Psi_{k}\left(\eta\right)\left( x + s_{\infty}\left( \Psi_{k}\left(\eta\right), 0 \right) \right) \\
                    &= \widetilde{\Psi_{k}\left(\eta\right)}\left(x\right).
                \end{align}
        \end{proof}
    By combining \eqref{eq:semig_psi}, \eqref{eq:semig_zeta}, Lemma \ref{lem:palm_psi} and the diagram in Figure \ref{fig:diag_eta_zeta}, the relation between $\eta, \zeta$ and the $k$-skip map can be expressed by the diagram, see Figure \ref{fig:diag_psi}.
    \begin{figure}[H]
        \centering
        \begin{tikzcd}[column sep=2cm]
          \eta \ar[r, "\tau_{s_{\infty}\left(0\right)}"] \ar[d, "\Psi_{1}"] & \tilde{\eta} \ar[r, "\zeta"] \ar[d, "\Psi_{1}"] & \ar[l, "\zeta^{-1}"] \left(\zeta_{k}\left(i\right)\right)_{k \in \N, i \in \Z} \ar[d, "\Psi{1}"] \\
          \Psi_{1}\left(\eta\right) \ar[r, "\tau_{s_{\infty}\left(\Psi_{1}\left(\eta\right), 0\right)}"] \ar[d, "\Psi_{1}"] & \widetilde{\Psi_{1}\left(\eta\right)} \ar[r, "\zeta"] \ar[d, "\Psi_{1}"] & \ar[l, "\zeta^{-1}"] \left(\zeta_{k+1}\left(i\right)\right)_{k \in \N, i \in \Z} \ar[d, "\Psi_{1}"] \\
          \Psi_{2}\left(\eta\right) \ar[r, "\tau_{s_{\infty}\left(\Psi_{2}\left(\eta\right), 0\right)}"] \ar[d, "\Psi_{1}"] & \widetilde{\Psi_{2}\left(\eta\right)} \ar[r, "\zeta"] \ar[d, "\Psi_{1}"] & \ar[l, "\zeta^{-1}"] \left(\zeta_{k+2}\left(i\right)\right)_{k \in \N, i \in \Z} \ar[d, "\Psi_{1}"] \\
          \cdots & \cdots & \cdots
        \end{tikzcd}
        \caption{The relationships between $\eta$, $\zeta$ and the $k$-skip map.}\label{fig:diag_psi}
    \end{figure}

    Next we claim that for any $k > \ell$ and $i \in \N$, there is a one-to-one correspondence between the $i$-th $k$-soliton in $\eta$ and the $i$-th $(k - \ell)$-soliton in $\Psi_{\ell}(\eta)$. 
            \begin{lemma}\label{lem:shift_X}
                Assume that $\eta \in \Omega$.  For any $k, \ell \in \N$, $k > \ell$, $h \in \Z_{\ge 0}$ and $i \in \Z$, we have
                    \begin{align}
                        \xi_{\ell + h}\left(\eta, s_{k + h}\left(\eta, i\right)\right) 
                &= \xi_{\ell}\left(\Psi_{h}\left(\eta\right), s_{k}\left(\Psi_{h}\left(\eta\right), i\right) \right) \\
                &= s_{k - \ell}\left( \Psi_{\ell + h}\left(\eta\right), i \right) - s_{\infty}\left(\Psi_{\ell + h}\left(\eta\right), 0\right) \label{eq:shift_X_1} .
                    \end{align}
                In particular, if $s_{k + h}(\eta, i) = X^{(j)}_{k + h}(\eta,0)$ for some $j \in \Z$, then 
                    \begin{align}
                        s_{k - \ell}\left(\Psi_{\ell + h}\left(\eta\right), i \right) = X^{(j)}_{k-\ell}\left(\Psi_{\ell + h}\left(\eta\right),0\right) \label{eq:shift_X_2},
                    \end{align}
                and thus
                    \begin{align}
                        \xi_{\ell + h}\left(\eta, X^{(j)}_{k + h}\left(\eta,0\right)\right) &= \xi_{\ell}\left(\Psi_{h}\left(\eta\right), X^{(j)}_{k}\left(\Psi_{h}\left(\eta\right),0\right) \right) \\
                        &= X^{(j)}_{k - \ell}\left( \Psi_{\ell+h}\left(\eta\right),0 \right) - s_{\infty}\left(\Psi_{\ell + h}\left(\eta\right), 0\right) \\
                        &= X^{(j)}_{k - \ell}\left( \Psi_{\ell+h}\left(\tilde{\eta}\right) ,0 \right) \label{eq:shift_X}. 
                    \end{align}
            \end{lemma}
            \begin{proof}
                First we note that \eqref{eq:shift_X_1} is a direct consequence of \eqref{eq:semig_psi}, \eqref{eq:S_417} and \eqref{eq:s_psi_xi}. 
                
                Next we will show \eqref{eq:shift_X_2}. From the assumption we get 
                    \begin{align}
                        \sum_{x = s_{k + h}\left(\eta, i\right) + 1}^{s_{k + h}\left(\eta, i + 1\right)} \left( \eta^{\uparrow}_{k + h}\left(x\right) - \eta^{\uparrow}_{k + h + 1}\left(x\right) \right) = \zeta_{k+h}\left(\eta, i\right) > 0.
                    \end{align}
                Since for any $i \in \Z$, 
                    \begin{align}
                        s_{\ell + h}\left( \eta, \xi_{\ell + h}\left( \eta, s_{k + h}\left(\eta, i\right) \right) \right) = s_{k + h}\left(\eta, i\right),
                    \end{align}
                we obtain 
                    \begin{align}
                        &\sum_{x = s_{k + h}\left(\eta, i\right) + 1}^{s_{k + h}\left(\eta, i + 1\right)} \left( \eta^{\uparrow}_{k + h}\left(x\right) - \eta^{\uparrow}_{k + h + 1}\left(x\right) \right) \\
                        &= \sum_{x = \xi_{\ell + h}\left( \eta, s_{k + h}\left(\eta, i\right) \right) + 1}^{\xi_{\ell + h}\left( \eta, s_{k + h}\left(\eta, i + 1\right) \right)} \left( \eta^{\uparrow}_{k + h}\left(s_{\ell + h}\left( \eta, x \right)  \right) - \eta^{\uparrow}_{k + h + 1}\left(s_{\ell + h}\left( \eta, x \right)  \right) \right) \\
                        &= \sum_{x = \xi_{\ell + h}\left( \eta, s_{k + h}\left(\eta, i\right) \right) - \xi_{\ell + h}\left( \eta, 0 \right) + 1}^{\xi_{\ell + h}\left( \eta, s_{k + h}\left(\eta, i + 1\right) \right) - \xi_{\ell + h}\left( \eta, 0 \right)  } \left( \Psi_{\ell + h}\left(\eta\right)^{\uparrow}_{k - \ell}\left(x\right) - \Psi_{\ell + h}\left(\eta\right)^{\uparrow}_{k - \ell + 1}\left(x\right) \right) \\
                        &= \sum_{x = s_{k - \ell}\left( \Psi_{\ell + h}\left(\eta\right), i \right) + 1}^{s_{k - \ell}\left( \Psi_{\ell + h}\left(\eta\right), i + 1 \right) } \left( \Psi_{\ell + h}\left(\eta\right)^{\uparrow}_{k - \ell}\left(x\right) - \Psi_{\ell + h}\left(\eta\right)^{\uparrow}_{k - \ell + 1}\left(x\right) \right) \\
                        &= \zeta_{k - \ell}\left( \Psi_{\ell + h}\left(\eta\right), i \right) > 0.
                    \end{align}
                Hence there is a $(k-\ell)$-soliton with volume at site $s_{k - \ell}\left( \Psi_{\ell + h}\left(\eta\right), i \right)$ in $\Psi_{\ell + h}\left(\eta\right)$, and thus there exists some $\tilde{j} \in \Z$ such that 
                    \begin{align}
                        s_{k - \ell}\left(\Psi_{\ell + h}\left(\eta\right), i \right) = X^{(\tilde{j})}_{k-\ell}\left(\Psi_{\ell + h}\left(\eta\right),0\right).
                    \end{align}
                Now we show $j = \tilde{j}$.  For the case $i \ge 0$, from \eqref{eq:semig_zeta},
                    \begin{align}
                        j 
                        &= \left| \left\{ 0 \le i' \le i \ ; \ \zeta_{k - \ell}\left( \Psi_{\ell + h}\left(\eta\right), i' \right) > 0 \right\}  \right| \\
                        &= \left| \left\{ 0 \le i' \le i - 1 \ ; \ \zeta_{k + \ell}\left(\eta, i' \right) > 0 \right\}  \right| \\
                        &= \tilde{j}.
                    \end{align}
                For the case $i < 0$, from \eqref{eq:semig_zeta},
                    \begin{align}
                        j - 1 
                        &= \left| \left\{ i \le i' \le -1 \ ; \ \zeta_{k - \ell}\left( \Psi_{\ell + h}\left(\eta\right), i' \right) > 0 \right\}  \right| \\
                        &= \left| \left\{ 0 \le i' \le i - 1 \ ; \ \zeta_{k + \ell}\left(\eta, i' \right) > 0 \right\}  \right| \\ 
                        &= \tilde{j} - 1.
                    \end{align}
                Thus we have $j = \tilde{j}$. {Therefore we obtain \eqref{eq:shift_X_2}. \eqref{eq:shift_X} is a direct consequence of \eqref{eq:shift_X_1} and \eqref{eq:shift_X_2}. }
                
            \end{proof}
    Thanks to Lemma \ref{lem:shift_X}, we have the following. 
    \begin{lemma}\label{lem:shift_NM}
        {Assume that $\eta \in \Omega$.} For any $k, \ell, h \in \N$, $i \in \Z$ and $n \in \N$, we have 
            \begin{align}
                N^{(i)}_{k + h,\ell + h}\left(\eta, n\right) &= N^{(i)}_{k,\ell}\left(\Psi_{h}\left(\eta\right), n\right), \label{eq:shift_N_k} \\
                M^{(i)}_{k+h}\left(\eta, n\right) &= M^{(i)}_{k}\left(\Psi_{h}\left(\eta\right), n\right) \label{eq:shift_M_k}.
            \end{align}
        In particular, \eqref{eq:shift_M_k} implies 
            \begin{align}\label{eq:lem_shift_NM}
                Y^{\left(i\right)}_{1}\left(\Psi_{k-1}\left(\eta\right), n\right) = n - M^{(i)}_{k}\left(\eta, n\right). 
            \end{align}
    \end{lemma}
    \begin{proof}

        We use induction for $n \in \N$. First we consider the case $n = 1$. 
        For \eqref{eq:shift_M_k} with $n = 1$, 
        from \eqref{eq:shift_X}, we see that the $i$-th $k + h$-soliton in $\eta$ is not free at time $0$ if and only if the $i$-th $k$-soliton in $\Psi_{h}\left(\eta\right)$ is not free at time $0$. Hence we have 
            \begin{align}
                M^{(i)}_{k+h}\left(\eta, 1\right) &= M^{(i)}_{k}\left(\Psi_{h}\left(\eta\right), 1\right). 
            \end{align}
        Next we show \eqref{eq:shift_N_k} with $n = 1$. We fix $i \in \Z$ and $k, \ell, h \in \N$ such that $k > \ell$. 
        Then there exists some $j \in \Z$ such that 
            \begin{align}
                \xi_{k + h - 1}\left(\eta, X^{(i)}_{k + h}\left(\eta,0\right)\right) &= j.
            \end{align}
        In other words, we have 
            \begin{align}
                X^{(i)}_{k + h}\left(\eta,0\right) = s_{k + h - 1}\left(\eta, j \right).
            \end{align}
        In this case, we also have 
            \begin{align}
                H_{k+h}\left( \gamma^{i}_{k + h}\left(n - 1\right) \right) = s_{k + h - 1}\left(\eta, j + 1 \right). 
            \end{align}
        We observe that from \eqref{eq:overtake},  
            \begin{align}
                &N^{(i)}_{k + h,\ell + h}\left(\eta, 1\right) \\ 
                &= 
                \begin{dcases}
                    \sum_{z = \xi_{\ell + h}\left(\eta, X^{(i)}_{k + h}\left(\eta\right)\right) + 1}^{\xi_{\ell + h}\left(\eta,  s_{k + h - 1}\left( \eta, j + 1 \right)\right)} \zeta_{\ell + h}\left(\eta, z\right) \ & \ \text{if the $i$-th $k + \ell$-soliton is free}, \\
                    0 \ & \ \text{otherwise}.
                \end{dcases}
            \end{align}
        On the other hand, from Lemma \ref{lem:shift_X}, we get
            \begin{align}
                \xi_{\ell + h}\left(\eta,  X^{(i)}_{k + h}\left(\eta,0\right)\right) &= \xi_{\ell}\left(\Psi_{h}\left(\eta\right), X^{(i)}_{k}\left(\Psi_{h}\left(\eta\right),0\right) \right), \\
                \xi_{\ell + h}\left(\eta,  s_{k + h - 1}\left( \eta, j + 1 \right)\right) &= \xi_{\ell}\left(\Psi_{h}\left(\eta\right),  s_{k - 1}\left( \Psi_{h}\left(\eta\right), j + 1 \right)\right),
            \end{align}
        and 
            \begin{align}
                s_{k - 1}\left( \Psi_{h}\left(\eta\right), j + 1 \right) = H_{k}\left( \gamma^{(i)}_{k}\left(\Psi_{h}\left(\eta\right), 1\right) \right).
            \end{align}
        Hence if the $i$-th $k + \ell$-soliton is free, we obtain
            \begin{align}
                N^{(i)}_{k + h,\ell + h}\left(\eta, 1\right) &= \sum_{z = \xi_{\ell + h}\left(\eta, X^{(i)}_{k + h}\left(\eta\right)\right) + 1}^{\xi_{\ell + h}\left(\eta,  s_{k + h - 1}\left( \eta, j + 1 \right)\right)} \zeta_{\ell + h}\left(\eta, z\right) \\
                &= \sum_{\xi_{\ell}\left(\Psi_{h}\left(\eta\right), X^{(i)}_{k}\left(\Psi_{h}\left(\eta\right)\right) \right) + 1}^{\xi_{\ell}\left(\Psi_{h}\left(\eta\right),  H_{k}\left( \gamma^{(i)}_{k}\left(\Psi_{h}\left(\eta\right), n\right) \right)\right)} \zeta_{\ell}\left(\Psi_{h}\left(\eta\right), z\right) \\
                &= N^{(i)}_{k,\ell}\left(\Psi_{h}\left(\eta\right), 1\right). 
            \end{align}

        Now we assume that \eqref{eq:shift_N_k} and \eqref{eq:shift_M_k} hold up to $n \in \N$. Then, 
            \begin{align}
                N^{(i)}_{k + h,\ell + h}\left(\eta, n+1\right) &= N^{(i)}_{k + h,\ell + h}\left(\eta, n+1\right) - N^{(i)}_{k + h,\ell + h}\left(\eta, n\right) + N^{(i)}_{k + h,\ell + h}\left(\eta, n+1\right) \\
                &= N^{(j)}_{k + h,\ell + h}\left(T^{n}\eta, 1\right) + N^{(i)}_{k,\ell}\left(\Psi_{h}\left(\eta\right), n\right) \\
                &= N^{(j)}_{k,\ell}\left(\Psi_{h}\left(T^{n}\eta\right), 1\right) + N^{(i)}_{k,\ell}\left(\Psi_{h}\left(\eta\right), n\right),
            \end{align}
        where $j = j(\eta,n)$ is uniquely determined via 
            \begin{align}
                X^{\left(j\right)}_{k+h}\left(T^{n}\eta,0\right) = X^{\left(i\right)}_{k+h}\left(\eta,n\right),
            \end{align}
        i.e., $j$ is the number assigned to the $k$-soliton at $X^{(i)}_{k+h}(\eta,n)$ in $T^{n}\eta$. From \cite[Proposition 4.{12}]{S}, we have 
            \begin{align}
                \Psi_{h}\left(T^{n}\eta\right) = \tau_{\sum_{s = 1}^{n}\sum_{m = 1}^{h} r\left(\Psi_{m-1}\left(T^{s}\eta\right),0 \right) }T^{n}\Psi_{h}\left(\eta\right),
            \end{align}
        and in particular, we get 
            \begin{align}
                X^{\left(j\right)}_{k}\left(\Psi_{h}\left(T^{n}\eta\right),0 \right) = X^{\left(i\right)}_{k}\left(\Psi_{h}\left(\eta\right),n\right) - \sum_{s = 1}^{n}\sum_{m = 1}^{h} r\left(\Psi_{m-1}\left(T^{s}\eta\right),0 \right).
            \end{align}
        Since the number of $\ell$-solitons overtaken by a tagged $k$-soliton from time $0$ to $1$ is conserved by constant spatial shift, we have
            \begin{align}
                N^{(j)}_{k,\ell}\left(\Psi_{h}\left(T^{n}\eta\right), 1\right) = N^{(i)}_{k,\ell}\left(\Psi_{h}\left(\eta\right), n + 1\right) - N^{(i)}_{k,\ell}\left(\Psi_{h}\left(\eta\right), n \right),
            \end{align}
        and thus \eqref{eq:shift_N_k} holds for $n + 1$. By using the same argument, we can also show that \eqref{eq:shift_M_k} holds for $n + 1$. 
    \end{proof}

    Now we will derive some estimates for $N_{k,\ell}, M_{k,\ell}, M_{k}$ by using $\zeta$ and the $k$-skip map. A key observation is that from Lemma \ref{lem:shift_X}, if we apply the $\ell$-skip map to $\eta$, then $k$-solitons in $\eta$ with $k > \ell$ become $(k - \ell)$-solitons in $\Psi_{\ell}(\eta)$, and $\ell$-solitons in $\eta$ become certain sites in $\Psi_{\ell}(\eta)$. Thus we see that a $k$-soltion overtaking $\ell$-soltions in $\eta$ corresponds to a $(k-\ell)$-soliton passing a certain site in $\Psi_{\ell}(\eta)$. We note that different solitons may correspond to the same site, and if $\xi_{\ell}(X^{(j)}_{\ell}) = x$ for some $j \in \Z$ and $x \in \Z$, then the site $x$ of $\Psi_{\ell}(\tilde{\eta})$ corresponds to $\zeta_{\ell}(x)$ $\ell$-solitons.  
    Hence, to find the total number of $\ell$-solitons overtaken by the $i$-th $k$-soliton in $\eta$, we only need to calculate the sum of $\zeta_{\ell}(x)$ on $x \in [X^{(i)}_{k-\ell}(\Psi_{\ell}(\eta)) + 1, X^{(i)}_{k-\ell}(\Psi_{\ell}(\eta), n))]$. Conversely, $M^{(i)}_{k}(n)$ can be calculated by counting the number of solitons passing through the site in $\Psi_{k}(\eta)$ corresponding to the $i$-th $k$-soliton with volume. 
        \begin{lemma}\label{lem:NM_kl}
        Assume that $\eta \in \Omega$. Then, for any $k,\ell \in \N$, $i \in \Z$ and $n \in \Z_{\ge 0}$, we have 
            \begin{align}\label{eq:N_kl}
                \sum_{m = 1}^{n} N^{{i}}_{k,\ell}\left(\eta,m\right) = \sum_{j = X^{{i}}_{k - \ell}\left(\Psi_{\ell}\left(\tilde{\eta}\right), 0 \right) + 1}^{X^{{i}}_{k - \ell}\left(\Psi_{\ell}\left(\tilde{\eta}\right), n \right)} \zeta_{\ell}\left(\tilde{\eta},j\right),
            \end{align}
        and 
            \begin{align}\label{eq:M_kl}
                M^{(i)}_{k}\left(\eta,n\right) = \sum_{m = 0}^{n-1} \left( 1 - r\left( T^{m}\Psi_{k}\left(\tilde{\eta}\right), J_{k}\left(\tilde{\eta}, i\right) \right) \right),
            \end{align}
        where
            \begin{align}\label{def:J_stop}
                J_{k}\left(\eta, i\right) := 
                \begin{dcases}
                    \min\left\{ j \in \Z_{\ge 0} \ ; \ \sum_{h = 0}^{j} \mathbf{1}_{\left\{ \zeta_{k}\left(\eta, h\right) > 0 \right\}} = i \right\} \ & \ i \ge 1, \\
                    - \min\left\{ j \in \Z_{\ge 0} \ ; \ \sum_{h = -j}^{-1} \mathbf{1}_{\left\{ \zeta_{k}\left(\eta, h\right) > 0 \right\}} = - i + 1 \right\}  \ & \ i \le 0.
                \end{dcases}
            \end{align} 
        In addition, 
                \begin{align}\label{ineq:M_kl}
                \sum^{J_{\ell}\left( \eta, \sigma^{(i)}_{k, \ell}\left(\eta, 0 \right) \right) - 1}_{j = J_{\ell}\left( \eta, \sigma^{(i)}_{k, \ell}\left(\eta, n \right) \right)} \zeta_{\ell}\left(\eta, j \right) 
                \le M^{(i)}_{k,\ell}\left(\eta,n\right) 
                \le \sum^{J_{\ell}\left( \eta, \sigma^{(i)}_{k, \ell}\left(\eta, 0 \right) \right)}_{j = J_{\ell}\left( \eta, \sigma^{(i)}_{k, \ell}\left(\eta, n \right) \right) - 1} \zeta_{\ell}\left(\eta, j \right),
            \end{align}
        where 
            \begin{align}\label{def:sigma_kl}
                \sigma^{(i)}_{k, \ell}\left(\eta, n\right) := \inf\left\{ j \in \Z \ ; \ X^{(j)}_{\ell - k}\left(\Psi_{k}\left(\tilde{\eta}\right), n \right) \ge J_{k}\left(\tilde{\eta}, i\right)  \right\}.
            \end{align}
    \end{lemma}
    \begin{proof}
        First we note that thanks to Lemma \ref{lem:sp_shift} and \eqref{eq:s_psi_xi}, without loss of generality we can assume that $s_{\infty}(0) = 0$. 
        
        First we prove \eqref{eq:N_kl}. Observe that {the} $i$-th $k$-soliton overtakes the $j$-th $\ell$-soliton {with volume} up to time $n$ if and only if 
            \begin{align}
                \xi_{\ell}\left(\eta, X^{{i}}_{k}\left(\eta,0\right)\right) + 1 &\le \xi_{\ell}\left(\eta, X^{(j)}_{\ell}\left(\eta,0\right)\right),
            \end{align}
        and 
            \begin{align}
                \xi_{\ell}\left(T^n\eta, X^{{i}}_{k}\left(\eta, n\right)\right) & \ge \xi_{\ell}\left(T^n\eta, X^{(j)}_{\ell}\left(\eta, n\right)\right).
            \end{align}
        {On the other hand, for any $i \in \Z$, there exists a unique $i' \in \Z$ such that $\gamma^{i}_{k} \in Con\left( \gamma^{(i')}_{k} \right)$. Since the map $\Psi_{\ell}$ skips $h$-seats with $1 \le h \le \ell$, we have 
            \begin{align}
                &X^{i}_{k-\ell}\left( \Psi_{\ell}\left(\eta\right), 0 \right) - X^{(i')}_{k-\ell}\left( \Psi_{\ell}\left(\eta\right), 0 \right) \\
                &= X^{i}_{k}\left( \eta, 0 \right) - X^{(i')}_{k}\left( \eta, 0 \right) - \sum_{h = 1}^{\ell} \sum_{\sigma \in \{\uparrow, \downarrow\}} \sum_{x = X^{i}_{k}\left( \eta, 0 \right) + 1}^{X^{(i')}_{k}\left( \eta, 0 \right)} \eta^{\sigma}_{h}\left(x\right) \\
                &= \xi_{\ell}\left(\eta, X^{i}_{k}\left( \eta, 0 \right)\right) - \xi_{\ell}\left(\eta, X^{(i')}_{k}\left( \eta, 0 \right)\right).
            \end{align}
        Also, we get 
            \begin{align}
                &X^{i}_{k-\ell}\left( \Psi_{\ell}\left(\eta\right), n \right) - X^{(i')}_{k-\ell}\left( \Psi_{\ell}\left(\eta\right), n \right) \\
                &= \xi_{\ell}\left(T^n \eta, X^{i}_{k}\left( \eta, n \right)\right) - \xi_{\ell}\left(T^n \eta, X^{(i')}_{k}\left( \eta, n \right)\right).
            \end{align}}
        By combining the above with Lemmas \ref{lem:dif_xi}, \ref{lem:shift_X} and \ref{lem:shift_NM}, {we see that 
            \begin{align}
                \xi_{\ell}\left(\eta, X^{i}_{k}\left( \eta, 0 \right)\right) = X^{i}_{k-\ell}\left( \Psi_{\ell}\left(\eta\right), 0 \right),
            \end{align}
        and 
            \begin{align}
                &\xi_{\ell}\left(T^n\eta, X^{i}_{k}\left(\eta, n\right)\right) - \left( \xi_{\ell}\left(T^n\eta, X^{(j)}_{\ell}\left(\eta, n\right)\right) - \xi_{\ell}\left(\eta, X^{(j)}_{\ell}\left(\eta,0\right)\right) \right) \\
                &= \xi_{\ell}\left(T^n \eta, X^{i}_{k}\left( \eta, n \right)\right) - \xi_{\ell}\left(T^n \eta, X^{(i')}_{k}\left( \eta, n \right)\right) \\
                & \quad + \xi_{\ell}\left(\eta, X^{(i')}_{k}\left(\eta,0\right)\right) + \left(k - \ell\right)\left(n - M_{k}\left(\eta, n\right)\right) + 2\sum_{m = 1}^{n} \sum_{h = \ell + 1}^{k - 1} N^{(i')}_{k,h}\left(\eta,m\right) \\
                &= \xi_{\ell}\left(T^n \eta, X^{i}_{k}\left( \eta, n \right)\right) - \xi_{\ell}\left(T^n \eta, X^{(i')}_{k}\left( \eta, n \right)\right) + X^{(i')}_{k - \ell}\left( \Psi_{\ell}\left(\eta\right), n \right) \\
                &= X^{i}_{k - \ell}\left( \Psi_{\ell}\left(\eta\right), n \right).
            \end{align}
        Thus, the $i$-th $k$-soliton overtakes the $j$-th $\ell$-soliton {with volume} up to time $n$ if and only if 
            \begin{align}
                X^{i}_{k-\ell}\left( \Psi_{\ell}\left(\eta\right), 0 \right) + 1 \le \xi_{\ell}\left(\eta, X^{(j)}_{\ell}\left(\eta,0\right)\right) \le X^{i}_{k - \ell}\left( \Psi_{\ell}\left(\eta\right), n \right).
            \end{align}} 
        Since 
            \begin{align}\label{eq:H_xi}
                \xi_{\ell}\left(\eta, X^{(j)}_{\ell}\left(\eta,0 \right)\right) = J_{\ell}\left(\eta, j \right),
            \end{align}
        {and there are $\zeta_{\ell}\left(\eta, J_{\ell}\left(\eta, j \right)\right)$ $\ell$-solitons at the site $J_{\ell}\left(\eta, j \right)$,} we have \eqref{eq:N_kl}. 

        Next we show \eqref{eq:M_kl} for $k < \ell$. We observe that the $i$-th $k$-soliton is free at time $n$ if and only if the site,
            \begin{align}
                s_{k}\left( T^{n}\eta, \xi_{k}\left( T^{n}\eta, X^{(i)}_{k}(\eta,n) \right) \right),
            \end{align}
        is a record in $T^n \eta$. In addition, the function $\xi_{k}(T^n\eta, \ \cdot \ )$ increases at each record in $T^{n} \eta$. Hence, the $i$-th $k$-soliton is free at time $n$ if and only if 
            \begin{align}
                \xi_{k}\left( T^{n}\eta, X^{(i)}_{k}(\eta,n) \right) \nin \left[ \xi_{k}\left( T^{n}\eta, X^{(j)}_{\ell}(\eta,n) \right) + 1, \xi_{k}\left( T^{n}\eta, \bar{X}^{(j)}_{\ell}(\eta,n) \right) \right],
            \end{align}
        for any $j \in \Z$ and $\ell > k$, where 
            \begin{align}
                \bar{X}^{(j)}_{\ell}(\eta,n) := \max\left\{ x \in \Z \  ; \  x \in Con\left(\gamma^{(j)}_{\ell}(n)\right) \right\}. 
            \end{align}
        On the other hand, from Lemmas \ref{lem:dif_xi} and  \ref{lem:shift_X}, we have 
            \begin{align}
                \xi_{k}\left( T^{n}\eta, X^{(i)}_{k}(\eta,n) \right) - \xi_{k}\left( \eta, X^{(i)}_{k}(\eta,0) \right) = \sum_{m=1}^{n}\left( k + o_{k}\left( T^{m-1}\eta \right) \right),
            \end{align}
        and 
            \begin{align}
                &\xi_{k}\left( T^{n}\eta, X^{(j)}_{\ell}(\eta,n) \right) - \xi_{k}\left( \eta, X^{(j)}_{\ell}(\eta,0) \right) \\ 
                &= \left( k - \ell\right) \left(n - M^{(j)}_{\ell}\left(\eta,n\right)\right) + 2 \sum_{h = k + 1}^{\ell - 1}(h - k) N^{(i)}_{\ell,h}\left(\eta, n\right) \\
                & \ + \sum_{m=1}^{n}\left( k + o_{k}\left( T^{m-1}\eta \right) \right). 
            \end{align}
        By \eqref{eq:shift_X}, \eqref{eq:shift_N_k} and \eqref{eq:shift_M_k} we get 
            \begin{align}
                &\xi_{k}\left( T^{n}\eta, X^{(j)}_{\ell}(\eta,n) \right) \\
                &= X^{(j)}_{\ell - k}\left(\Psi_{k}\left(\eta\right),0\right) + \left( k - \ell\right) \left(n - M^{(j)}_{\ell - k}\left(\Psi_{k}\left(\eta\right),n\right)\right) \\ 
                & \ + 2 \sum_{h = 1}^{\ell - k - 1} h N^{(i)}_{\ell - k,h}\left(\Psi_{k}\left(\eta\right), n\right)
                 + \sum_{m=1}^{n}\left( k + o_{k}\left( T^{m-1}\eta \right) \right) \\
                &= X^{(j)}_{\ell - k}\left(\Psi_{k}\left(\eta\right),n\right) + \sum_{m=1}^{n}\left( k + o_{k}\left( T^{m-1}\eta \right) \right).
            \end{align}
        Now we consider an expression of $\xi_{k}\left( T^{n}\eta, \bar{X}^{(j)}_{\ell}(\eta,n) \right)$. Observe that there exists   $j' = j'(n) \in \Z$ such that 
            \begin{align}
                X^{(j')}_{\ell}\left(T^n\eta,0\right) = X^{(j)}_{\ell}\left(\eta,n\right).
            \end{align}
        Since from Remark \ref{rem:eff_dis}, the volume of solitons are conserved in time, we have  $|Con\left(\gamma^{(j')}_{\ell}\left(T^n\eta, 0\right)\right)| = |Con\left(\gamma^{(j)}_{\ell}\left(\eta, n\right)\right)|$. In particular, $\bar{X}^{(j')}_{\ell}\left(T^n\eta,0\right)=\bar{X}^{(j)}_{\ell}(\eta,n)$. Since there are no $h$-solitons with $h \ge \ell + 1$ in the interval $[X^{(j')}_{\ell}\left(T^n\eta,0\right), \bar{X}^{(j')}_{\ell}\left(T^n\eta,0\right)]$ at time $n$, from Remark \ref{rem:seat_soliton}, the difference of $\xi_{k}\left( T^{n}\eta, \bar{X}^{(j')}_{\ell}\left(T^n\eta,0\right) \right) - \xi_{k}\left( T^{n}\eta, X^{(j')}_{\ell}\left(T^n\eta,0\right) \right)$ is equal to the total number of $h$-th head and tail with $h \ge k + 1$ in $[X^{(j')}_{\ell}\left(T^n\eta,0\right), \bar{X}^{(j')}_{\ell}\left(T^n\eta,0\right)]$, i.e.,
            \begin{align}
                &\xi_{k}\left( T^{n}\eta, \bar{X}^{(j)}_{\ell}(\eta,n) \right) - \xi_{k}\left( T^{n}\eta, X^{(j)}_{\ell}(\eta,n) \right) \\
                &= \xi_{k}\left( T^{n}\eta, \bar{X}^{(j')}_{\ell}\left(T^n\eta,0\right) \right) - \xi_{k}\left( T^{n}\eta, X^{(j')}_{\ell}\left(T^n\eta,0\right) \right) \\
                &= 2\sum_{h = k + 1}^{\ell} \sum_{x = \xi_{h}\left(T^{n}\eta, X^{(j')}_{\ell}\left(T^n\eta,0\right)\right) + 1}^{\xi_{h}\left(T^{n}\eta,\bar{X}^{(j')}_{\ell}\left(T^n\eta,0\right)\right)} (h - k) \zeta_{h}\left(T^{n}\eta,x\right).
            \end{align}
        Then from \eqref{eq:semig_zeta} and \eqref{eq:shift_X_2}, we get 
            \begin{align}
                &2\sum_{h = k + 1}^{\ell} \sum_{x = \xi_{h}\left(T^{n}\eta, X^{(j')}_{\ell}\left(T^n\eta,0\right)\right) + 1}^{\xi_{h}\left(T^{n}\eta,\bar{X}^{(j')}_{\ell}\left(T^n\eta,0\right)\right)} (h - k) \zeta_{h}\left(T^{n}\eta,x\right) \\
                &= 2\sum_{h = 1}^{\ell - k} \sum_{x = \xi_{h-k}\left(\Psi_{k}\left(T^n\eta\right), X^{(j')}_{\ell-k}\left(\Psi_{k}\left(T^n\eta\right),0\right)\right) + 1}^{\xi_{h-k}\left(\Psi_{k}\left(T^n\eta\right),\bar{X}^{(j')}_{\ell}\left(\Psi_{k}\left(T^n\eta\right),0\right)\right)} h \zeta_{h}\left(\Psi_{k}\left(T^{n}\eta\right),x\right) \\
                &= \bar{X}^{(j')}_{\ell-k}\left(\Psi_{k}\left(T^n\eta\right),0\right) - X^{(j')}_{\ell-k}\left(\Psi_{k}\left(T^n\eta\right),0\right) \\
                &= \bar{X}^{(j)}_{\ell - k}\left(\Psi_{k}\left(\eta\right),n\right) - X^{(j)}_{\ell - k}\left(\Psi_{k}\left(\eta\right),n\right).
            \end{align}
        From the above we have 
            \begin{align}
                \xi_{k}\left( T^{n}\eta, \bar{X}^{(j)}_{\ell}(\eta,n) \right) = \bar{X}^{(j)}_{\ell - k}\left(\Psi_{k}\left(\eta\right),n\right) + \sum_{m=1}^{n}\left( k + o_{k}\left( T^{m-1}\eta \right) \right).
            \end{align}
        By combining the above, we see that $i$-th $k$-soliton is free at time $n$ if and only if 
            \begin{align}
                \xi_{k}\left( \eta, X^{(i)}_{k}\left(\eta,0\right) \right)  \nin 
                \left[ X_{\ell - k}^{(j)}\left( \Psi_{k}\left(\eta\right), n \right)  + 1, \bar{X}_{\ell - k}^{(j)}\left( \Psi_{k}\left(\eta\right), n \right) \right],
            \end{align}
        for any $j \in \Z$ and $\ell > k$, and this is equivalent to 
            \begin{align}
                r\left( T^{n}\Psi_{k}\left(\eta\right), \xi_{k}\left(\eta, X^{\left(i\right)}_{k}\left(\eta,0\right) \right)\right) = 1. 
            \end{align}
        {By \eqref{eq:H_xi},} we have \eqref{eq:M_kl}.

        Finally we show \eqref{ineq:M_kl}. By the same computation as above, we see that $i$-th $k$-soliton will be overtaken by the $j$-th $\ell$-soliton up to time $n$ if and only if 
            \begin{align}
                X^{(j)}_{\ell - k}\left( \Psi_{k}\left(\eta\right),0 \right) + 1 \le \xi_{k}\left(\eta, X^{(i)}_{k}\left(\eta,0\right)\right) \le X^{(j)}_{\ell - k}\left( \Psi_{k}\left(\eta\right), n \right). 
            \end{align}
        On the other hand, we see that 
            \begin{align}
                X^{\left( \sigma^{(i)}_{k, \ell}\left(\eta, 0\right) - 1 \right)}_{\ell - k}\left( \Psi_{k}\left(\eta\right), 0 \right) < \xi_{k}\left(\eta, X^{(i)}_{k}\left(\eta, 0\right)\right)  \le X^{\left( \sigma^{(i)}_{k, \ell}\left(\eta, 0\right)\right)}_{\ell - k}\left( \Psi_{k}\left(\eta\right) \right), 
            \end{align}
        and 
            \begin{align}
                X^{\left( \sigma^{(i)}_{k, \ell}\left(\eta, n\right) - 1 \right)}_{\ell - k}\left( \Psi_{k}\left(\eta\right) , n \right) 
                < \xi_{k}\left(\eta, X^{(i)}_{k}\left(\eta,0\right)\right) 
                \le X^{\left( \sigma^{(i)}_{k, \ell}\left(\eta, n\right)\right)}_{\ell - k}\left( \Psi_{k}\left(\eta\right), n \right). 
            \end{align}
        Hence, if $\sigma^{(i)}_{k, \ell}\left(\eta, n\right) \le j \le \sigma^{(i)}_{k, \ell}\left(\eta, 0\right)- 1$, then the $j$-th $\ell$-soliton will overtake the $i$-th $k$-soliton up to time $n$. Now we observe that 
            \begin{align}
                \sum_{J_{\ell}\left(\eta, j \right) + 1}^{J_{\ell}\left(\eta, j + 1 \right) - 1} \zeta_{\ell}\left(\eta, j\right) = 0,
            \end{align}
        for any $\ell \in \N$ and $j \in \N$. From the above and \eqref{eq:H_xi}, we have \eqref{ineq:M_kl}. 
    \end{proof}

    The following representation of $Y^{{i}}_{k}(n)$ is a key to show the main results. As we will see later in Proposition \ref{prop:dec}, the representation of $Y^{{i}}_{k}(n)$ in Lemma \ref{lem:rep_W} is an orthogonal decomposition of $Y^{{i}}_{k}(n)$, unlike the original formula \eqref{eq:pos_n}. 
        \begin{lemma}\label{lem:rep_W}
            For any $\mathbf{q} \in \mathcal{Q}$, $k \in \N$, $0 \le \ell \le k - 1$, $i \in \Z$ and $n \in \Z_{\ge 0}$, we have
                \begin{align}
                    &Y^{{i}}_{k-\ell}\left(\Psi_{\ell}\left(\eta\right),n\right) \\
                    &= \frac{v^{\mathrm{eff}}_{k-\ell}\left(\theta^{\ell}\mathbf{q}\right)}{v^{\mathrm{eff}}_{1}\left(\theta^{k-1}\q\right)} \left(n -  M^{{i}}_{k}\left(\eta,n\right) \right) \\
                    & \ + 2\sum_{h = \ell + 1}^{k - 1} \frac{v^{\mathrm{eff}}_{h-\ell}\left(\theta^{\ell}\mathbf{q}\right)}{v^{\mathrm{eff}}_{1}\left(\theta^{h-1}\q\right)} \sum_{j = X^{{i}}_{k - h}\left(\Psi_{h}\left(\tilde{\eta}\right), 0 \right)  + 1}^{X^{{i}}_{k - h}\left(\Psi_{h}\left(\tilde{\eta}\right), n \right) } \left( \zeta_{h}\left({\eta,}j\right) - \a_{h}\left(\q\right) \right) \label{eq:orthogonal0},
                \end{align}
            with convention $\sum_{\ell = 1}^{0} = 0$. In particular, we have 
                \begin{align}
                    Y^{{i}}_{k}\left(\eta,n\right) 
                    &= \frac{v^{\mathrm{eff}}_{k}\left(\mathbf{q}\right)}{v^{\mathrm{eff}}_{1}\left(\theta^{k-1}\q\right)}\left(n -  M^{{i}}_{k}\left(\eta,n\right) \right) \\
                    & \ + 2\sum_{h = 1}^{k - 1} \frac{v^{\mathrm{eff}}_{h}\left(\mathbf{q}\right)}{v^{\mathrm{eff}}_{1}\left(\theta^{h-1}\q\right)} \sum_{j = X^{{i}}_{k - h}\left(\Psi_{h}\left(\tilde{\eta}\right), 0 \right)  + 1}^{X^{{i}}_{k - h}\left(\Psi_{h}\left(\tilde{\eta}\right), n \right) } \left( \zeta_{h}\left({\eta,}j\right) - \a_{h}\left(\q\right) \right) \label{eq:orthogonal}.
                \end{align}
        \end{lemma}

\begin{proof}[Proof of Lemma \ref{lem:rep_W}]
        First we note that from Lemma \ref{lem:sp_shift}, for any $k \in \N$, $i \in \Z$ and $n \in \Z_{\ge 0}$, we have $Y^{{i}}_{k}(\eta, n) = Y^{{i}}_{k}(\tilde{\eta}, n)$.
        Hence, without loss of generality, we can assume that $s_{\infty}(0) = 0$.

        We fix $\q \in \mathcal{Q}$, {$k \in \N$, $i \in \Z$ and $n \in \Z_{\ge 0}$. Then, there exists a unique $j \in \Z$ such that $M^{i}_{k}\left(\eta,n\right) = M^{(j)}_{k}\left(\eta,n\right)$.}  
        From \eqref{eq:semig_psi}, \eqref{eq:semig_zeta}, \eqref{eq:shift_q}, Lemmas \ref{lem:shift_NM} and \ref{lem:NM_kl}, for any $0 \le \ell < k$, we get 
        \begin{align} 
            &Y^{{i}}_{k- \ell}\left(\Psi_{\ell}\left(\eta\right), n\right) \\ 
            &= \left(k - \ell\right)\left(n - M^{{(j)}}_{k - \ell}\left(\Psi_{\ell}\left(\eta\right),n\right)\right)  + 2 \sum_{h = 1}^{k - \ell - 1} h \sum_{j = X^{{i}}_{k - \ell - h}\left(\Psi_{\ell + h}\left(\eta\right), 0 \right)  + 1}^{X^{{i}}_{k - \ell - h}\left(\Psi_{\ell + h}\left(\eta\right), n \right) } \zeta_{h}\left(\Psi_{\ell}\left(\eta\right),j\right) \\
            &= \left(k - \ell\right)\left(n - M^{{(j)}}_{k}\left({\eta,}n\right)\right) \\
            & \quad + 2 \sum_{h = \ell + 1}^{k - 1} \left(h - \ell \right) \sum_{j = X^{{i}}_{k - h}\left(\Psi_{h}\left(\eta\right), 0 \right)  + 1}^{X^{{i}}_{k - h}\left(\Psi_{h}\left(\eta\right), n \right) } \left( \zeta_{h}\left({\eta,}j\right) - \a_{h}\left(\q\right) \right) \\
            & \quad + 2 \sum_{h = \ell + 1}^{k - 1} \left(h - \ell \right) \a_{h}\left(\q\right)  Y^{{i}}_{k- h}\left(\Psi_{h}\left(\eta\right), n\right) \label{eq:lem_rep_W}.
        \end{align}
    Hence, if we write 
        \begin{align}
            A_{k - \ell, \ell} &:= Y^{{i}}_{k - \ell}\left(\Psi_{\ell}\left(\eta\right), n\right), \\
            B_{\ell} &:= \sum_{j = X^{{i}}_{k - \ell}\left(\Psi_{\ell}\left(\eta\right), 0 \right)  + 1}^{X^{{i}}_{k - \ell}\left(\Psi_{\ell}\left(\eta\right), n \right) } \left( \zeta_{\ell}\left({\eta,}j\right) - \a_{\ell}\left(\q\right) \right), \\
            C &:= n - M^{{(j)}}_{k}\left({\eta,}n\right),
        \end{align}
    then for any $0 \le \ell \le k - 1$ we have the following system. 
        \begin{align}\label{eq:rec_ABC}
            A_{k-\ell, \ell} = (k - \ell) C + 2 \sum_{h = \ell + 1}^{k - 1} (h - \ell) \a_{h}\left(\q\right) A_{k-h, h} + 2 \sum_{h = \ell + 1}^{k - 1} (h - \ell) B_{h}, 
        \end{align}
    with convention $\sum_{h = k}^{k - 1} = 0$. By using \eqref{eq:rec_ABC} recursively starting from $\ell=k-1$ and then $\ell=k-2$, and so on to $\ell=0$, we can represent $A_{k-\ell,\ell}$ as a linear combination of $C$ and $B_{h}$, $\ell + 1 \le h \le k - 1$. Hence, for any $k \in \N$ and $0 \le \ell \le k - 1$, there exist some positive constants $b_{k,\ell,h}\left(\q\right)$, $\ell + 1 \le h \le k - 1$ and $c_{k,\ell}\left(\q\right)$ such that 
        \begin{align}\label{eq:k_l_Y}
            A_{k-\ell, \ell} = c_{k,\ell}\left(\q\right) C + \sum_{h = \ell + 1}^{k - 1} b_{k,\ell,h}\left(\q\right) B_{h},
        \end{align}
    with convention $\sum_{h = k}^{k-1} = 0$. 
    In the rest of the proof, we will show that 
        \begin{align}\label{eq:b_vhl}
            b_{k,\ell,h}\left(\q\right) = 
                \frac{2 v^{\mathrm{eff}}_{h-\ell}\left(\theta^{\ell}\mathbf{q}\right)}{v^{\mathrm{eff}}_{1}\left(\theta^{h-1}\q\right)},
        \end{align}
    for any $k \ge 2$, $0 \le \ell \le h - 1 \le k - 2$,
    and 
        \begin{align}\label{eq:c_vkl}
            c_{k,\ell}\left(\q\right) = \frac{v^{\mathrm{eff}}_{k-\ell}\left(\theta^{\ell}\mathbf{q}\right)}{v^{\mathrm{eff}}_{1}\left(\theta^{\ell-1}\q\right)},
        \end{align}
    for any $k \in \N$ and $0 \le \ell \le k - 1$. By using \eqref{eq:rec_ABC}, we have 
        \begin{align}
            A_{k-\ell, \ell} &= (k - \ell) C + 2 \sum_{h = \ell+1}^{k - 1} \left(h - \ell\right) \a_{h}\left(\q\right) A_{k-h, h} + 2 \sum_{h = \ell+1}^{k - 1} \left(h - \ell\right) B_{h} \\
            &= \left( k - \ell\right)C + 2 \sum_{h = \ell+1}^{k - 1} \left(h - \ell\right) \a_{h}\left(\q\right) \left( c_{k,h}\left(\q\right) C + \sum_{h' = h + 1}^{k - 1} b_{k,h,h'}\left(\q\right) B_{h'} \right) \\
            & \ + 2 \sum_{h = \ell+1}^{k - 1} \left(h - \ell\right) B_{h} \\
            &= \left( k - \ell + 2 \sum_{h = \ell+1}^{k - 1} \left(h - \ell\right) \a_{h}\left(\q\right) c_{k,h} \right) C \\
            & \ + 2 B_{\ell+1} + 2 \sum_{h = \ell + 2}^{k-1} \left( h - \ell  + \sum_{h' = \ell+1}^{h-1} \left(h' - \ell \right) \a_{h'}\left(\q\right) b_{k,h',h} \right) B_{h}.
        \end{align}
    Hence we have 
        \begin{align}
            b_{k,\ell,h}\left(\q\right) = 2\left(h - \ell \right) + 2\sum_{h' = \ell + 1}^{h-1} \left(h' - \ell \right) \a_{h'}\left(\q\right) b_{k,h',h}\left(\q\right) \label{eq:key_1},
        \end{align}
    and 
        \begin{align}
            c_{k,\ell} = k - \ell + 2 \sum_{h = \ell+1}^{k - 1} \left(h - \ell\right) \a_{h}\left(\q\right) c_{k,h}\label{eq:key_2},
        \end{align}
    with convention $\sum_{h'=\ell + 1}^{\ell} = 0$. 
    On the other hand, from \eqref{eq:v_eff} and \eqref{eq:shift_q}, we have 
        \begin{align}
            v^{\mathrm{eff}}_{k-\ell}\left(\theta^{\ell}\mathbf{q}\right) 
            &= \left( k - \ell \right) v^{\mathrm{eff}}_{1}\left(\theta^{k - 1}\q\right) + 2 \sum_{h = 1}^{k - \ell - 1} h \a_{\ell + h}\left( \q \right) v^{\mathrm{eff}}_{k-\ell-h}\left(\theta^{\ell + h}\mathbf{q}\right) \\
            &= \left( k - \ell \right) v^{\mathrm{eff}}_{1}\left(\theta^{k - 1}\q\right) + 2 \sum_{h = \ell + 1}^{k - 1} \left(h - \ell\right) \a_{h}\left( \q \right) v^{\mathrm{eff}}_{k-h}\left(\theta^{h}\mathbf{q}\right)\label{eq:key_3}.
        \end{align}
    By comparing \eqref{eq:key_1}, \eqref{eq:key_3}, we see that for fixed $k \in \N$, both the sequences $2 v^{\mathrm{eff}}_{k-\ell}\left(\theta^{\ell}\mathbf{q}\right) v^{\mathrm{eff}}_{1}\left(\theta^{k - 1}\q\right)^{-1}$ and $b_{h,\ell,k}\left(\q\right) $ satisfy the same inductive system for $0 \le \ell \le k - 1$, and these two sequence have the same value $2$ with $\ell = k - 1$. Hence we have \eqref{eq:b_vhl}. By the same argument for \eqref{eq:key_2} and \eqref{eq:key_3}, we also get \eqref{eq:c_vkl}. 
    Therefore Lemma \ref{lem:rep_W} is proved. 

\end{proof}

    In the rest of this subsection, we note some consequences from Lemma{s \ref{lem:NM_kl} and} \ref{lem:rep_W} and some materials in its proof. 
    Before describing those, we consider the following remark.  
    
    \begin{remark}\label{lem:indep_skip}
        From \eqref{eq:semig_zeta} and the bijectivity of $\zeta$, $\Psi_{k}(\tilde{\eta})$ can be described as a function of $\left(\zeta_{\ell}(i)\right)_{\ell \ge k + 1, i \in \Z}$. In particular, for any $k \in \N$, $\Psi_{k}({\eta})$ and $\left(\zeta_{\ell}(i)\right)_{\ell \le k, i \in \Z}$ are independent under ${\nu}_{\q}$, $\q \in \mathcal{Q}$.
    \end{remark} 

    First we prove the exponential bound for $Y^{(i)}_{k}\left(n\right)$. To describe the result, we prepare some functions. For any $\q \in \mathcal{Q}$, $k \in \N$ and $\lambda \in \R$, we define
        \begin{align}
            u_{\q, k}\left(\lambda\right) &:= \log\left(\E_{\nu_{\mathbf{q}}}\left[ \exp\left( 2 \lambda \zeta_{k}(0) \right) \right] \right) \\
            &= 
            \begin{dcases}
                \infty \ & \ \lambda \ge \frac{\log q_{k}^{-1}}{2}, \\
                \log\left( \frac{1 - q_{k}}{1 - e^{2\lambda} q_{k}} \right) \ & \ \lambda < \frac{\log q_{k}^{-1}}{2}.
            \end{dcases} \label{def:uqk}
        \end{align}
    By using $u_{\q, k}\left(\lambda\right)$, we inductively define $U_{\q, k}\left( \lambda \right)$ as $U_{\q, 1}\left(\lambda \right) := \lambda$, and
        \begin{align}\label{def:U}
            U_{\q, k}\left(\lambda \right) &:= k \lambda +  \sum_{\ell = 1}^{k - 1} \left(k - \ell\right) u_{\q, \ell}\left( U_{\q,\ell}\left(\lambda \right) \right),
        \end{align}
    for any $k \ge 2$. We note that {$\delta_{\q,k} := \sup\left\{ \lambda \in \R \ ; \ U_{\q, k}\left( \lambda \right) < \infty \right\}$ is positive  for any $\q, k$.} In addition, $U_{\q, k}\left( \lambda \right)$ is a smooth monotone convex function on $(-\infty, \delta_{\q,k})$. 
    \begin{lemma}\label{lem:expbound} 
        For any $\q \in \mathcal{Q}$ and $k \in \N$ with $q_k > 0$, {$\lambda < \delta_{\q,k}$,} $i \in \Z$ and $n {\in \Z_{\ge 0}}$, we have 
            \begin{align}\label{eq:expmoment_YM}
                \E_{\nu_{\q}}\left[\exp\left( \lambda Y^{{i}}_{k}\left(n\right)\right) \right] = \E_{\nu_{\q}}\left[\exp\left( U_{\q, k}\left(\lambda \right) \left(n - M^{{i}}_{k}\left(n\right)\right) \right) \right].
            \end{align}
    \end{lemma}
    \begin{proof}
        {First we observe that for any $k \in \N$, $i \in \Z$ and $n \in \Z_{\ge 0}$, \eqref{eq:orthogonal0} with $\ell = k - 1$ implies
            \begin{align}
                 n -  M^{{i}}_{k}\left(\eta,n\right) = Y^{{i}}_{1}\left(\Psi_{k-1}\left(\eta\right),n\right) \label{eq:expbound_0},
            \end{align}
        and we see that $M^{{i}}_{k}\left(\eta,n\right)$ is a function of $\Psi_{k-1}\left(\eta\right)$.}
        {Thus,} from Remark \ref{lem:indep_skip}, Fubini's theorem and \eqref{eq:shift_q}, we have 
            \begin{align}
                &\E_{\nu_{\q}}\left[\exp\left( \lambda Y^{(i)}_{k}(n) \right) \right] \\
                &= \E_{\nu_{\q}}\Bigg[\exp\left( k \lambda Y^{(i)}_{1}\left(\Psi_{k-1}\left({\eta}\right),n\right) + 2\lambda  \sum_{{h} = 2}^{k-1} {h} \sum_{j = X^{(i)}_{k - {h}}\left(\Psi_{{h}}\left({\eta}\right), 0 \right) + 1}^{X^{(i)}_{k - {h}}\left(\Psi_{{h}}\left({\eta}\right), n \right)} \zeta_{{h}}\left(j\right) \right)  \\
                & \quad  \quad \quad \times \exp\left( u_{\q, 1}\left(\lambda\right) Y^{(i)}_{k - 1}\left(\Psi_{1}\left({\eta}\right), n \right) \right)
                \Bigg] \label{eq:expbound_1}. 
            \end{align}
        By Lemmas \ref{lem:sp_shift},  \ref{lem:palm_psi}, \eqref{eq:lem_shift_NM} and \eqref{eq:N_kl}, {for any $0 \le \ell \le k - 1$,} we have 
            \begin{align} 
                Y^{(i)}_{k - {\ell}}\left({\Psi_{\ell}\left(\eta\right)},n\right) 
                &= {\left(k - \ell \right)} Y^{(i)}_{1}\left(\Psi_{k-1}\left(\tilde{\eta}\right),n\right) \\
                & \quad + 2  \sum_{{h} = {\ell +} 1}^{k-1} {\left(h - \ell\right)} \sum_{j = X^{(i)}_{k - {h}}\left(\Psi_{{h}}\left(\tilde{\eta}\right), 0 \right) + 1}^{X^{(i)}_{k - {h}}\left(\Psi_{{h}}\left(\tilde{\eta}\right), n \right)} \zeta_{{h}}\left(\tilde{\eta},j\right) \label{eq:expbound_2}.
            \end{align}
        By substituting {\eqref{eq:expbound_2} with $\ell = 1$} to \eqref{eq:expbound_1}, we get  
            \begin{align}
                &\E_{\nu_{\q}}\left[\exp\left( \lambda Y^{(i)}_{k}(n) \right) \right] \\
                &= \E_{\nu_{\q}}\Bigg[\exp\left( \left( \lambda k + u_{\q, 1}\left( \lambda\right)\left(k - 1\right) \right) Y^{(i)}_{1}\left(\Psi_{k-1}\left({\eta}\right),n\right) \right)  \\
                & \quad \times \exp\left( 2 \sum_{h = 2}^{k - 1} \left(\lambda h + u_{\q, 1}\left( \lambda\right) \left(h - 1\right) \right) \sum_{j = X_{k - h}\left( \Psi_{h}\left({\eta}\right), 0 \right) + 1}^{X_{k - h}\left( \Psi_{h}\left({\eta}\right), n \right)  } \zeta_{h}\left(j\right)  \right)
                \Bigg]. 
            \end{align}
        By repeating the above computation, we have \eqref{eq:expmoment_YM}. 
        
    \end{proof}
        
    Next, for any $\mathbf{q} \in \mathcal{Q}$, $k \in \N$, $i \in \Z$ and $n \in \Z_{\ge 0}$, we define $\Delta Y^{{i}}_{k, \ell}(\eta, \mathbf{q},n)$, $1 \le \ell \le k - 1$ as 
        \begin{align}\label{def:Delta_Y}
            \Delta Y^{{i}}_{k, \ell}\left(\eta, \mathbf{q}, n\right) 
            &:= \sum_{j = X^{{i}}_{k - \ell}\left(\Psi_{\ell}\left(\tilde{\eta}\right), 0 \right) + 1}^{X^{{i}}_{k - \ell}\left(\Psi_{\ell}\left(\tilde{\eta}\right), n \right) } \left( \zeta_{\ell}\left(\eta,j\right) - \a_{\ell}\left(\mathbf{q}\right) \right).
        \end{align}
    Note that from Lemma \ref{lem:palm_psi}, \eqref{eq:M_kl}, Lemmas \ref{lem:rep_W} and Remark \ref{lem:indep_skip}, $Y^{{i}}_{k}\left(n\right)$ can be represented as
        \begin{align}
            Y^{{i}}_{k}\left(\eta,n\right) - \E_{\nu_{\q}}\left[ Y^{{i}}_{k}\left(n\right) \right] &= -\frac{v^{\mathrm{eff}}_{k}\left(\mathbf{q}\right)}{v^{\mathrm{eff}}_{1}\left(\theta^{k-1}\q\right)} \left(M^{{i}}_{k}\left({\eta,}n\right) - \E_{\nu_{\mathbf{q}}}\left[ M^{{i}}_{k}\left(n\right)  \right]\right) \\
            & \ + 2\sum_{\ell = 1}^{k - 1} \frac{v^{\mathrm{eff}}_{\ell}\left(\mathbf{q}\right)}{v^{\mathrm{eff}}_{1}\left(\theta^{\ell-1}\q\right)} \Delta Y^{{i}}_{k, \ell}\left(\eta, \mathbf{q}, n\right) \label{eq:decom_Y}. 
        \end{align}
    From Lemma \ref{lem:sp_shift} and Remark \ref{lem:indep_skip}, we have the following proposition. 
        \begin{proposition}\label{prop:dec}
            For any $\mathbf{q} \in \mathcal{Q}$, $k \in \N$, $i \in \Z$ and $n \in \Z_{\ge 0}$, we have 
                \begin{align}\label{eq:uncorrelated_1}
                    \E_{\nu_{\mathbf{q}}}\left[  \Delta Y^{{i}}_{k, \ell}\left( \mathbf{q}, n\right) \right] = 0,
                \end{align}
            and
                \begin{align}
                    &\E_{\nu_{\mathbf{q}}}\left[  \Delta Y^{{i}}_{k, \ell}\left( \mathbf{q}, n\right) \Delta Y^{{i}}_{k, \ell'}\left( \mathbf{q}, n\right)  \right] \\ 
                    &= \begin{dcases}
                        \E_{\nu_{\theta^{\ell}\q}}\left[Y^{{i}}_{k-\ell}\left(n\right)\right]  \b_{\ell}\left(\q\right) \ & \ \ell = \ell', \\
                        0 \ & \ \ell \neq \ell'.
                    \end{dcases}\label{eq:uncorrelated_2}
                \end{align}
            In addition, for any $1 \le \ell \le k - 1$, we have 
                \begin{align}\label{eq:uncorrelated}
                    \E_{\nu_{\mathbf{q}}}\left[  \Delta Y^{{i}}_{k, \ell}\left( \mathbf{q}, n\right) M^{{i}}_{k}\left(n\right) \right] = 0.
                \end{align}
            By combining the above, we have 
                \begin{align}
                    &\E_{\nu_{\mathbf{q}}}\left[ \left| Y^{{i}}_{k}\left(n\right) - \E_{\nu_{\mathbf{q}}}\left[ Y^{{i}}_{k}\left( n\right) \right] \right|^2 \right] \\
                    &= \frac{v^{\mathrm{eff}}_{k}\left(\mathbf{q}\right)^2}{v^{\mathrm{eff}}_{1}\left(\theta^{k-1}\q\right)^2}  \E_{\nu_{\mathbf{q}}}\left[ \left| M^{{i}}_{k}\left(n\right) - \E_{\nu_{\mathbf{q}}}\left[ M^{{i}}_{k}\left(n\right)  \right] \right|^2 \right] \\
                    & \ + 4 \sum_{\ell = 1}^{k - 1} \frac{v^{\mathrm{eff}}_{\ell}\left(\mathbf{q}\right)^2 \E_{\nu_{\theta^{\ell}\q}}\left[Y^{{i}}_{k-\ell}\left(n\right)\right] \b_{\ell}\left(\q\right)}{v^{\mathrm{eff}}_{1}\left(\theta^{\ell-1}\q\right)^2}.
                \end{align}
        \end{proposition}
    \begin{proof}[Proof of Proposition \ref{prop:dec}]
        Since the case $k = 1$ is trivial, we consider the case $k \ge 2$. 
        We fix $i \in \Z$. 
        Since $X^{{i}}_{k-\ell}\left(\Psi_{\ell}\left(\tilde{\eta}\right),n\right)$ is $\sigma\left( \zeta_{h} \ ; \ h \ge \ell + 1 \right)$-m'ble for any $1 \le \ell \le k - 1$ and $n \in \Z_{\ge 0}$, from Remark \ref{lem:indep_skip} we have 
            \begin{align}
                \E_{\nu_\q}\left[ \Delta Y^{{i}}_{k, \ell}\left(\mathbf{q}, n\right) \Big| \sigma\left( \zeta_{h} \ ; \ h \ge \ell + 1 \right) \right]{\left(\eta\right)} = 0 \quad \nu_{\q}\text{-a.s.}
            \end{align}
        Hence we obtain \eqref{eq:uncorrelated_1} and \eqref{eq:uncorrelated_2}. In addition, { from \eqref{eq:expbound_0},} $M^{{i}}_{k}\left(n\right)$ is $\sigma\left( \zeta_{h} \ ; \ h \ge k \right)$-m'ble for any $n \in \Z_{\ge 0}$. Hence for any $1 \le \ell \le k - 1$ and $n \in \Z_{\ge 0}$ we get
            \begin{align}
                &\E_{\nu_\q}\left[\Delta Y^{{i}}_{k, \ell}\left(\mathbf{q}, n\right) M^{{i}}_{k}\left(n\right) \Big| \sigma\left( \zeta_{h} \ ; \ h \ge \ell + 1 \right) \right]{\left(\eta\right)} \\
                &= M^{{i}}_{k}\left({\eta},n\right) \E_{\nu_\q}\left[\Delta Y^{{i}}_{k, \ell}\left(\mathbf{q}, n\right) \Big| \sigma\left( \zeta_{h} \ ; \ h \ge \ell + 1 \right) \right]{\left(\eta\right)} \\
                &= 0 \quad \nu_{\q}\text{-a.s.}
            \end{align}
        Therefore we have \eqref{eq:uncorrelated}. 
        
    \end{proof}
        \begin{remark} 
            The decomposition \eqref{eq:decom_Y} might be useful to consider the long-time correlations between solitons with different sizes. Actually, from Remark \ref{lem:indep_skip}, for any $\mathbf{q} \in \mathcal{Q}$,  $k, \ell \in \N$, $k < \ell$, $i, j \in \N$, we have 
                \begin{align}
                    &\lim_{n \to \infty} \frac{1}{n} \E_{\nu_{\mathbf{q}}}\left[  \Delta Y^{{i}}_{k, h}\left( \mathbf{q}, n\right) \Delta Y^{{j}}_{\ell, h'}\left( \mathbf{q}, n\right)  \right]
                     \\
                    &= 
                    \begin{dcases}
                        v_{k - h - 1}\left(\theta^{h+1}\q\right) v^{\mathrm{eff}}_{1}\left(\theta^{k-1}\q\right) \b_{\ell+1}\left(\q\right) \ & \ h = h', \\
                        0 \ & \  h \neq h'. 
                    \end{dcases}
                \end{align}
            In addition, for any $0 \le h \le k - 2$ and $0 \le h' \le \ell - 2$, we have 
                \begin{align}
                    &\lim_{n \to \infty} \frac{1}{n} \E_{\nu_{\mathbf{q}}}\left[  \Delta Y^{{i}}_{k, h}\left( \mathbf{q}, n\right) Y^{{j}}_{1}\left(\Psi_{\ell - 1}\left(\eta\right), n\right) \right] \\
                    &= \lim_{n \to \infty} \frac{1}{n} \E_{\nu_{\mathbf{q}}}\left[ Y^{{i}}_{1}\left(\Psi_{k - 1}\left(\eta\right), n\right) \Delta Y^{{j}}_{\ell, h'}\left( \mathbf{q}, n\right)  \right] = 0.
                \end{align}
            {Hence, if the covariance of   $Y^{{i}}_{1}\left(\Psi_{k - 1}\left(\eta\right), n\right)$ and $Y^{{j}}_{1}\left(\Psi_{\ell - 1}\left(\eta\right), n\right)$ can be computed explicitly, then one can obtain the explicit correlation between the $i$-th $k$-soliton and $j$-th $\ell$-soliton,} but it does not seem to be easy to compute. 
        \end{remark}

{\section{Proof of Theorem \ref{thm:lln_lp}} \label{sec:lln_lp}

    Since $Y^{{i}}_{k}(n) / n$ converges to $v^{\mathrm{eff}}_{k}(\q)$ a.s., to show the $\mathbb{L}^p$ convergence, it is sufficient to prove that $\left( \left| Y^{{i}}_{k}(n) / n \right|^{p} \right)_{n \in \N}$ is uniformly integrable, i.e., 
        \begin{align}
            \lim_{L \to \infty} \sup_{n \in \N} \nu_{\q}\left(  \frac{Y^{{i}}_{k}(n)}{n}  > L \right) = 0, 
        \end{align}
    and
        \begin{align}
            \lim_{L \to \infty}\sup_{n \in \N}\E_{\nu_{\q}}\left[ \left| \frac{Y^{{i}}_{k}(n)}{n} \right|^{p} \mathbf{1}_{\left\{ Y^{{i}}_{k}(n) / n \ge L \right\}} \right] = 0.
        \end{align}
    We recall that $U_{\q, k}\left(\lambda\right)$ is defined in \eqref{def:U} and is smooth on $(-\infty, \delta_{\q,k})$. 
    Thanks to \eqref{eq:expmoment_YM}, we get 
        \begin{align}
            \E_{\nu_{\q}}\left[\exp\left( \frac{\lambda Y^{{i}}_{k}\left(n\right)}{n} \right) \right] 
            &= \E_{\nu_{\q}}\left[\exp\left( nU_{\q, k}\left(\frac{\lambda}{n} \right) \frac{\left(n - M^{{i}}_{k}\left(n\right)\right)}{n} \right) \right] \\
            &\le \exp\left( nU_{\q, k}\left(\frac{\lambda}{n} \right) \right),
        \end{align}
    where we use the fact $0 \le M^{{i}}_{k}\left(n\right) \le n$. 
    By the Chebyshev inequality we have 
        \begin{align}
            &\sup_{n \in \N}\nu_{\q}\left(  \frac{Y^{{i}}_{k}(n)}{n} > L \right) 
            \le e^{-\lambda L}\sup_{n \in \N} \E_{\nu_{\q}}\left[\exp\left( \frac{\lambda Y^{{i}}_{k}\left(n\right)}{n} \right) \right] \\
            &\le e^{-\lambda L} \sup_{n \in \N}\exp\left( n U_{\q, k}\left(\frac{\lambda}{n} \right) \right)  \to 0 \quad \text{as} \ L \to \infty,
        \end{align}
        because of the smoothness of $U_{\q, k}$.
        
    Moreover, from an elementary inequality $x^{p} \le \left( \lfloor p \rfloor  + 1\right)!  e^{x}$, $x \ge 0$, and the Schwarz inequality, by choosing $0 < \lambda < \delta_{\q,k} / 2$, we obtain 
        \begin{align}
            &\sup_{n \in \N}\E_{\nu_{\q}}\left[ \left| \frac{Y^{{i}}_{k}(n)}{n} \right|^{p} \mathbf{1}_{\left\{ Y^{{i}}_{k}(n) / n \ge L \right\}} \right] \\
            &\le \frac{\left( \lfloor p \rfloor  + 1\right)!}{\lambda^{p}} \sup_{n \in \N} \E_{\nu_{\q}}\left[ \exp\left( \frac{\lambda Y^{{i}}_{k}\left(n\right)}{n} \right)  \mathbf{1}_{\left\{ Y^{{i}}_{k}(n) / n \ge L \right\}} \right] \\
            &\le \frac{\left( \lfloor p \rfloor  + 1\right)!}{\lambda^{p}} \sup_{n \in \N} \E_{\nu_{\q}}\left[ \exp\left( \frac{2\lambda Y^{{i}}_{k}\left(n\right)}{n} \right) \right]^{\frac{1}{2}}  \nu_{\q}\left(  \frac{Y^{{i}}_{k}(n)}{n} > L \right)^{\frac{1}{2}} \\
            &\le \frac{\left( \lfloor p \rfloor  + 1\right)!}{\lambda^{p}}  \left(\sup_{n \in \N}\exp\left( n U_{\q, k}\left(\frac{2\lambda}{n} \right) \right)^{\frac{1}{2}} \right) \left( \sup_{n \in \N}\nu_{\q}\left(  \frac{Y^{{i}}_{k}(n)}{n} > L \right)^{\frac{1}{2}} \right) \\
            & \to 0 \quad \text{as } L \to \infty. 
        \end{align}
    Therefore Theorem \ref{thm:lln_lp} is proved.}

\section{Proof of Theorem \ref{thm:main}}\label{sec:main_proof}

\subsection{Proof of (\ref{item:1})}

    First we prepare the following simple lemmas. 
        \begin{lemma}\label{lem:clt_1}
            Let $\zeta(i)$, $i \in \N$ be i.i.d. random variables define on any probability space with $\E[\zeta(0)] = 0$ and $\E[\zeta(0)^2] = 1$, and define $S(n) := \sum_{i = 1}^{n} \zeta(i)$, $n \in \N$. Assume that $a_{\e}\left(t\right), b_{\e}\left(t\right)$ are non-deacresing function on $[0,\infty)$ such that {for any ${\mathbf{T}} > 0$,}
                \begin{align}
                    \lim_{\e \to 0} \sup_{0 \le t \le {\mathbf{T}}}\left|a_{\e}\left(t\right) -  b_{\e}\left(t\right)\right| = 0, \quad \lim_{\e \to 0} \sup_{0 \le t \le {\mathbf{T}}}\left|a_{\e}\left(t\right) -  a t \right| = 0,
                \end{align}
            with some constant $a > 0$. Then, for any ${\mathbf{T}} > 0$ and $\delta > 0$, we have 
                \begin{align}
                    \lim_{\e \to 0} \mathbb{P}\left( \sup_{0 \le t \le {\mathbf{T}}} \e \left| S\left( \left\lfloor \frac{a_{\e}\left(t\right)}{\e^2} \right\rfloor \right) - S\left( \left\lfloor \frac{b_{\e}\left(t\right)}{\e^2} \right\rfloor \right) \right| > \delta \right) = 0.
                \end{align}
        \end{lemma}
        \begin{proof}
            Let $B(t), t \ge 0$ be a standard Brownian motion defined on some probability space. 
            Thanks to the Skorokhod embedding theorem (cf. \cite[Theorem 37.7]{B}), there exists a sequence of stopping times $\tau\left(n\right)$, $n \in \Z_{\ge 0}$, $\tau_{0} := 0$ such that $\tau\left(n\right) - \tau\left(n-1\right)$, $n \in \N$ are i.i.d. and 
                \begin{align}
                    \left(B\left( \tau\left(n\right) \right), n \in \N \right) \overset{d}{=} \left(S\left(n\right), n \in \N \right).
                \end{align}
            Hence, we have 
                \begin{align}
                    &\mathbb{P}\left( \sup_{0 \le t \le T} \e \left| S\left( \left\lfloor \frac{a_{\e}\left(t\right)}{\e^2} \right\rfloor \right) - S\left( \left\lfloor \frac{b_{\e}\left(t\right)}{\e^2} \right\rfloor \right) \right| > \delta \right) \\
                    &= \mathbb{P}\left( \sup_{0 \le t \le T} \e \left| B\left( \tau\left( \left\lfloor \frac{a_{\e}\left(t\right)}{\e^2} \right\rfloor \right) \right) - B\left( \tau\left( \left\lfloor \frac{b_{\e}\left(t\right)}{\e^2} \right\rfloor \right) \right) \right| > \delta \right) \\
                    &= \mathbb{P}\left( \sup_{0 \le t \le T} \left| B\left( \e^2 \tau\left( \left\lfloor \frac{a_{\e}\left(t\right)}{\e^2} \right\rfloor \right) \right) - B\left( \e^2 \tau\left( \left\lfloor \frac{b_{\e}\left(t\right)}{\e^2} \right\rfloor \right) \right) \right| > \delta \right).
                \end{align}
            Now we claim that 
                \begin{align}\label{lim:tau_a_b}
                    \varlimsup_{\e \to 0} \sup_{0 \le t \le T} \left|\e^2 \tau\left( \left\lfloor \frac{a_{\e}\left(t\right)}{\e^2} \right\rfloor \right) - \e^2 \tau\left( \left\lfloor \frac{b_{\e}\left(t\right)}{\e^2} \right\rfloor \right) \right| = 0 \quad \text{a.s.} 
                \end{align}
            Actually, for any $t \ge 0$, we have 
                \begin{align}
                    \e^2 \tau\left( \left\lfloor \frac{a_{\e}\left(t\right)}{\e^2} \right\rfloor \right) 
                    &= \e^2 \sum_{n = 1}^{\left\lfloor a_{\e}\left(t\right) \e^{-2} \right\rfloor} \left(\tau\left(n\right) - \tau\left(n-1\right)\right)
                    \to a t \quad \text{a.s.} , 
                \end{align}
            where we use $\E[\tau(1)] = 1$. Then, for any $n \in \N$, we have 
                \begin{align}\label{eq:clt_1}
                    \varlimsup_{\e \to 0} \max_{0 \le m \le n} \left|\e^2 \tau\left( \left\lfloor \frac{a_{\e}\left(\frac{mT}{n}\right)}{\e^2} \right\rfloor \right) - \frac{amT}{n} \right| = 0 \quad \text{a.s}.
                \end{align}
            On the other hand, by the monotonicity of $\tau, a_{\e}$, we have 
                \begin{align}
                    &\sup_{t \in \left[\frac{mT}{n}, \frac{(m+1)T}{n} \right]} \left|\e^2 \tau\left( \left\lfloor \frac{a_{\e}\left(t\right)}{\e^2} \right\rfloor \right) - at \right| \\
                    &\le \left|\e^2 \tau\left( \left\lfloor \frac{a_{\e}\left(\frac{(m+1)T}{n}\right)}{\e^2} \right\rfloor \right) - \frac{amT}{n} \right| + \left|\e^2 \tau\left( \left\lfloor \frac{a_{\e}\left(\frac{mT}{n}\right)}{\e^2} \right\rfloor \right) - \frac{a(m+1)T}{n} \right| \\
                    &\le \left|\e^2 \tau\left( \left\lfloor \frac{a_{\e}\left(\frac{(m+1)T}{n}\right)}{\e^2} \right\rfloor \right) - \frac{a(m+1)T}{n} \right| + \left|\e^2 \tau\left( \left\lfloor \frac{a_{\e}\left(\frac{mT}{n}\right)}{\e^2} \right\rfloor \right) - \frac{amT}{n} \right| + \frac{2aT}{n} \\
                    &\le 2\max_{0 \le m \le n} \left|\e^2 \tau\left( \left\lfloor \frac{a_{\e}\left(\frac{mT}{n}\right)}{\e^2} \right\rfloor \right) - \frac{amT}{n} \right| + \frac{2aT}{n}.
                \end{align}
            From \eqref{eq:clt_1}, we see that 
                \begin{align}
                    \varlimsup_{\e \to 0} \sup_{0 \le t \le T} \left|\e^2 \tau\left( \left\lfloor \frac{a_{\e}\left(t\right)}{\e^2} \right\rfloor \right) - at \right| &\le \varlimsup_{\e \to 0} \max_{0 \le m \le n} \sup_{t \in \left[\frac{mT}{n}, \frac{(m+1)T}{n} \right]} \left|\e^2 \tau\left( \left\lfloor \frac{a_{\e}\left(t\right)}{\e^2} \right\rfloor \right) - at \right| \\
                    &\le \frac{2aT}{n} \quad \text{a.s.},
                \end{align}
            for any $n \in \N$. Hence we get 
                \begin{align}
                    \varlimsup_{\e \to 0} \sup_{0 \le t \le T} \left|\e^2 \tau\left( \left\lfloor \frac{a_{\e}\left(t\right)}{\e^2} \right\rfloor \right) - at \right| = 0 \quad \text{a.s.,} 
                \end{align}
            and thus from the assumption of this lemma we obtain \eqref{lim:tau_a_b}. From \eqref{lim:tau_a_b}, for any $\delta' > 0$ we have 
                \begin{align}
                    &\varlimsup_{\e \to 0}\mathbb{P}\left( \sup_{0 \le t \le T} \left| B\left( \e^2 \tau\left( \left\lfloor \frac{a_{\e}\left(t\right)}{\e^2} \right\rfloor \right) \right) - B\left( \e^2 \tau\left( \left\lfloor \frac{b_{\e}\left(t\right)}{\e^2} \right\rfloor \right) \right) \right| > \delta \right) \\
                    &\le \mathbb{P}\left( \sup_{\substack{0 \le t, s \le a(T + \delta'), \\  |t - s| \le \delta'}} \left| B\left( t \right) - B\left(s\right) \right| > \delta \right) \\ 
                    & \ + \varlimsup_{\e \to 0}\mathbb{P}\left( \sup_{0 \le t \le T} \left|\e^2 \tau\left( \left\lfloor \frac{a_{\e}\left(t\right)}{\e^2} \right\rfloor \right) - \e^2 \tau\left( \left\lfloor \frac{b_{\e}\left(t\right)}{\e^2} \right\rfloor \right) \right| > \delta' \right) \\
                    & \ + 2\varlimsup_{\e \to 0}\mathbb{P}\left( \sup_{0 \le t \le T} \left|\e^2 \tau\left( \left\lfloor \frac{a_{\e}\left(t\right)}{\e^2} \right\rfloor \right) - at \right| > \delta' \right) \\
                    &= \mathbb{P}\left( \sup_{\substack{0 \le t, s \le a(T + \delta'), \\  |t - s| \le \delta''}} \left| B\left( t \right) - B\left(s\right) \right| > \delta \right) \to 0 \quad \text{ as } \delta' \to 0. 
                \end{align}
            From the above, Lemma \ref{lem:clt_1} is proved. 
            
        \end{proof}
    Recall that $\Delta Y^{{i}}_{k, \ell}\left(\eta, \mathbf{q}, n \right)$ is defined in \eqref{def:Delta_Y}. 
        \begin{lemma}\label{lem:ran_det}
            For any $\q \in \mathcal{Q}$, $k \in \N$, $1 \le \ell \le k - 1$, $i \in \Z$, $\mathbf{T} > 0$ and $\delta > 0$, we have
                \begin{align}
                    \varlimsup_{{n \to \infty}}\nu_{\q}\left( \sup_{0 \le t \le \mathbf{T}} \frac{1}{n} \left| \Delta Y^{{i}}_{k, \ell}\left(\mathbf{q}, \lfloor n^2 t \rfloor\right) - \sum_{j = 1}^{ \lfloor v_{k-\ell}\left(\theta^{\ell}\q\right)  n^2 t \rfloor} \left( \zeta_{\ell}\left(j\right) - \a_{\ell}\left(\mathbf{q}\right) \right) \right| > \delta \right) = 0.
                \end{align}
        \end{lemma}
        \begin{proof}
            
            First we observe that since $n \mapsto X^{{i}}_{k}(n)$ is increasing in $n$, by using \eqref{eq:LLN_q} and the same argument used to derive \eqref{lim:tau_a_b} in Lemma \ref{lem:clt_1}, for any $\q \in \mathcal{Q}$, $k \in \N$ $i \in \N$ and $0 \le t \le \mathbf{T}$,  we have 
                \begin{align}\label{lim:lln_uni}
                    \lim_{\e \to 0} \sup_{0 \le t \le \mathbf{T}} \left| \frac{1}{n} X^{{i}}_{k}\left({\eta,}\left\lfloor nt \right\rfloor\right) - v_{k}\left(\q\right)t \right| = 0 \quad \nu_{\q}\text{-a.s.}
                \end{align}
            Hence, from \eqref{eq:shift_q} Remark \ref{lem:indep_skip}, Lemma \ref{lem:clt_1} and \eqref{lim:lln_uni}, the assertion of this lemma is proved.
        \end{proof}

    Thanks to Remark \ref{lem:indep_skip}, we see that the following  stochastic processes,
        \begin{align}\label{map:thm_CLT_1}
            t \mapsto \frac{1}{n} \sum_{j = 1}^{ \lfloor v^{\mathrm{eff}}_{k-\ell}\left(\theta^{\ell}\q\right) n^2 t   \rfloor} \left( \zeta_{\ell}\left({\eta,}j\right) - \a_{\ell}\left(\mathbf{q}\right) \right),
        \end{align}
    for $1 \le \ell \le k - 1$, and 
        \begin{align}\label{map:thm_CLT_2}
            t \mapsto  \frac{1}{n} \left( M^{{i}}_{k}\left(\eta, \left\lfloor n^2 t \right\rfloor \right) - \E_{\nu_{\mathbf{q}}}\left[ M^{{i}}_{k}\left(\left\lfloor n^2 t \right\rfloor \right)  \right] \right),
        \end{align}
    are independent under $\nu_{\q}$. 
    By following the standard way one can show that \eqref{map:thm_CLT_1} converges weakly to the Brownian motion in $D[0,T]$ with mean $0$ and variance $v^{\mathrm{eff}}_{k-\ell}(\theta^{\ell}\mathbf{q})\b_{\ell}(\mathbf{q})$. 
    {Now we show that for any $i \in \Z$ and $\delta > 0 $,  
        \begin{align}\label{lim:clt_1}
            \lim_{n \to \infty} \nu_{\q}\left( \frac{1}{n}\sup_{0 \le t \le \mathbf{T}} \left| M^{i}_{k}\left( \left\lfloor n^2 t \right\rfloor \right) - M^{\left(0\right)}_{k}\left( \left\lfloor n^2 t \right\rfloor \right) \right| > \delta \right) = 0.
        \end{align}
    We observe that from \eqref{ineq:overtaken}, the difference between $M^{\left(i\right)}_{k}\left(\eta,n\right)$ and $M^{\left(j\right)}_{k}\left(\eta,n\right)$ can be estimated via the number of solitons that overtake only one of them, i.e., for any $\eta \in \Omega$, $i , j \in \Z$ with $i < j$ and $n \in \N$, we have 
        \begin{align}
            &\left| M^{\left(i\right)}_{k}\left(\eta,n\right) - M^{\left(j\right)}_{k}\left(\eta,n\right) \right| \\
            &\le 2 \left| \left\{ \gamma \in \cup_{\ell \ge k+1} \Gamma_{\ell}\left(\eta\right) \ ; \  X^{\left(i\right)}_{k}\left(\eta, 0\right) < X\left(\gamma\right) < X^{\left(j\right)}_{k}\left(\eta, 0\right) \right\} \right| \\
            & \ + 2 \left| \left\{ \gamma \in \cup_{\ell \ge k+1} \Gamma_{\ell}\left(\eta\right) \ ; \  X^{\left(i\right)}_{k}\left(\eta, n\right) < X\left(\gamma\left(n\right)\right) < X^{\left(j\right)}_{k}\left(\eta, n\right) \right\} \right| + 1.
        \end{align}
    Then, by the same argument used in the proof of Lemma \ref{lem:NM_kl}, we get 
        \begin{align}
            &\left| \left\{ \gamma \in \cup_{\ell \ge k+1} \Gamma_{\ell}\left(\eta\right) \ ; \  X^{\left(i\right)}_{k}\left(\eta, 0\right) < X\left(\gamma\right) < X^{\left(j\right)}_{k}\left(\eta, 0\right) \right\} \right| \\
            &= \left| \left\{ \gamma \in \cup_{\ell \ge 1} \Gamma_{\ell}\left(\Psi_{k}\left(\tilde{\eta}\right)\right) \ ; \  J_{k}\left(\tilde{\eta},i\right) < X\left(\gamma\right) < J_{k}\left(\tilde{\eta},j\right) \right\} \right|,
        \end{align}
    and 
        \begin{align}
            &\left| \left\{ \gamma \in \cup_{\ell \ge k+1} \Gamma_{\ell}\left(\eta\right) \ ; \  X^{\left(i\right)}_{k}\left(\eta, n\right) < X\left(\gamma\left(n\right)\right) < X^{\left(j\right)}_{k}\left(\eta, n\right) \right\} \right| \\
            &= \left| \left\{ \gamma \in \cup_{\ell \ge 1} \Gamma_{\ell}\left(\Psi_{k}\left(\tilde{\eta}\right)\right) \ ; \  J_{k}\left(\tilde{\eta},i\right) < X\left(\gamma\left(n\right)\right) < J_{k}\left(\tilde{\eta},j\right) \right\} \right|.
        \end{align}
    From the above observations, we have the following uniform estimate of $\left| M^{\left(i\right)}_{k}\left(\eta,n\right) - M^{\left(j\right)}_{k}\left(\eta,n\right) \right|$ with respect to $n$: 
        \begin{align}\label{ineq:MiMj}
            \left| M^{\left(i\right)}_{k}\left(\eta,n\right) - M^{\left(j\right)}_{k}\left(\eta,n\right) \right| \le 2 \left(J_{k}\left(\tilde{\eta},j\right) - J_{k}\left(\tilde{\eta},i\right) - 1\right) + 1.
        \end{align}
    In particular, for any $i \in \Z$, we have
        \begin{align}\label{ineq:MiMj_1}
            \left| M^{i}_{k}\left(\eta,n\right) - M^{\left(0\right)}_{k}\left(\eta,n\right) \right| =  \left| M^{(i^{\prime})}_{k}\left(\eta,n\right) - M^{\left(0\right)}_{k}\left(\eta,n\right) \right| \\
            \le
            \left| M^{(i)}_{k}\left(\eta,n\right) - M^{\left(0\right)}_{k}\left(\eta,n\right) \right|  \le 
            2 \left|J_{k}\left(\tilde{\eta},i\right) - J_{k}\left(\tilde{\eta},0\right)\right| - 1,
        \end{align}
    where we use the fact that there exists a unique $i^{\prime}$ such that $\gamma^{i}_{k} \in Con\left(\gamma^{(i^{\prime})}_{k}\right)$ and $|i^{\prime}| \le |i|$. Since $\left(J_{k}\left(j\right) - J_{k}\left(j-1\right) \right)_{j \in \Z}$ are i.i.d. geometric distribution random variables with mean $q_{k}$, we have \eqref{lim:clt_1}. }
    Hence, from the assumption of this theorem, Lemmas \ref{lem:shift_NM},  \ref{lem:ran_det}, the representation \eqref{eq:decom_Y} {and \eqref{lim:clt_1},} the process \eqref{def:step_Y} converges weakly to a sum of independent Brownian motions, and its variance $D_{k}\left(\q\right)$ is given by the sum of their variance. Therefore Theorem \ref{thm:main} (\ref{item:1}) is proved.

\subsection{Proof of (\ref{item:2})}

    {We recall that $U_{\q,k}\left(\lambda\right)$ is defined in \eqref{def:U}. From Lemma \ref{lem:expbound}, if $\lambda < { \delta_{\q,k}}$, then for any  $i \in \Z$ and $n \ge 0$, we have 
            \begin{align}
                \frac{1}{n}\log \E_{\nu_{\q}}\left[\exp\left( \lambda Y^{{i}}_{k}\left(n\right)\right) \right] = \frac{1}{n}\log  \E_{\nu_{\q}}\left[\exp\left( U_{\q, k}\left(\lambda \right) \left(n - M^{{i}}_{k}\left(n\right)\right) \right) \right].
            \end{align} 
    Hence from the assumption of Theorem \ref{thm:main} (\ref{item:2}), we have 
        \begin{align}\label{eq:Lambda_Y}
            \Lambda^{Y{,i}}_{\q,k}\left(\lambda\right) = \Lambda^{M{,i}}_{\q,k}\left( U_{\q,k}\left(  \lambda \right)\right).
        \end{align}}
    We also recall that $U_{\q, k}\left( \ \cdot \  \right)$ is a smooth monotone convex function, and $\sup_{\lambda \le \delta} U_{\q, k}\left(\lambda \right) < \infty$ for sufficiently small $\delta > 0$. In addition, since $0 \le M^{(i)}_{k}(n) \le n$, we have $|\Lambda^{M{,i}}_{\q,k}(\lambda)| \le |\lambda|$ for any $\lambda$.
    Hence if $\Lambda^{M{,i}}_{\q,k}$ is essentially smooth, then so is $\Lambda^{Y{,i}}_{\q,k}$. Therefore Theorem \ref{thm:main} (2) is proved.

\section{Proof of Theorem \ref{thm:st_cor}}\label{sec:st_proof}

    First we prepare {several} lemmas. 
    {\begin{lemma}\label{lem:T_inv}
        Assume that $\q \in \mathcal{Q}$. For any $n \in \N$, $\nu_{\q}$ is $\tau_{s_{\infty}\left(T^n\eta, 0\right)} \circ T^{n}$-invariant. 
    \end{lemma}
    \begin{proof}[Proof of Lemma \ref{lem:T_inv}]
        Thanks to \cite[Lemma 5.8, footnote on page 25 and the proof of Lemma 4.3]{FNRW}, by replacing $T$ in their discussion by $T^{n}$, one can show that     
            \begin{align}
                \mu_{\q} \circ \left(T^{n}\right)^{-1} = \nu_{\q} \circ \left(\tau_{s_{\infty}\left(T^n\eta, 0\right)} \circ T^{n}\right)^{-1}.
            \end{align}
        Since $\mu_{\q}$ is $T^{n}$-invariant, $\nu_{\q}$ is $\tau_{s_{\infty}\left(T^n\eta, 0\right)} \circ T^{n}$-invariant. 
    \end{proof}}

    To describe {the next lemma,} we define $\Xi_{k}\left(\eta,i\right) := \xi_{k}\left(\eta, s_{\infty}\left(\eta,i\right) \right)$, and
        \begin{align}
            \bar{J}_{k}\left(\q\right) &:= \E_{\nu_{\q}}\left[ J_{k}(1) \right], \\
            \bar{s}_{\infty}\left(\mathbf{q}\right) &:= {\E_{\nu_{\q}}\left[ s_{\infty}\left(1\right) \right]}, \\
            \bar{\Xi}_{k}\left(\mathbf{q}\right) &:= {\E_{\nu_{\q}}\left[ \Xi_{k}\left(1\right)\right]},
        \end{align}
    where $J_k\left(i\right)$, $i \in \Z$ is defined in \eqref{def:J_stop}. 
        \begin{lemma}\label{lem:lln_ini_prob} 
                Suppose that $\mathbf{q} \in \mathcal{Q}$ and $\E_{\nu_\q}\left[ {s_{\infty}\left(1\right)}^2 \right] < \infty$. Then, for any $k \in \N$, $u \in \R$, we have
                    \begin{align}\label{eq:lln_ini_as}
                        \lim_{n \to \infty} \frac{1}{n} X^{\left( \left\lfloor nu \right\rfloor \right)}_{k}\left(0\right) = \frac{\bar{s}_{\infty}\left(\mathbf{q}\right) \bar{J}_{k}\left(\mathbf{q}\right) u}{\bar{\Xi}_{k}\left(\mathbf{q}\right)}, \quad \nu_{\q}\text{-a.s.}, 
                    \end{align}
                and 
                    \begin{align}\label{eq:lln_ini_prob}
                        \varlimsup_{L \to \infty} \varlimsup_{n \to \infty} \nu_{\mathbf{q}}\left( \left| \frac{1}{n} X^{\left( \left\lfloor nu \right\rfloor \right)}_{k}{\left(0\right)} -  \frac{\bar{s}_{\infty}\left(\mathbf{q}\right) \bar{J}_{k}\left(\mathbf{q}\right) u}{\bar{\Xi}_{k}\left(\mathbf{q}\right)} \right| > \frac{L}{\sqrt{n}} \right) = 0.
                    \end{align}             
            \end{lemma}
        \begin{proof}[Proof of Lemma \ref{lem:lln_ini_prob}]

            {First we observe that 
                \begin{align}
                    s_{\infty}\left(\eta,i+1\right) - s_{\infty}\left(\eta,i \right) 
                    &= |\mathbf{e}^{(i)}\left(\eta\right)| \\
                    &= |\mathbf{e}^{(i)}\left(\tilde{\eta}\right)| \\
                    &= s_{\infty}\left(\tilde{\eta},i+1\right) - s_{\infty}\left(\tilde{\eta},i\right),
                \end{align}
            where $\mathbf{e}^{(i)}\left(\eta\right)$ and $\tilde{\eta}$ are defined in \eqref{def:excursion} and \eqref{def:palm_eta}, respectively. Similarly, we have 
                \begin{align}
                    \Xi_{k}\left(\eta,i+1\right) - \Xi_{k}\left(\eta,i\right) 
                    &= 1 + \sum_{\ell \in \N} \sum_{x = s_{\infty}\left(\eta,i \right) + 1}^{s_{\infty}\left(\eta,i+1\right)} \left( \eta^{\uparrow}_{k + \ell}\left(x\right) + \eta^{\downarrow}_{k + \ell}\left(x\right) \right) \\
                    &= 1 + \sum_{\ell \in \N} \sum_{x = s_{\infty}\left(\tilde{\eta},i\right) + 1}^{s_{\infty}\left(\tilde{\eta},i+1\right)}  \left( \tilde{\eta}^{\uparrow}_{k + \ell}\left(x\right) + \tilde{\eta}^{\downarrow}_{k + \ell}\left(x\right) \right) \\
                    &= 1 + \sum_{\ell \in \N} \left| \left\{ \left(k + \ell, \sigma\right)\text{-seats in } \mathbf{e}^{(i)}\left(\tilde{\eta}\right), \ \sigma \in \{\uparrow, \downarrow\}\right\}  \right|.
                \end{align}
            In particular, both $s_{\infty}\left(\tilde{\eta},i+1\right) - s_{\infty}\left(\tilde{\eta},i\right)$ and $\Xi_{k}\left(\tilde{\eta},i+1\right) - \Xi_{k}\left(\tilde{\eta},i\right)$ are functions of $\mathbf{e}^{(i)}\left(\tilde{\eta}\right)$. In addition, 
                \begin{align}
                    \Xi_{k}\left(\eta,i+1\right) - \Xi_{k}\left(\eta,i\right) \le s_{\infty}\left(\eta,i+1\right) - s_{\infty}\left(\eta,i \right).
                \end{align}
            Hence from Remark \ref{lem:ex_iid} and the assumption of this lemma, both $(s_{\infty}\left(i+1\right)$ $-$ $ s_{\infty}\left(i\right))_{i \in \Z}$ and $\left(\Xi_{k}\left(i+1\right) - \Xi_{k}\left(i\right)\right)_{i \in \Z}$ are i.i.d. $\mathbb{L}^2$ sequences under $\nu_{\q}$. In addition, since $ \left(\zeta_{k}(i)\right)_{i \in \Z}$ are i.i.d. geometric random variables under $\nu_\q$, $\left(J_{k}\left(i+1\right) - J_{k}\left(i\right) \right)_{i \in \Z}$ is also an i.i.d. $\mathbb{L}^2$ sequence under $\nu_\q$.} Thus we see that for any $u \in \R$,
                    \begin{align}
                        &\lim_{n \to \infty} \frac{1}{n} s_{\infty}\left({\eta,}\left\lfloor nu \right\rfloor \right) = \bar{s}_{\infty}\left(\mathbf{q}\right) u, \quad \nu_{\q}\text{-a.s. and in } {\mathbb{L}^{2}}, \label{lim:lln_ini_s_as} \\
                        &\varlimsup_{L \to \infty} \varlimsup_{n \to \infty} \nu_{\mathbf{q}}\left( \left|\frac{1}{n} s_{\infty}\left(\left\lfloor \frac{u}{\e} \right\rfloor \right) - \bar{s}_{\infty}\left(\mathbf{q}\right) u \right| > \frac{L}{\sqrt{n}} \right) = 0 \label{lim:lln_ini_s_prob},
                    \end{align}
                    \begin{align}
                        &\lim_{n \to \infty} \frac{1}{n} J_{k}\left({\eta,}\left\lfloor nu \right\rfloor\right) = \bar{J}_{k}\left(\mathbf{q}\right) u, \quad \nu_{\q}\text{-a.s. and in } {\mathbb{L}^{2}}, \label{lim:lln_ini_H_as} \\
                        &\varlimsup_{L \to \infty} \varlimsup_{n \to \infty} \nu_{\mathbf{q}}\left( \left| \frac{1}{n} J_{k}\left(\left\lfloor nu \right\rfloor\right) - \bar{J}_{k}\left(\mathbf{q}\right) u \right| > \frac{L}{\sqrt{n}} \right) = 0,
                    \end{align}
                and 
                    \begin{align}
                        &\lim_{n \to \infty} \frac{1}{n} \Xi_{k}\left({\eta,} \left\lfloor nu \right\rfloor\right) = \bar{\Xi}_{k}\left(\mathbf{q}\right) u, \quad \nu_{\q}\text{-a.s. and in } {\mathbb{L}^{2}}, \\
                        &\varlimsup_{L \to \infty}\varlimsup_{n \to \infty} \nu_{\mathbf{q}}\left( \left| \frac{1}{n} \Xi_{k}\left( \left\lfloor nu \right\rfloor\right) - \bar{\Xi}_{k}\left(\mathbf{q}\right) u \right| > \frac{L}{\sqrt{n}} \right) = 0.
                    \end{align}
                Hence we get 
                    \begin{align}
                        &\lim_{n \to \infty} \left| \frac{1}{n} J_{k}\left(\left\lfloor nu \right\rfloor\right) -  \frac{1}{n} \Xi_{k}\left( \left\lfloor \frac{n\bar{J}_{k}\left(\mathbf{q}\right) u}{\bar{\Xi}_{k}\left(\mathbf{q}\right)} \right\rfloor\right) \right| = 0, \quad \nu_{\q}\text{-a.s. and in } {\mathbb{L}^{2}}, \label{lim:lln_ini_as} \\
                        &\varlimsup_{L \to \infty} \varlimsup_{n \to \infty} \nu_{\mathbf{q}}\left( \left| \frac{1}{n} J_{k}\left(\left\lfloor nu \right\rfloor\right) -  \frac{1}{n} \Xi_{k}\left( \left\lfloor \frac{n \bar{J}_{k}\left(\mathbf{q}\right) u}{\bar{\Xi}_{k}\left(\mathbf{q}\right)} \right\rfloor\right) \right| > \frac{L}{\sqrt{n}} \right) = 0 \label{lim:lln_ini_prob}.
                    \end{align}
                First we show \eqref{eq:lln_ini_as}. We fix $\delta > 0$, $m \in \N$ and 
                {define $A_{\delta,m} \subset \Omega_{0}$ as }
                    \begin{align}
                        {A_{\delta,m} :=} \left\{ \eta \in \Omega_{0} \ ; \ {\sup_{n \ge m}} \left| \frac{1}{n} J_{k}\left(\eta, \left\lfloor nu \right\rfloor\right) - \frac{1}{n} \Xi_{k}\left(\eta, \left\lfloor \frac{n\bar{J}_{k}\left(\mathbf{q}\right)u}{\bar{\Xi}_{k}\left(\mathbf{q}\right)} \right\rfloor\right) \right| \le \delta \right\}.
                    \end{align}
                Then, since $\xi_{k}(\eta, X^{(i)}_{k}(\eta)) =J_{k}(\eta,i)$ for any $k \in \N$ and $i \in \Z$, and 
                $\xi_{k}( \cdot )$ increases at each record, {for any $\eta \in A_{\delta,m}$ and $n \ge m$, we have}
                    \begin{align}
                        &\Xi_{k}\left(\eta, \left\lfloor \frac{n\bar{J}_{k}\left(\mathbf{q}\right)u}{\bar{\Xi}_{k}\left(\mathbf{q}\right)} \right\rfloor - \lfloor n \delta \rfloor\right) \\
                        &\le \Xi_{k}\left(\eta, \left\lfloor \frac{n\bar{J}_{k}\left(\mathbf{q}\right)u}{\bar{\Xi}_{k}\left(\mathbf{q}\right)} \right\rfloor \right) - \lfloor n \delta \rfloor \\
                        &\le \xi_{k}\left(\eta, X^{\left( \left\lfloor nu \right\rfloor \right)}_{k}\left(\eta\right)\right) \\
                        &\le \Xi_{k}\left(\eta, \left\lfloor \frac{n\bar{J}_{k}\left(\mathbf{q}\right)u}{\bar{\Xi}_{k}\left(\mathbf{q}\right)} \right\rfloor \right) + \lfloor n \delta \rfloor \\
                        &\le \Xi_{k}\left(\eta, \left\lfloor \frac{n\bar{J}_{k}\left(\mathbf{q}\right)u}{\bar{\Xi}_{k}\left(\mathbf{q}\right)} \right\rfloor + \lfloor n \delta \rfloor\right).
                    \end{align}
                {Since $\xi_{k}\left(\cdot\right)$ is an increasing function, the above inequalities imply}
                    \begin{align}
                         \frac{1}{n} s_{\infty}\left(\eta, \left\lfloor \frac{n\bar{J}_{k}\left(\mathbf{q}\right)u}{\bar{\Xi}_{k}\left(\mathbf{q}\right)} \right\rfloor - \lfloor n \delta \rfloor \right) 
                         &\le \frac{1}{n} X^{\left( \left\lfloor nu \right\rfloor \right)}_{k}\left(\eta{,0}\right) \\
                         &\le \frac{1}{n} s_{\infty}\left(\eta, \left\lfloor \frac{n\bar{J}_{k}\left(\mathbf{q}\right)u}{\bar{\Xi}_{k}\left(\mathbf{q}\right)} \right\rfloor + \lfloor n \delta \rfloor \right).
                    \end{align}
                {Thanks to \eqref{lim:lln_ini_as} and the above inequalities, we have 
                    \begin{align}
                        &\varlimsup_{n \to \infty}\left| \frac{1}{n} X^{\left( \left\lfloor nu \right\rfloor \right)}_{k}\left(\eta,0\right) - \frac{\bar{s}_{\infty}\left(\mathbf{q}\right) \bar{J}_{k}\left(\mathbf{q}\right) u}{\bar{\Xi}_{k}\left(\mathbf{q}\right)} \right|  \\
                        &\le \varlimsup_{n \to \infty}\left| \frac{1}{n} s_{\infty}\left(\eta, \left\lfloor \frac{n\bar{J}_{k}\left(\mathbf{q}\right)u}{\bar{\Xi}_{k}\left(\mathbf{q}\right)} \right\rfloor + \lfloor n \delta \rfloor \right) -  \frac{\bar{s}_{\infty}\left(\mathbf{q}\right) \bar{J}_{k}\left(\mathbf{q}\right) u}{\bar{\Xi}_{k}\left(\mathbf{q}\right)} \right| \\
                        &\quad + \varlimsup_{n \to \infty}\left| \frac{1}{n} s_{\infty}\left(\eta, \left\lfloor \frac{n\bar{J}_{k}\left(\mathbf{q}\right)u}{\bar{\Xi}_{k}\left(\mathbf{q}\right)} \right\rfloor - \lfloor n \delta \rfloor \right) -  \frac{\bar{s}_{\infty}\left(\mathbf{q}\right) \bar{J}_{k}\left(\mathbf{q}\right) u}{\bar{\Xi}_{k}\left(\mathbf{q}\right)} \right| \quad \nu_{\q}\text{-a.s.},
                    \end{align}
                for any $\delta > 0$.}
                {By \eqref{lim:lln_ini_s_as}, we get  
                    \begin{align}
                        \varlimsup_{n \to \infty}\left| \frac{1}{n} s_{\infty}\left(\eta, \left\lfloor \frac{n\bar{J}_{k}\left(\mathbf{q}\right)u}{\bar{\Xi}_{k}\left(\mathbf{q}\right)} \right\rfloor \pm \lfloor n \delta \rfloor \right) -  \frac{\bar{s}_{\infty}\left(\mathbf{q}\right) \bar{J}_{k}\left(\mathbf{q}\right) u}{\bar{\Xi}_{k}\left(\mathbf{q}\right)} \right| \le \delta \quad \nu_{\q}\text{-a.s.}
                    \end{align}
                Since $\delta > 0$ is arbitrary, we obtain \eqref{eq:lln_ini_as}. 
                Next we show \eqref{eq:lln_ini_prob}. For any $L > 0$ and $n \in \N$, we define 
                    \begin{align}
                        A_{L,n} := \left\{ \eta \in \Omega_{0} \ ; \  \left| \frac{1}{n} J_{k}\left(\eta, \left\lfloor nu \right\rfloor\right) - \frac{1}{n} \Xi_{k}\left(\eta, \left\lfloor \frac{n\bar{J}_{k}\left(\mathbf{q}\right)u}{\bar{\Xi}_{k}\left(\mathbf{q}\right)} \right\rfloor\right) \right| \le \frac{L}{4 \bar{s}_{\infty}\left(\mathbf{q}\right)\sqrt{n}} \right\}.
                    \end{align}
                From \eqref{lim:lln_ini_s_prob}, \eqref{lim:lln_ini_prob} and similar arguments used to show \eqref{eq:lln_ini_as} above, we have 
                    \begin{align}
                        &\nu_{\q}\left( \left| \frac{1}{n} X^{\left( \left\lfloor nu \right\rfloor \right)}_{k}\left(0\right) - \frac{\bar{s}_{\infty}\left(\mathbf{q}\right) \bar{J}_{k}\left(\mathbf{q}\right) u}{\bar{\Xi}_{k}\left(\mathbf{q}\right)} \right| > \frac{L}{\sqrt{n}} \right) \\
                        &\le \nu_{\q}\left( \mathbf{1}_{A_{L,n}} \left| \frac{1}{n} X^{\left( \left\lfloor nu \right\rfloor \right)}_{k}\left(0\right) - \frac{\bar{s}_{\infty}\left(\mathbf{q}\right) \bar{J}_{k}\left(\mathbf{q}\right) u}{\bar{\Xi}_{k}\left(\mathbf{q}\right)} \right| > \frac{L}{\sqrt{n}} \right) + \nu_{\q}\left(A_{L,n}^{c}\right) \\
                        &\le \nu_{\q}\left( \mathbf{1}_{A_{L,n}}\left| \frac{1}{n} s_{\infty}\left( \left\lfloor \frac{n\bar{J}_{k}\left(\mathbf{q}\right)u}{\bar{\Xi}_{k}\left(\mathbf{q}\right)} \right\rfloor + \left\lfloor \frac{L\sqrt{n}}{4 \bar{s}_{\infty}\left(\mathbf{q}\right)} \right\rfloor \right) -  \frac{\bar{s}_{\infty}\left(\mathbf{q}\right) \bar{J}_{k}\left(\mathbf{q}\right) u}{\bar{\Xi}_{k}\left(\mathbf{q}\right)} \right| > \frac{L}{2\sqrt{n}}  \right) \\
                        &\ + \nu_{\q}\left( \mathbf{1}_{A_{L,n}}\left| \frac{1}{n} s_{\infty}\left(\left\lfloor \frac{n\bar{J}_{k}\left(\mathbf{q}\right)u}{\bar{\Xi}_{k}\left(\mathbf{q}\right)} \right\rfloor - \left\lfloor \frac{L\sqrt{n}}{4 \bar{s}_{\infty}\left(\mathbf{q}\right)} \right\rfloor \right) -  \frac{\bar{s}_{\infty}\left(\mathbf{q}\right) \bar{J}_{k}\left(\mathbf{q}\right) u}{\bar{\Xi}_{k}\left(\mathbf{q}\right)} \right| > \frac{L}{2\sqrt{n}} \right) \\
                        & \ + \nu_{\q}\left(A_{L,n}^{c}\right) \\
                        &\le \nu_{\q}\left( \left| \frac{1}{n} s_{\infty}\left( \left\lfloor \frac{n\bar{J}_{k}\left(\mathbf{q}\right)u}{\bar{\Xi}_{k}\left(\mathbf{q}\right)} \right\rfloor + \left\lfloor \frac{L\sqrt{n}}{4 \bar{s}_{\infty}\left(\mathbf{q}\right)} \right\rfloor \right) -  \frac{\bar{s}_{\infty}\left(\mathbf{q}\right) \bar{J}_{k}\left(\mathbf{q}\right) u}{\bar{\Xi}_{k}\left(\mathbf{q}\right)} - \frac{L}{4 \sqrt{n}} \right| > \frac{L}{4\sqrt{n}}  \right) \\
                        &\ + \nu_{\q}\left(\left| \frac{1}{n} s_{\infty}\left( \left\lfloor \frac{n\bar{J}_{k}\left(\mathbf{q}\right)u}{\bar{\Xi}_{k}\left(\mathbf{q}\right)} \right\rfloor - \left\lfloor \frac{L\sqrt{n}}{4 \bar{s}_{\infty}\left(\mathbf{q}\right)} \right\rfloor \right) -  \frac{\bar{s}_{\infty}\left(\mathbf{q}\right) \bar{J}_{k}\left(\mathbf{q}\right) u}{\bar{\Xi}_{k}\left(\mathbf{q}\right)} + \frac{L}{4 \sqrt{n}} \right| > \frac{L}{4\sqrt{n}} \right) \\
                        & \ + \nu_{\q}\left(A_{L,n}^{c}\right) \\
                        &\to 0, \quad n \to \infty, \text{ then } L \to \infty.
                    \end{align}
                Therefore \eqref{eq:lln_ini_prob} holds.}
                
        \end{proof}

        Recall that $\sigma^{\left( i \right)}_{k,\ell}\left(\eta,n\right)$ is defined in \eqref{def:sigma_kl}.
        \begin{lemma}\label{lem:st_cor_3}
            Suppose that $\q \in \mathcal{Q}$ and $\E_{\nu_\q}\left[ {s_{\infty}\left(1\right)}^2 \right] < \infty$. Then for any $k, \ell \in \N$ with $k < \ell$ and $u, v \in \R$, we have 
                \begin{align}\label{eq:st_cor_3_1}
                    &\lim_{n \to \infty} \frac{1}{n} \sigma^{\left( \left\lfloor n u \right\rfloor \right)}_{k,\ell}\left({\eta}, 0\right) = f_{k,\ell}\left(\q, u\right), \quad \nu_{\q}\text{-a.s.,} \\
                    &\varlimsup_{L \to \infty} \varlimsup_{n \to \infty}\nu_{\q}\left( \left| \frac{1}{n} \sigma^{\left( \left\lfloor n u \right\rfloor \right)}_{k,\ell}\left(0\right) - f_{k,\ell}\left(\q, u\right) \right| > \frac{L}{\sqrt{n}} \right) = 0 \label{eq:st_cor_3_3},
                \end{align}
            and 
                \begin{align}
                    &\varlimsup_{L \to \infty} \varlimsup_{n \to \infty}\nu_{\q}\left( \left| \frac{1}{n} \sigma^{\left( \left\lfloor n u \right\rfloor \right)}_{k,\ell}\left(n^2\right) - \frac{1}{n} \sigma^{\left( \left\lfloor n v \right\rfloor \right)}_{k,\ell}\left(n^2\right) - f_{k,\ell}\left(\q, u - v\right)  \right| > \frac{L}{\sqrt{n}} \right) = 0, \\  \label{eq:st_cor_3_4}
                \end{align}
            where 
                \begin{align}
                    f_{k,\ell}\left(\q, u\right) := \frac{\bar{\Xi}_{\ell - k}\left(\theta^{k}\q\right) \bar{J}_{k}\left(\q\right)u}{\bar{s}_{\infty}\left(\theta^{k}\q\right) \bar{J}_{\ell}\left(\q\right)}.
                \end{align}
        \end{lemma}
        \begin{proof}
            We fix $\q$, $k < \ell$ and $u, v$. Without loss of generality, we can assume that $u > v$. For notational simplicity, we will write $\tilde{u} := f_{k,\ell}\left(\q, u\right)$. 
            
            First we show \eqref{eq:st_cor_3_1}. 
            We observe that from {\eqref{eq:shift_q} and} Lemma \ref{lem:lln_ini_prob}, 
                \begin{align}
                    \lim_{n \to \infty}\frac{1}{n}X^{\left( \left\lfloor n \tilde{u} \right\rfloor \right)}_{\ell - k}\left( \Psi_{k}\left(\eta\right){,0} \right) = \bar{J}_{k}\left(\q\right) u, \quad \nu_{\q}\text{-a.s.},
                \end{align}
            where we use \eqref{eq:semig_zeta} and \eqref{eq:shift_q} to show $\bar{J}_{\ell - k}\left(\theta^{k}\q\right) = \bar{J}_{\ell}\left(\q\right)$.
            We fix $\delta > 0$ and $\eta \in \Omega$ such that 
                \begin{align}
                    \varlimsup_{n \to \infty}\left| \frac{1}{n}X^{\left( \left\lfloor n \tilde{u} \right\rfloor \right)}_{\ell - k}\left( \Psi_{k}\left(\eta\right){,0} \right) - \bar{J}_{k}\left(\q\right) {u} \right| \le \delta,
                \end{align}
            and 
                \begin{align}
                    \varlimsup_{n \to \infty} \left| \frac{1}{n}J_{k}\left({\eta,}\left\lfloor n u \right\rfloor\right) - \bar{J}_{k}\left(\q\right) u \right| \le \delta.
                \end{align}
            Then, we see that 
                \begin{align}
                    {X^{\left( \left\lfloor n \tilde{u} \right\rfloor -  \left\lfloor n \delta / k \right\rfloor \right)}_{\ell - k}\left( \Psi_{k}\left(\eta\right),0 \right) \le J_{k}\left(\eta, \left\lfloor n u \right\rfloor\right) \le X^{\left( \left\lfloor n \tilde{u} \right\rfloor + \left\lfloor n \delta / k \right\rfloor \right)}_{\ell - k}\left( \Psi_{k}\left(\eta\right),0 \right),}
                \end{align}
            where we use the fact that the number of $k$-solitons contained in an interval $[a, b], a < b$ is at most $(b - a) / (2k)$. 
            {The above inequality implies 
                \begin{align}
                    \left\lfloor n \tilde{u} \right\rfloor - \left\lfloor n \delta / k \right\rfloor \le \sigma^{\left( \left\lfloor n u \right\rfloor \right)}_{k,\ell}\left(\eta,0\right) \le \left\lfloor n \tilde{u} \right\rfloor + \left\lfloor n \delta / k \right\rfloor
                \end{align}}
            Hence, by \eqref{lim:lln_ini_H_as}, we obtain
                \begin{align}
                    \lim_{n \to \infty} \frac{1}{n} \left| \sigma^{\left( \left\lfloor n u \right\rfloor \right)}_{k,\ell}\left({\eta,}0\right) - \tilde{u} \right| \le {\frac{\delta}{k}}, \quad \nu_{\q}\text{-a.s.},
                \end{align}
            for any $\delta > 0$. 
            Therefore we have \eqref{eq:st_cor_3_1}. 

            Next we show \eqref{eq:st_cor_3_3}.
            We fix $L_1 > 0$. Then, on the following event,
                \begin{align}
                    A_{L_1} &:= \left\{ \left| \frac{1}{n}X^{\left( \left\lfloor n \tilde{u} \right\rfloor \right)}_{\ell - k}\left( \Psi_{k}\left(\eta\right){,0} \right) - \bar{J}_{k}\left(\q\right) \tilde{u} \right| \le \frac{L_{1}}{\sqrt{n}} \right\} \\
                    & \ \cap \left\{ \left| \frac{1}{n}J_{k}\left({\eta,}\left\lfloor n u \right\rfloor\right) - \bar{J}_{k}\left(\q\right) u \right| \le \frac{L_{1}}{\sqrt{n}} \right\},
                \end{align}
            {by the argument used to show \eqref{eq:st_cor_3_1},} we see that 
                \begin{align}
                    \left| \frac{1}{n} \sigma^{\left( \left\lfloor n u \right\rfloor \right)}_{k,\ell}\left({\eta,}0\right) - \tilde{u} \right| \le \frac{L_1}{k \sqrt{n}} + \frac{1}{n}.
                \end{align}
            Hence we have 
                \begin{align}
                    &\varlimsup_{L_2 \to \infty} \varlimsup_{n \to \infty}\nu_{\q}\left( \left| \frac{1}{n} \sigma^{\left( \left\lfloor n u \right\rfloor \right)}_{k,\ell}\left(0\right) - \tilde{u} \right| > \frac{L_{2}}{\sqrt{n}} \right) \\
                    &\le \varlimsup_{L_2 \to \infty} \varlimsup_{n \to \infty}\nu_{\q}\left( \mathbf{1}_{A_{L_1}} \left| \frac{1}{n} \sigma^{\left( \left\lfloor n u \right\rfloor \right)}_{k,\ell}\left(0\right) - \tilde{u} \right| > \frac{L_{2}}{\sqrt{n}} \right) + \nu_{\q}\left( A^{{c}}_{L_1} \right) \\
                    &= \varlimsup_{n \to \infty} \nu_{\q}\left( A^{{c}}_{L_1} \right) \to 0, \quad L_1 \to \infty,
                \end{align}
            which concludes \eqref{eq:st_cor_3_1}. 

            Finally we show \eqref{eq:st_cor_3_4}. We observe that {for any $m \in \Z_{\ge 0}$, the difference} $\sigma^{\left( \left\lfloor n u \right\rfloor \right)}_{k,\ell}\left({\eta,m}\right) - \sigma^{\left( \left\lfloor n v \right\rfloor \right)}_{k,\ell}\left({\eta,m}\right)$ is equal to the number of $\left(\ell - k\right)$-solitons with volume in {$\Psi_{k}\left(\eta\right)$} contained in $\left[J_{k}\left({\eta,} \left\lfloor n v \right\rfloor \right), J_{k}\left({\eta,} \left\lfloor n u \right\rfloor \right)  \right]$ at time ${m}$, i.e., 
                \begin{align}
                    &\sigma^{\left( \left\lfloor n u \right\rfloor \right)}_{k,\ell}\left({\eta,m}\right) - \sigma^{\left( \left\lfloor n v \right\rfloor \right)}_{k,\ell}\left({\eta,m}\right) \\
                    &= \left| \left\{ j \in \Z \ ; \ J_{k}\left({\eta,} \left\lfloor n v \right\rfloor \right) \le X^{(j)}_{\ell - k}\left(T^{{m}} \Psi_{k}\left({\eta}\right){,0} \right) \le J_{k}\left({\eta,} \left\lfloor n u \right\rfloor \right) \right\} \right|.
                \end{align}
    {From Remark \ref{lem:indep_skip}, for any $a \in \Z_{\ge 0}$, we have
        \begin{align}
            &\nu_{\q}\left(\left| \left\{ j \in \Z ; J_{k}\left( \left\lfloor n v \right\rfloor \right) \le X^{(j)}_{\ell - k}\left(\tau_{s_{\infty}\left(T^{m}\Psi_{k}\left(\eta\right), 0\right)} T^{m}\Psi_{k}\left(\eta\right),0 \right) \le J_{k}\left( \left\lfloor n u \right\rfloor \right) \right\} \right| = a\right) \\
            &= \sum_{b,c \in \Z} \nu_{\q}\left(J_{k}\left( \left\lfloor n v \right\rfloor \right) = b, J_{k}\left( \left\lfloor n u \right\rfloor \right) = c\right) \\
            & \quad \times \nu_{\q}\left(\left| \left\{ j \in \Z \ ; b \le X^{(j)}_{\ell - k}\left(\tau_{s_{\infty}\left(T^{m}\Psi_{k}\left(\eta\right), 0\right)} T^{m}\Psi_{k}\left(\eta\right),0 \right) \le c \right\} \right| = a\right),
        \end{align}
    and by \eqref{eq:shift_q} and Lemma \ref{lem:T_inv}, we get
        \begin{align}
            &\nu_{\q}\left(\left| \left\{ j \in \Z \ ; b \le X^{(j)}_{\ell - k}\left(\tau_{s_{\infty}\left(T^{m}\Psi_{k}\left(\eta\right), 0\right)} T^{m}\Psi_{k}\left(\eta\right),0 \right) \le c \right\} \right| = a\right) \\
            &=\nu_{\theta^{k}\q}\left(\left| \left\{ j \in \Z \ ; b \le X^{(j)}_{\ell - k}\left(\tau_{s_{\infty}\left(T^{m}\eta, 0\right)} T^{m}\eta,0 \right) \le c \right\} \right| = a\right) \\
            &= \nu_{\theta^{k}\q}\left(\left| \left\{ j \in \Z \ ; b \le X^{(j)}_{\ell - k}\left(0\right) \le c \right\} \right| = a\right) \\
            &= \nu_{\q}\left(\left| \left\{ j \in \Z \ ; \ b \le X^{(j)}_{\ell - k}\left(\Psi_{k}\left(\eta\right),0 \right) \le c \right\} \right| = a\right).
        \end{align}
    Thus we obtain 
        \begin{align}
            &\nu_{\q}\left(\left| \left\{ j \in \Z ; J_{k}\left( \left\lfloor n v \right\rfloor \right) \le X^{(j)}_{\ell - k}\left(\tau_{s_{\infty}\left(T^{m}\Psi_{k}\left(\eta\right), 0\right)} T^{m}\Psi_{k}\left(\eta\right),0 \right) \le J_{k}\left( \left\lfloor n u \right\rfloor \right) \right\} \right| = a\right) \\
            &= \nu_{\q}\left(\left| \left\{ j \in \Z ; J_{k}\left( \left\lfloor n v \right\rfloor \right) \le X^{(j)}_{\ell - k}\left(\Psi_{k}\left(\eta\right),0 \right) \le J_{k}\left( \left\lfloor n u \right\rfloor \right) \right\} \right| = a\right) \\
            &= \nu_{\q}\left(\sigma^{\left( \left\lfloor n u \right\rfloor \right)}_{k,\ell}\left(0\right) - \sigma^{\left( \left\lfloor n v \right\rfloor \right)}_{k,\ell}\left(0\right) = a\right),
        \end{align}
    that is, for any $m \in \Z_{\ge 0}$, we have
        \begin{align}
            &\left| \left\{ j \in \Z ; J_{k}\left( \left\lfloor n v \right\rfloor \right) \le X^{(j)}_{\ell - k}\left(\tau_{s_{\infty}\left(T^{m}\Psi_{k}\left(\eta\right), 0\right)} T^{m}\Psi_{k}\left(\eta\right),0 \right) \le J_{k}\left( \left\lfloor n u \right\rfloor \right) \right\} \right| \\
            &\overset{d}{=} \sigma^{\left( \left\lfloor n u \right\rfloor \right)}_{k,\ell}\left(0\right) - \sigma^{\left( \left\lfloor n v \right\rfloor \right)}_{k,\ell}\left(0\right),
        \end{align}
    under $\nu_{\q}$.}
            On the other hand, since the length of a $k$-soliton is $2k$, we get 
                \begin{align}
                    &\Bigg| \sigma^{\left( \left\lfloor n u \right\rfloor \right)}_{k,\ell}\left(n^2\right) - \sigma^{\left( \left\lfloor n v \right\rfloor \right)}_{k,\ell}\left(n^2\right)   \\
                    & \quad - \left| \left\{ j \in \Z \ ; \ J_{k}\left( \left\lfloor n v \right\rfloor \right) \le X^{(j)}_{\ell - k}\left({\tau_{s_{\infty}\left(T^{n^2}\Psi_{k}\left(\eta\right), 0\right)} T^{n^2}}\Psi_{k}\left(\eta\right){,0} \right) \le J_{k}\left( \left\lfloor n u \right\rfloor \right) \right\} \right| \Bigg| \\
                    &\le \frac{\left| s_{\infty}\left(T^{n^2}\Psi_{k}\left(\eta\right), 0\right) \right|}{k} \\
                    &\le \frac{ s_{\infty}\left(T^{n^2}\Psi_{k}\left(\eta\right), 1\right)  - s_{\infty}\left(T^{n^2}\Psi_{k}\left(\eta\right), 0\right)}{k}  \\
                    &= \frac{ s_{\infty}\left({\tau_{s_{\infty}\left(T^{n^2}\Psi_{k}\left(\eta\right), 0\right)} T^{n^2}}\Psi_{k}\left(\eta\right), 1\right)}{k}.
                \end{align}
            By \eqref{eq:shift_q}{, Remark \ref{lem:indep_skip}} and Lemma \ref{lem:T_inv}, we have 
                \begin{align}
                    \nu_{\q}\left(\frac{ s_{\infty}\left({\tau_{s_{\infty}\left(T^{n^2}\Psi_{k}\left(\eta\right), 0\right)} T^{n^2}}\Psi_{k}\left(\eta\right), 1\right)}{n k} > \frac{L}{\sqrt{n}}\right) 
                    &= \nu_{\theta^{k}\q}\left(\frac{ s_{\infty}\left(1\right)}{n k} > \frac{L}{\sqrt{n}}\right)
                    \to 0,
                \end{align}
            as $n \to \infty$ for any $L > 0$. From the above and \eqref{eq:st_cor_3_3}, we obtain \eqref{eq:st_cor_3_4}.
    \end{proof}

\begin{proof}[Proof of Theorem \ref{thm:st_cor}]

    Without loss of generality, we can assume $u > v$. 
    From Lemma \ref{lem:rep_W}, for any $0 \le \ell \le k - 1$, we have 
        \begin{align}
            &Y^{\left( \left\lfloor n^{a} u \right\rfloor \right)}_{k-\ell}\left(\Psi_{\ell}\left(\eta\right), n^2 \right) -  Y^{\left(\left\lfloor n^{a} v \right\rfloor \right)}_{k-\ell}\left(\Psi_{\ell}\left(\eta\right),n^2\right) \\
                    &= \frac{v^{\mathrm{eff}}_{k-\ell}\left(\theta^{\ell}\mathbf{q}\right)}{\bar{r}_{k}\left(\q\right)} \left(M^{\left( \left\lfloor n^{a} v \right\rfloor \right)}_{k}\left(\tilde{\eta},n^2\right) -  M^{\left( \left\lfloor n^{a} u \right\rfloor \right)}_{k}\left(\tilde{\eta},n^2\right) \right) \\
                    & \ + 2\sum_{h = \ell + 1}^{k - 1} \frac{v^{\mathrm{eff}}_{h-\ell}\left(\theta^{\ell}\mathbf{q}\right)}{\bar{r}_{h}\left(\q\right)} \left( \sum_{j = X^{\left( \left\lfloor n^{a} u \right\rfloor \right)}_{k - h}\left(\Psi_{h}\left(\tilde{\eta}\right), 0 \right)  + 1}^{X^{\left( \left\lfloor n^{a} u \right\rfloor \right)}_{k - h}\left(\Psi_{h}\left(\tilde{\eta}\right), n^2 \right) } - \sum_{j = X^{\left( \left\lfloor n^{a} v \right\rfloor \right)}_{k - \ell}\left(\Psi_{h}\left(\tilde{\eta}\right), 0 \right)  + 1}^{X^{\left( \left\lfloor n^{a} v \right\rfloor \right)}_{k - h}\left(\Psi_{h}\left(\tilde{\eta}\right), n^2 \right) } \right) \\
                    & \quad \quad \quad \times \left( \zeta_{h}\left(\tilde{\eta},j\right) - \a_{h}\left(\q\right) \right) \label{eq:st_cor_s_0}.
        \end{align}
    By using Remark \ref{lem:indep_skip} and \eqref{eq:st_cor_s_0}, we get
        \begin{align}
            &\E_{\nu_{\mathbf{q}}}\left[ \left| \frac{1}{n} Y^{\left( \left\lfloor n^{a} u \right\rfloor \right)}_{k-\ell}\left( \Psi_{\ell}\left(\eta\right), n^2 \right) - \frac{1}{n} Y^{\left( \left\lfloor n^{a} v \right\rfloor \right)}_{k-\ell}\left( \Psi_{\ell}\left(\eta\right), n^2\right) \right|^{2} \right] \\
            &= \frac{v^{\mathrm{eff}}_{k-\ell}\left(\theta^{\ell}\mathbf{q}\right)^2}{\bar{r}_{k}\left(\q\right)^2 n^2} \E_{\nu_{\mathbf{q}}}\left[ \left| M^{\left( \left\lfloor n^{a} u \right\rfloor \right)}_{k}\left(n^2\right) - M^{\left( \left\lfloor n^{a} v \right\rfloor \right)}_{k}\left(n^2\right) \right|^2 \right] \\
            & \ + 4 \sum_{h = \ell + 1}^{k - 1} \frac{v^{\mathrm{eff}}_{h-\ell}\left(\theta^{\ell}\mathbf{q}\right)^2 \b_{h}\left(\mathbf{q}\right)}{\bar{r}_{h}\left(\q\right)^2 n^2}  \\
            & \quad \quad \times \E_{\nu_{\q}}\left[ \left| Y^{\left( \left\lfloor n^{a} u \right\rfloor \right)}_{k - h}\left(\Psi_{h}\left(\eta\right), n^2\right) - Y^{\left( \left\lfloor n^{a} v \right\rfloor \right)}_{k - h}\left(\Psi_{h}\left(\eta\right), n^2\right) \right|  \right] \label{eq:st_cor_s}.
        \end{align}
    where at the last line we use Lemma \ref{lem:sp_shift}. 
    For notational simplicity, we define 
        \begin{align}
            \widetilde{M}^{u,v}_{k,n} := \E_{\nu_{\mathbf{q}}}\left[ \left| M^{\left( \left\lfloor n^{a} u \right\rfloor \right)}_{k}\left(n^2\right) - M^{\left( \left\lfloor n^{a} v \right\rfloor \right)}_{k}\left(n^2\right) \right|^2 \right]. 
        \end{align}
    Then, from \eqref{eq:st_cor_s_0} with $\ell = k - 1$ and the Schwarz inequality, we get 
        \begin{align}
            &\E_{\nu_{\mathbf{q}}}\left[ \left| Y^{\left( \left\lfloor n^{a} u \right\rfloor \right)}_{1}\left(\Psi_{k-1}\left(\eta\right), n^2\right) - Y^{\left( \left\lfloor n^{a} v \right\rfloor \right)}_{1}\left(\Psi_{k-1}\left(\eta\right), n^2\right) \right|  \right] \\
            &= \frac{v^{\mathrm{eff}}_{1}\left(\theta^{k-1}\mathbf{q}\right)}{\bar{r}_{k}\left(\q\right)} \E_{\nu_{\mathbf{q}}}\left[ \left| M^{\left( \left\lfloor n^{a} v \right\rfloor \right)}_{k}\left(n^2\right) -  M^{\left( \left\lfloor n^{a} u \right\rfloor \right)}_{k}\left(n^2\right) \right|  \right] \\
            &\le \frac{v^{\mathrm{eff}}_{1}\left(\theta^{k-1}\mathbf{q}\right)}{\bar{r}_{k}\left(\q\right)} \left(\widetilde{M}^{u,v}_{k,n}\right)^{\frac{1}{2}}. 
        \end{align}
    By using this, \eqref{eq:st_cor_s} with $\ell = k - 2$ and the 
    Schwarz inequality, we see that 
        \begin{align}
            &\E_{\nu_{\mathbf{q}}}\left[ \left| Y^{\left( \left\lfloor n^{a} u \right\rfloor \right)}_{2}\left(\Psi_{k-2}\left(\eta\right), n^2\right) - Y^{\left( \left\lfloor n^{a} v \right\rfloor \right)}_{2}\left(\Psi_{k-2}\left(\eta\right), n^2\right) \right|  \right] \\
            &\le \E_{\nu_{\mathbf{q}}}\left[ \left| Y^{\left( \left\lfloor n^{a} u \right\rfloor \right)}_{2}\left(\Psi_{k-2}\left(\eta\right), n^2\right) - Y^{\left( \left\lfloor n^{a} v \right\rfloor \right)}_{2}\left(\Psi_{k-2}\left(\eta\right), n^2\right) \right|^2  \right]^{\frac{1}{2}}  \\
            &\le \left( \frac{v^{\mathrm{eff}}_{2}\left(\theta^{k-2}\mathbf{q}\right)^2}{\bar{r}_{k}\left(\q\right)^2}\widetilde{M}^{u,v}_{k,n} + \frac{4 v^{\mathrm{eff}}_{1}\left(\theta^{k-2}\q\right)^2 \b_{k-1}\left(\mathbf{q}\right) v^{\mathrm{eff}}_{1}\left(\theta^{k-1}\mathbf{q}\right)}{\bar{r}_{k - 1}\left(\q\right)^2 \bar{r}_{k}\left(\q\right)} \left(\widetilde{M}^{u,v}_{k,n}\right)^{\frac{1}{2}}  \right)^{\frac{1}{2}} \\
            &\le \frac{v^{\mathrm{eff}}_{2}\left(\theta^{k-2}\mathbf{q}\right)}{\bar{r}_{k}\left(\q\right)}\left(\widetilde{M}^{u,v}_{k,n}\right)^{\frac{1}{2}} + \frac{2 v^{\mathrm{eff}}_{1}\left(\theta^{k-2}\q\right) \b_{k-1}\left(\mathbf{q}\right) v^{\mathrm{eff}}_{1}\left(\theta^{k-1}\mathbf{q}\right)^{\frac{1}{2}}}{\bar{r}_{k - 1}\left(\q\right) \bar{r}_{k}\left(\q\right)^{\frac{1}{2}}} \left(\widetilde{M}^{u,v}_{k,n}\right)^{\frac{1}{4}}. 
        \end{align}
    By repeating the above procedure from $\ell = k - 1$ to $0$, we see that there exists some constant $c_{k} = c_{k}\left(\q\right)$ such that for any $0 \le \ell \le k - 1$, 
        \begin{align}
            \E_{\nu_{\mathbf{q}}}\left[ \left| Y^{\left( \left\lfloor n^{a} u \right\rfloor \right)}_{\ell}\left(\Psi_{k-\ell}\left(\eta\right), n^2\right) - Y^{\left( \left\lfloor n^{a} v \right\rfloor \right)}_{\ell}\left(\Psi_{k-\ell}\left(\eta\right), n^2\right) \right|  \right] \le c_{k}   \sum_{h = 0}^{k} \left(\widetilde{M}^{u,v}_{k,n}\right)^{\left(\frac{1}{2}\right)^{h}}. 
        \end{align}
    Hence, it is sufficient to show that 
        \begin{align}
            \varlimsup_{n \to \infty} \frac{1}{n^2} \E_{\nu_{\mathbf{q}}}\left[ \left| M^{\left( \left\lfloor n^{a} u \right\rfloor \right)}_{k}\left(n^2\right) - M^{\left( \left\lfloor n^{a} v \right\rfloor \right)}_{k}\left(n^2\right) \right|^2 \right] = 0 \label{eq:stcor_pur}. 
        \end{align}

    From now on we prove \eqref{eq:stcor_pur}. From \eqref{ineq:overtaken}, \eqref{ineq:M_kl} and the triangle inequality, we have 
        \begin{align}
            &\E_{\nu_{\mathbf{q}}}\left[ \left| M^{\left( \left\lfloor n^{a} u \right\rfloor \right)}_{k}\left(n^2\right) - M^{\left( \left\lfloor n^{a} v \right\rfloor \right)}_{k}\left(n^2\right) \right|^2 \right]^{\frac{1}{2}} \\
            &\le 2\E_{\nu_{\mathbf{q}}}\left[ \left| \sum_{\ell = k + 1}^{\infty} \left( M^{\left( \left\lfloor n^{a} u \right\rfloor \right)}_{k,\ell}\left(n^2\right) - M^{\left( \left\lfloor n^{a} v \right\rfloor \right)}_{k,\ell}\left(n^2\right) \right) \right|^2 \right]^{\frac{1}{2}} \\
            &\le 2\sum_{\ell = k + 1}^{\infty} \E_{\nu_{\mathbf{q}}}\left[ \left| M^{\left( \left\lfloor n^{a} u \right\rfloor \right)}_{k,\ell}\left(n^2\right) - M^{\left( \left\lfloor n^{a} v \right\rfloor \right)}_{k,\ell}\left(n^2\right) \right|^2 \right]^{\frac{1}{2}} \\
            &\le 2\sum_{\ell = k + 1}^{\infty} \E_{\nu_{\mathbf{q}}}\left[ \left| \sum^{J_{\ell}\left( \sigma^{\left(\left\lfloor n^{a} u \right\rfloor\right)}_{k, \ell}\left( 0 \right) \right) - 1}_{j = J_{\ell}\left(  \sigma^{\left(\left\lfloor n^{a} u \right\rfloor\right)}_{k, \ell}\left(n^2 \right) \right)} \zeta_{\ell}\left( j \right) - \sum^{J_{\ell}\left( \sigma^{\left(\left\lfloor n^{a} v \right\rfloor\right)}_{k, \ell}\left( 0 \right) \right)}_{j = J_{\ell}\left(  \sigma^{\left(\left\lfloor n^{a} v \right\rfloor\right)}_{k, \ell}\left(n^2 \right) \right) - 1} \zeta_{\ell}\left( j \right) \right|^2 \right]^{\frac{1}{2}} \\
            & \ + 2\sum_{\ell = k + 1}^{\infty} \E_{\nu_{\mathbf{q}}}\left[ \left| \sum^{J_{\ell}\left( \sigma^{\left(\left\lfloor n^{a} u \right\rfloor\right)}_{k, \ell}\left( 0 \right) \right)}_{j = J_{\ell}\left(  \sigma^{\left(\left\lfloor n^{a} u \right\rfloor\right)}_{k, \ell}\left(n^2 \right) \right) - 1} \zeta_{\ell}\left( j \right) - \sum^{J_{\ell}\left( \sigma^{\left(\left\lfloor n^{a} v \right\rfloor\right)}_{k, \ell}\left( 0 \right) \right) - 1}_{j = J_{\ell}\left(  \sigma^{\left(\left\lfloor n^{a} v \right\rfloor\right)}_{k, \ell}\left(n^2 \right) \right)} \zeta_{\ell}\left( j \right) \right|^2 \right]^{\frac{1}{2}}.
        \end{align}
    Hence, it is sufficient to show 
        \begin{align}
            &\varlimsup_{n \to \infty} \frac{1}{n^2} \E_{\nu_{\mathbf{q}}}\left[ \left| \sum^{J_{\ell}\left( \sigma^{\left(\left\lfloor n^{a} u \right\rfloor\right)}_{k, \ell}\left( 0 \right) \right) - 1}_{j = J_{\ell}\left(  \sigma^{\left(\left\lfloor n u \right\rfloor\right)}_{k, \ell}\left(n^2 \right) \right)} \zeta_{\ell}\left( j \right) - \sum^{J_{\ell}\left( \sigma^{\left(\left\lfloor n^{a} v \right\rfloor\right)}_{k, \ell}\left( 0 \right) \right)}_{j = J_{\ell}\left(  \sigma^{\left(\left\lfloor n v \right\rfloor\right)}_{k, \ell}\left(n^2 \right) \right) - 1} \zeta_{\ell}\left( j \right) \right|^2 \right] = 0, \\ \label{eq:stcor_1}\\
            &\varlimsup_{n \to \infty} \frac{1}{n^2} \E_{\nu_{\mathbf{q}}}\left[ \left| \sum^{J_{\ell}\left( \sigma^{\left(\left\lfloor n^{a} u \right\rfloor\right)}_{k, \ell}\left( 0 \right) \right)}_{j = J_{\ell}\left(  \sigma^{\left(\left\lfloor n u \right\rfloor\right)}_{k, \ell}\left(n^2 \right) \right) - 1} \zeta_{\ell}\left( j \right) - \sum^{J_{\ell}\left( \sigma^{\left(\left\lfloor n^{a} v \right\rfloor\right)}_{k, \ell}\left( 0 \right) \right) - 1}_{j = J_{\ell}\left(  \sigma^{\left(\left\lfloor n v \right\rfloor\right)}_{k, \ell}\left(n^2 \right) \right)} \zeta_{\ell}\left( j \right) \right|^2 \right] = 0, \\ \label{eq:stcor_2}
        \end{align}
    for any $u > v$ and $\ell \ge k + 1$. In the following we will only show \eqref{eq:stcor_1}. We note that  \eqref{eq:stcor_2} can be proved by the same computation. 

    First we prepare an estimate for $\sigma^{\left(i\right)}_{k, \ell}\left( 0 \right)$. Since $|X^{(i)}_{h}(0)| \ge 2h |i|$ for any $i \in \Z$ and $h \in \N$, we have 
        \begin{align}
            0 \le \sigma^{\left(i\right)}_{k, \ell}\left( {\eta,} 0 \right) \le J_{k}\left({\eta}, i \right), \label{eq:stcor_3}
        \end{align}
    for any $i \ge 1$, and 
        \begin{align}
            J_{k}\left({\eta}, i \right) \le \sigma^{\left(i\right)}_{k, \ell}\left( {\eta,} 0 \right) \le 1, \label{eq:stcor_4}
        \end{align}
    for any  $i \le 0$. 
    In addition, for notational simplicity, we define 
        \begin{align} 
            I^{(i)}_{k,\ell}\left(\eta {,n} \right) :=  \sigma^{(i)}_{k, \ell}\left( {\eta,} 0 \right) -  \sigma^{(i)}_{k, \ell}\left( {\eta,} n^2 \right). 
        \end{align}
    Now we estimate \eqref{eq:stcor_1}. Observe that 
        \begin{align}
            &\sum^{J_{\ell}\left( \sigma^{\left(\left\lfloor n^{a} u \right\rfloor\right)}_{k, \ell}\left( \eta, 0 \right) \right) - 1}_{j = J_{\ell}\left(  \sigma^{\left(\left\lfloor n^{a} u \right\rfloor\right)}_{k, \ell}\left( \eta, n^2 \right) \right)}\zeta_{\ell}\left(\eta, j \right) - \sum^{J_{\ell}\left(  \sigma^{\left(\left\lfloor n^{a} v \right\rfloor\right)}_{k, \ell}\left( \eta, 0 \right) \right)}_{j = J_{\ell}\left( \sigma^{\left(\left\lfloor n^{a} v \right\rfloor\right)}_{k, \ell}\left( \eta, n^2 \right) \right) - 1} \zeta_{\ell}\left(\eta, j \right) \\
            &= \sum^{\sigma^{\left(\left\lfloor n^{a} u \right\rfloor\right)}_{k, \ell}\left( \eta, 0 \right) - 1}_{j =   \sigma^{\left(\left\lfloor n^{a} u \right\rfloor\right)}_{k, \ell}\left( \eta, n^2 \right)}\zeta_{\ell}\left(\eta, J_{\ell}\left(\eta, j\right) \right) - \sum^{ \sigma^{\left(\left\lfloor n^{a} v \right\rfloor\right)}_{k, \ell}\left( \eta, 0 \right) }_{j = \sigma^{\left(\left\lfloor n^{a} v \right\rfloor\right)}_{k, \ell}\left( \eta, n^2 \right) - 1} \zeta_{\ell}\left(\eta, J_{\ell}\left(\eta, j\right) \right) \\
            &=  \sum^{\sigma^{\left(\left\lfloor n^{a} u \right\rfloor\right)}_{k, \ell}\left( \eta, 0 \right) - 1}_{j =   \sigma^{\left(\left\lfloor n^{a} v \right\rfloor\right)}_{k, \ell}\left( \eta, 0 \right) + 1}\left( \zeta_{\ell}\left(\eta, J_{\ell}\left(\eta, j\right) \right) - \frac{1}{1 - q_\ell} \right) \\
            & \ - \sum^{ \sigma^{\left(\left\lfloor n^{a} u \right\rfloor\right)}_{k, \ell}\left( \eta, n^2 \right) }_{j = \sigma^{\left(\left\lfloor n^{a} v \right\rfloor\right)}_{k, \ell}\left( \eta, n^2 \right)} \left( \zeta_{\ell}\left(\eta, J_{\ell}\left(\eta, j\right) \right) - \frac{1}{1 - q_\ell} \right)  \\
            & \ + \frac{1}{1-q_\ell} \left( I^{\left(\left\lfloor n^{a} u \right\rfloor\right)}_{k,\ell}\left(\eta{,n}\right) - I^{\left(\left\lfloor n^{a} v \right\rfloor\right)}_{k,\ell}\left(\eta{,n}\right) \right)  \label{eq:st_cor_fin0}. 
        \end{align}
    For the first term in \eqref{eq:st_cor_fin0}, by using \eqref{eq:stcor_3} and \eqref{eq:stcor_4}, we have 
        \begin{align}
            &\left| \sum^{\sigma^{\left(\left\lfloor n^{a} u \right\rfloor\right)}_{k, \ell}\left( \eta, 0 \right) - 1}_{j =   \sigma^{\left(\left\lfloor n^{a} v \right\rfloor\right)}_{k, \ell}\left( \eta, 0 \right) + 1}\left( \zeta_{\ell}\left(\eta, J_{\ell}\left(\eta, j\right) \right) - \frac{1}{1 - q_\ell} \right) \right| \\
            &\le \sup_{ J_{k}\left( \eta,  -n^{a}\left(|u| + |v|\right) \right) \le m < m' \le J_{k}\left( \eta,  n^{a}\left(|u| + |v|\right) \right) } \left| \sum^{m'-1}_{j = m+1} \left( \zeta_{\ell}\left(\eta, J_{\ell}\left(\eta, j\right) \right) - \frac{1}{1 - q_\ell} \right) \right|.
        \end{align}
    Since $\left(\zeta_{\ell}\left( J_{\ell}\left(j\right)\right) \right)_{j \in \Z}$ are i.i.d. with mean $(1-q_\ell)^{-1}$, and $\zeta_\ell$ is independent of $\zeta_k$, by Doob's inequality we get 
        \begin{align}
            &\E_{\nu_\q}\left[ \sup_{ J_{k}\left(  -n^{a}\left(|u| + |v|\right) \right) \le m < m' \le J_{k}\left(  n^{a}\left(|u| + |v|\right) \right) } \left| \sum^{m'-1}_{j = m+1} \left( \zeta_{\ell}\left( J_{\ell}\left( j\right) \right) - \frac{1}{1 - q_\ell} \right) \right|^2 \right] \\
            &\le 4 \E_{\nu_\q}\left[ \left| \sum^{J_{k}\left(  n^{a}\left(|u| + |v|\right) \right)-1}_{j = J_{k}\left(  -n^{a}\left(|u| + |v|\right)\right) + 1} \left( \zeta_{\ell}\left( J_{\ell}\left( j\right) \right) - \frac{1}{1 - q_\ell} \right) \right|^2 \right] \\
            &\le 8 n^{a}\left(|u| + |v|\right) \bar{J}_{k}\left(\q\right) \E_{\nu_\q}\left[ \left(\zeta_{\ell}\left( J_{\ell}\left( 0\right) \right) - \frac{1}{1 - q_\ell}\right)^2 \right].
        \end{align}
    Hence we have 
        \begin{align}
            \lim_{n \to \infty} \frac{1}{n^2}\E_{{\nu_{\q}}}\left[ \left| \sum^{\sigma^{\left(\left\lfloor n^{a} u \right\rfloor\right)}_{k, \ell}\left(  0 \right) - 1}_{j =   \sigma^{\left(\left\lfloor n^{a} v \right\rfloor\right)}_{k, \ell}\left( 0 \right) + 1}\left( \zeta_{\ell}\left( J_{\ell}\left( j\right) \right) - \frac{1}{1 - q_\ell} \right) \right|^2 \right] = 0.
        \end{align}
    For the second term in \eqref{eq:st_cor_fin0}, we observe that
        \begin{align}
            &\sum^{ \sigma^{\left(\left\lfloor n^{a} u \right\rfloor\right)}_{k, \ell}\left( \eta, n^2 \right) }_{j = \sigma^{\left(\left\lfloor n^{a} v \right\rfloor\right)}_{k, \ell}\left( \eta, n^2 \right)} \left( \zeta_{\ell}\left(\eta, J_{\ell}\left(\eta, j\right) \right) - \frac{1}{1 - q_\ell} \right) \\
            &= \sum^{ \sigma^{\left(\left\lfloor n^{a} u \right\rfloor\right)}_{k, \ell}\left( \eta, n^2 \right) }_{j = \sigma^{\left(\left\lfloor n^{a} v \right\rfloor\right)}_{k, \ell}\left( \eta, n^2 \right)} \left( \zeta_{\ell - k}\left(\Psi_{k}\left(\eta\right), J_{\ell- k}\left(\Psi_{k}\left(\eta\right), j\right) \right) - \frac{1}{1 - q_\ell} \right) \\
            &= \sum^{ \tilde{\sigma}^{\left(\left\lfloor n^{a} u \right\rfloor\right)}_{k, \ell}\left( \eta, n^2 \right) }_{j = \tilde{\sigma}^{\left(\left\lfloor n^{a} v \right\rfloor\right)}_{k, \ell}\left( \eta, n^2 \right)} \left( \zeta_{\ell-k}\left(T^{n^2}\Psi_{k}\left(\eta\right), J_{\ell - k}\left(T^{n^2}\Psi_{k}\left(\eta\right), j\right) \right) - \frac{1}{1 - q_\ell} \right),
        \end{align}
    where at the second line we use \eqref{eq:semig_zeta}, and $\tilde{\sigma}^{\left(i\right)}_{k, \ell}\left( \eta, m \right)$ is defined as
        \begin{align}
            \tilde{\sigma}^{\left(i\right)}_{k, \ell}\left( \eta, m \right) := \inf\left\{ j \in \Z \ ; \ X^{(j)}_{\ell - k}\left( T^{m} \Psi_{k}\left(\tilde{\eta}\right), 0 \right) \ge J_{k}\left(\tilde{\eta}, i\right)  \right\},
        \end{align}
    for any $k < \ell$, $i \in \Z$ and $m \in \Z_{\ge 0}$. Since $\tilde{\sigma}^{\left(i\right)}_{k, \ell}\left( \eta, m \right)$ also satisfies \eqref{eq:stcor_3} and \eqref{eq:stcor_4} by replacing $\sigma^{\left(i\right)}_{k, \ell}\left( \eta, 0 \right)$ with $\tilde{\sigma}^{\left(i\right)}_{k, \ell}\left( \eta, m \right)$, we have 
        \begin{align}
            &\left| \sum^{ \tilde{\sigma}^{\left(\left\lfloor n^{a} u \right\rfloor\right)}_{k, \ell}\left( \eta, n^2 \right) }_{j = \tilde{\sigma}^{\left(\left\lfloor n^{a} v \right\rfloor\right)}_{k, \ell}\left( \eta, n^2 \right)} \left( \zeta_{\ell-k}\left(T^{n^2}\Psi_{k}\left({\eta}\right), J_{\ell - k}\left(T^{n^2}\Psi_{k}\left({\eta}\right), j\right) \right) - \frac{1}{1 - q_\ell} \right) \right| \\
            &\le \sup_{ J_{k}\left( \eta,  -n^{a}\left(|u| + |v|\right) \right) \le m < m' \le J_{k}\left( \eta,  n^{a}\left(|u| + |v|\right) \right) } \\
            & \quad \quad \quad \left| \sum^{m'-1}_{j = m+1} \left( \zeta_{\ell-k}\left(T^{n^2}\Psi_{k}\left(\eta\right), J_{\ell - k}\left(T^{n^2}\Psi_{k}\left(\eta\right), j\right) \right) - \frac{1}{1 - q_\ell} \right) \right|. 
        \end{align}
    {Since the spatial shift $\tau_{s_{\infty}\left(T^{n^2}\Psi_{k}\left(\eta\right), 0\right)}$ does not change the order of solitons in $T^{n^2}\Psi_{k}\left(\eta\right)$, we have 
        \begin{align}
            &\zeta_{\ell-k}\left(T^{n^2}\Psi_{k}\left(\eta\right), J_{\ell - k}\left(T^{n^2}\Psi_{k}\left(\eta\right), j\right) \right) \\ 
            &= \zeta_{\ell-k}\left(\tau_{s_{\infty}\left(T^{n^2}\Psi_{k}\left(\eta\right), 0\right)} T^{n^{2}}\Psi_{k}\left(\eta\right), J_{\ell - k}\left(\tau_{s_{\infty}\left(T^{n^2}\Psi_{k}\left(\eta\right), 0\right)} T^{n^{2}}\Psi_{k}\left(\eta\right), j\right) \right).
        \end{align}
    By \eqref{eq:shift_q}, Remark \ref{lem:indep_skip} and Lemma \ref{lem:T_inv}, we see that 
        \begin{align}
            \left(\zeta_{\ell-k}\left(T^{n^2}\Psi_{k}\left(\eta\right), J_{\ell - k}\left(T^{n^2}\Psi_{k}\left(\eta\right), j\right) \right)\right)_{j \in \Z} \overset{d}{=} \left( \zeta_{\ell}\left(\eta, J_{\ell}\left(\eta,j\right)\right) \right)_{j \in \Z},
        \end{align}
    under $\nu_{\q}$, and $\left(\zeta_{\ell-k}\left(T^{n^2}\Psi_{k}\left(\eta\right), J_{\ell - k}\left(T^{n^2}\Psi_{k}\left(\eta\right), j\right) \right)\right)_{j \in \Z}$ is independent of $\left(\zeta_{k}\left(\eta,j\right)\right)_{j \in \Z}$. 
    From the above discussion and} Doob's inequality, we obtain 
        \begin{align}
            &\E_{\nu_{\q}}\Bigg[ \sup_{ J_{k}\left(  -n^{a}\left(|u| + |v|\right) \right) \le m < m' \le J_{k}\left( n^{a}\left(|u| + |v|\right) \right) } \\
            & \quad \quad \quad \left| \sum^{m'-1}_{j = m+1} \left( \zeta_{\ell- k}\left(T^{n^2}\Psi_{k}\left(\eta\right), J_{\ell - k}\left(T^{n^2}\Psi_{k}\left(\eta\right), j\right) \right) - \frac{1}{1 - q_\ell} \right) \right|^2 \Bigg] \\
            &\le 8  n^{a}\left(|u| + |v|\right) \bar{J}_{k}\left(\q\right) \E_{\nu_\q}\left[ \left(\zeta_{\ell}\left( J_{\ell}\left( 0\right) \right) - \frac{1}{1 - q_\ell}\right)^2 \right].
        \end{align}
    Hence we have 
        \begin{align}
            \lim_{n \to \infty} \frac{1}{n^2}\E\left[ \left| \sum^{\sigma^{\left(\left\lfloor n^{a} u \right\rfloor\right)}_{k, \ell}\left(  n^2 \right) }_{j = \sigma^{\left(\left\lfloor n^{a} v \right\rfloor\right)}_{k, \ell}\left( n^2 \right)}\left( \zeta_{\ell}\left(J_{\ell}\left( j\right) \right) - \frac{1}{1 - q_\ell} \right) \right|^2 \right] = 0.
        \end{align}
    For the third term in \eqref{eq:st_cor_fin0}, since $\left(J_{k}(j) - J_{k}(j-1) \right)_{j \in \Z}$ are i.i.d. and have geometric distribution with mean $\left(1 - q_{k}\right) q_k^{-1}$, by using \eqref{eq:stcor_3} and \eqref{eq:stcor_4} we have 
        \begin{align}
            &\varlimsup_{n \to \infty} \E_{\nu_\q}\left[ \left( \frac{\sigma^{\left(\left\lfloor n^{a}u \right\rfloor\right)}_{k, \ell}\left(0\right) - \sigma^{\left(\left\lfloor n^{a}v \right\rfloor\right)}_{k, \ell}\left(0\right)}{n^{a}} \right)^4 \right] \\
            &\le 4\varlimsup_{n \to \infty} \E_{\nu_\q}\left[  \frac{J_{k}\left(\left\lfloor |n^{a}u| \right\rfloor\right)^4 +  J_{k}\left(\left\lfloor |n^{a}v| \right\rfloor\right)^4}{n^{4a}} \right] \\
            &\le \frac{4\left(1 - q_{k}\right)^4 \left(|u|^4 + |v|^4 \right)}{q_{k}^4}.
        \end{align}
    {Also, by Lemma \ref{lem:T_inv}, \eqref{eq:stcor_3} and \eqref{eq:stcor_4}, we get 
        \begin{align}
            &\varlimsup_{n \to \infty} \E_{\nu_\q}\left[ \left( \frac{\sigma^{\left(\left\lfloor n^{a}u \right\rfloor\right)}_{k, \ell}\left(n^2\right) - \sigma^{\left(\left\lfloor n^{a}v \right\rfloor\right)}_{k, \ell}\left(n^2\right)}{n^{a}} \right)^4 \right] \\
            &= \varlimsup_{n \to \infty} \E_{\nu_\q}\left[ \left( \frac{\tilde{\sigma}^{\left(\left\lfloor n^{a}u \right\rfloor\right)}_{k, \ell}\left(n^2\right) - \tilde{\sigma}^{\left(\left\lfloor n^{a}v \right\rfloor\right)}_{k, \ell}\left(n^2\right)}{n^{a}} \right)^4 \right] \\
            &\le 4\varlimsup_{n \to \infty} \E_{\nu_\q}\left[  \frac{J_{k}\left(\tau_{s_{\infty}\left(T^{n^2}\eta,0\right)}T^{n^2} \eta,  \left\lfloor |n^{a}u| \right\rfloor\right)^4 }{n^{4a}} \right] \\
            & \quad + 4\varlimsup_{n \to \infty} \E_{\nu_\q}\left[  \frac{  J_{k}\left(\tau_{s_{\infty}\left(T^{n^2}\eta,0\right)}T^{n^2} \eta,\left\lfloor |n^{a}v| \right\rfloor\right)^4}{n^{4a}} \right] \\
            &= 4\varlimsup_{n \to \infty} \E_{\nu_\q}\left[  \frac{J_{k}\left(\left\lfloor |n^{a}u| \right\rfloor\right)^4 +  J_{k}\left(\left\lfloor |n^{a}v| \right\rfloor\right)^4}{n^{4a}} \right] \\
            &\le \frac{4\left(1 - q_{k}\right)^4 \left(|u|^4 + |v|^4 \right)}{q_{k}^4}.
        \end{align}}
    Therefore, from \eqref{eq:st_cor_3_3}, by setting 
        \begin{align}
            A_{L,n} := \left\{  \left| \frac{1}{n^{a}} \sigma^{\left( \left\lfloor n^{a} u \right\rfloor \right)}_{k,\ell}\left(0\right) - \frac{1}{n^{a}} \sigma^{\left( \left\lfloor n^{a} v \right\rfloor \right)}_{k,\ell}\left(0\right) - f_{k,\ell}\left(\q, u-v \right) \right| > \frac{L}{\sqrt{n^{a}}}\right\},
        \end{align}     
    with some $L > 0$, and using the Schwarz inequality, we have 
        \begin{align}
            &\E_{\nu_{\q}}\left[ \left( \frac{\sigma^{\left( \left\lfloor n^{a} u \right\rfloor \right)}_{k, \ell}\left(0\right) - \sigma^{\left(\left\lfloor n^{a}v \right\rfloor\right)}_{k, \ell}\left(0\right)}{n^{a}} - f_{k,\ell}\left(\q, u-v \right) \right)^2 \right] \\
            &= \E_{\nu_{\q}}\left[ \mathbf{1}_{A_{L,n}^{c}} \left( \frac{\sigma^{\left( \left\lfloor n^{a} u \right\rfloor \right)}_{k, \ell}\left(0\right) - \sigma^{\left(\left\lfloor n^{a}v \right\rfloor\right)}_{k, \ell}\left(0\right)}{n^{a}} - f_{k,\ell}\left(\q,u -v \right) \right)^2 \right] \\
            & + \E_{\nu_{\q}}\left[ \mathbf{1}_{A_{L,n}} \left( \frac{\sigma^{\left( \left\lfloor n^{a} u \right\rfloor \right)}_{k, \ell}\left(0\right) - \sigma^{\left(\left\lfloor n^{a}v \right\rfloor\right)}_{k, \ell}\left(0\right)}{n^{a}} - f_{k,\ell}\left(\q,u -v \right) \right)^2 \right] \\
            &\le \frac{L^2}{n^{a}} + 2 \E_{\nu_{\q}}\left[ \mathbf{1}_{A_{L,n}} \left( \frac{\sigma^{\left( \left\lfloor n^{a} u \right\rfloor \right)}_{k, \ell}\left(0\right) - \sigma^{\left(\left\lfloor n^{a}v \right\rfloor\right)}_{k, \ell}\left(0\right)}{n^{a}} \right)^2 \right] \\
            & \quad + 2\nu_{\q}\left(A_{L,n}\right) \left|f_{k,\ell}\left(\q,u -v \right) \right|^2 \\
            &\le \frac{L^2}{n^{a}} + 2 \nu_{\q}\left(A_{L,n}\right)^{\frac{1}{2}} \E_{\nu_\q}\left[ \left( \frac{\sigma^{\left( \left\lfloor n^{a} u \right\rfloor \right)}_{k, \ell}\left(0\right) - \sigma^{\left(\left\lfloor n^{a}v \right\rfloor\right)}_{k, \ell}\left(0\right)}{n^{a}} \right)^4 \right]^{\frac{1}{2}} \\
            & \quad + 2\nu_{\q}\left(A_{L,n}\right) \left|f_{k,\ell}\left(\q, u-v \right) \right|^2,
        \end{align}
    and thus we obtain 
        \begin{align}
            \varlimsup_{n \to \infty} \E_{\nu_{\q}}\left[ \left( \frac{\sigma^{\left(\left\lfloor n^{a}u \right\rfloor\right)}_{k, \ell}\left(0\right) - \sigma^{\left(\left\lfloor n^{a}v \right\rfloor\right)}_{k, \ell}\left(0\right)}{n^{a}} - f_{k,\ell}\left(\q, -v \right) \right)^2 \right] = 0.
        \end{align}
    {By a similar argument and using \eqref{eq:st_cor_3_4} instead of \eqref{eq:st_cor_3_3}, we also have}
        \begin{align}
            \varlimsup_{n \to \infty} \E_{\nu_{\q}}\left[ \left( \frac{\sigma^{\left(\left\lfloor n^{a}u \right\rfloor\right)}_{k, \ell}\left(n^2\right) - \sigma^{\left(\left\lfloor n^{a}v \right\rfloor\right)}_{k, \ell}\left(n^2\right)}{n^{a}} - f_{k,\ell}\left(\q, -v \right) \right)^2 \right] = 0.
        \end{align}
    Since $0 \le a \le 1$, from the above estimates we obtain
        \begin{align}
            \varlimsup_{n \to \infty} \frac{1}{n^2} \E_{\nu_{\mathbf{q}}}\left[ \left| I^{\left(\left\lfloor n^{a}u \right\rfloor\right)}_{k,\ell}{\left(n\right)} - I^{\left(\left\lfloor n^{a} v\right\rfloor\right)}_{k,\ell}{\left(n\right)} \right|^{2} \right] = 0.
        \end{align}
    By combining the above and using the Schwarz inequality, we have \eqref{eq:stcor_1},
    and thus Theorem \ref{thm:st_cor} is proved.

    \end{proof}
    
\section{Proof of Theorem \ref{thm:Markov}}\label{sec:Markov_proof}

    In this section we show Theorem \ref{thm:Markov}. {First}, we prepare two lemmas. Before describing the lemmas, we recall that the inverse of one-step time evolution $T^{-1} : \Omega \to \Omega$ is given by 
        \begin{align}
            T^{-1}\eta\left(x\right) = \overleftarrow{\left(T\overleftarrow{\eta}\right)}\left(x\right), 
        \end{align}
    where $\overleftarrow{\eta}\left(x\right) := \eta(-x{-1})$, $x \in \Z$, see \cite[(2.12)]{CKST}. By using the carrier $W$ and \eqref{def:T_carrier}, $T^{-1}\eta$ can be represented as
        \begin{align}
            T^{-1}\eta\left(x\right) = \eta\left(x\right) - W\left(\overleftarrow{\eta}, -x\right) + W\left(\overleftarrow{\eta}, -x - 1\right).
        \end{align}
    We also recall that the ball density $\rho\left(\q\right)$ is defined in \eqref{def:ball_den}.
        
        \begin{lemma}\label{lem:markov_eta}
            {Suppose that $\q \in \mathcal{Q}_{\mathrm{M}}$ and fix $x \in \Z$. Under $\mu_{\q}$, } $\left(T^{n}\eta\left(x\right) \right)_{n \in \Z}$ is an irreducible and non-periodic two-sided Markov chain on $\{ 0,1 \}$ whose transition matrix is given by 
        \begin{align}\label{def:tran_R}
            R = \begin{pmatrix}
                R_{00} & R_{01} \\
                R_{10} & R_{11}
            \end{pmatrix} := 
            \begin{pmatrix}
                1 - \frac{{\rho\left(\q\right)}}{1 - {\rho\left(\q\right)}} & \frac{{\rho\left(\q\right)}}{1 - {\rho\left(\q\right)}} \\
                1 & 0
            \end{pmatrix},
        \end{align}
            and invariant measure $\pi \in [0,1]^2$ for $R$ is the Bernoulli measure with density ${\rho\left(\q\right)}$. 
        \end{lemma}
        \begin{proof}
            Since $\eta \in \Omega$ is a two-sided Markov chain under ${\mu_{\q}}$, $\left(\eta\left(y\right)\right)_{y \ge x}$ and $\left(\eta\left(x\right)\right)_{x \le y}$ are independent conditional on $\eta(x)$. On the other hand, since the carrier $W\left(\eta, x\right)$ is $\left(\eta\left(y\right)\right)_{y \le x}$-m'ble, we see that 
            $\left(T^{n}\eta\left(x\right) \right)_{n \ge 0}$ (resp. $\left(T^{n}\eta\left(x\right) \right)_{n \le 0}$ ) is $\left(\eta\left(y\right)\right)_{y \ge x}$-m'ble (resp. $\left(\eta\left(y\right)\right)_{y \le x}$-m'ble). Hence, {the processes} $\left(T^{n}\eta\left(x\right) \right)_{n \ge 0}$ and $\left(T^{n}\eta\left(x\right) \right)_{n \le 0}$ are independent conditional on $\eta(x)$, and this implies the Markov property at $n = 0$. Since ${\mu_{\q}}$ is $T$-invariant, $\left(T^{n+m}\eta\left(x\right) \right)_{n \ge 0}$ and $\left(T^{n+m}\eta\left(x\right) \right)_{n \le 0}$ are independent conditional on $T^{m}\eta(x)$ for any $m \in \Z$, and thus $\left(T^{n}\eta\left(x\right) \right)_{n \in \Z}$ is a two-sided Markov chain. 

            Since the invariant measure for $\left(T^{n}\eta\left(x\right) \right)_{n \in \Z}$ is the Bernoulli measure with density ${\rho\left(\q\right)}$, we can obtain $R$ by direct computation. 
        \end{proof}

        \begin{lemma}\label{lem:rec_Mix}
            Suppose that $\q \in \mathcal{Q}_{\mathrm{M}}$.  
            Then, for any {$x \in \Z$ and} $z \le 0$, the process $r\left(T^{n}\eta, x + z\right)$, $n \ge {1}$  and the event $\{ s_{\infty}(0) =z \}$ are independent conditional on ${T}\eta\left(x+z\right)$ {if $x \le 0$, and conditional on $\left(T\eta\left(y\right) ; z \le y \le x + z \right)$ if $x \ge 1$}. 
        \end{lemma} 
\begin{proof}[Proof of Lemma \ref{lem:rec_Mix}]
    
    Since $r(\eta, x) = \mathbf{1}_{\{\eta(x) = T\eta(x) = 0\}}$, the event $\{ s_{\infty}(0) = z \}$ is $\left(\eta\left(z\right),T\eta\left(z\right)\right)$-measurable. In addition, by taking the action $T^{-1}$, we see that both $\eta(z)$ and $T\eta(z)$ are $({T}\eta(y))_{y \ge z}$-measurable. Hence $\{ s_{\infty}(0) = z \}$ is $({T}\eta(y))_{y \ge z}$-measurable. On the other hand, $\left(r\left(T^{n}\eta, x+z\right)\right)$, $n \ge {1}$ is $({T}\eta(y))_{y \le x + z}$-measurable. Thus by the Markov property of ${T}\eta$, {the claim of this lemma holds.}
    \end{proof}

\subsection{Convergence of (\ref{def:step_M})}\label{subsec:IP}

    {In this subsection we will prove the weak convergence of \eqref{def:step_M}, and compute $G_{k}\left(\q\right)$. We fix $\q \in \mathcal{Q}_{\mathrm{AM}}$ and $k \in \N$ such that $k \ge K\left(\q\right)$. Thanks to Proposition \ref{prop:main} (\ref{item:prop1}) and \eqref{ineq:MiMj_1}, it suffices to consider the weak convergence under $\nu_\q$.}
    {First} we recall the formula \eqref{eq:M_kl}.
        \begin{align}
            M^{(0)}_{k}\left(\eta,n\right) 
            &= \sum_{m = 0}^{n-1} \left( 1-   r\left( T^{m}\Psi_{k}\left(\tilde{\eta}\right), J_{k}\left(\tilde{\eta}, 0\right) \right) \right).
        \end{align} 
    {We will compute the expectation of $n - M^{\left(0\right)}_{k}\left(\eta,n\right)$ under $\nu_\q$. Since $\theta^{k}\q \in \mathcal{Q}_{\mathrm{M}}$, by \eqref{eq:shift_q}, Remark \ref{lem:indep_skip}, Lemmas \ref{lem:markov_eta} and \ref{lem:rec_Mix}, for any $m \ge 1$, we get 
        \begin{align}
            &\E_{\nu_{\q}}\left[ r\left( T^{m}\Psi_{k}\left(\eta\right), J_{k}\left(0\right) \right) \right] \\
            &= \sum_{x \le -1} \nu_{\q}\left( J_{k}\left(0\right) = x \right) \E_{\nu_{\theta^k \q}}\left[ r\left( T^{m}\eta, x \right) \right] \\
            &=  \sum_{x \le -1}  \nu_{\q}\left( J_{k}\left(0\right) = x \right) \E_{\mu_{\theta^k \q}}\left[ r\left( T^{m}\eta, x \right) |s_{\infty}\left(0\right) = 0 \right] \\
            &=  \sum_{w = 0,1} \sum_{x \le -1} \frac{\mu_{\theta^k \q}\left( s_{\infty}\left(0\right) = 0, T\eta\left(x\right) = w  \right)}{ \mu_{\theta \q}\left( s_{\infty}\left(0\right) = 0 \right)}  \\
            & \quad \quad \times \nu_{\q}\left( J_{k}\left(0\right) = x \right) \E_{\mu_{\theta^k \q}}\left[ r\left( T^{m}\eta, x \right) | T\eta\left(x\right) = w \right] \\
            &= \sum_{w = 0,1} \E_{\mu_{\theta^k \q}}\left[ r\left( T^{m-1}\eta, 0 \right) | \eta\left(0\right) = w \right] \sum_{x \le -1} \nu_{\q}\left( J_{k}\left(0\right) = x \right) \nu_{\theta^k \q}\left( T\eta\left(x\right) = w  \right).
        \end{align}
    Since $\left(T^{m}\eta(0) \right)_{m \in \Z}$ is a finite ergodic Markov chain under $\mu_{\theta^k \q}$, and is strongly mixing with exponentially decay rate, for any $t > 0$ and $w = 0,1$, we have 
        \begin{align}
            \lim_{n \to \infty} \left| \frac{1}{n} \sum_{m = 0}^{\lfloor n^2 t  \rfloor -1} \left( \E_{\mu_{\theta^k \q}}\left[ r\left( T^{m}\eta, 0 \right) | \eta\left(0\right) = w \right] - \bar{r}_{k}\left(\q\right) \right) \right| = 0.
        \end{align}
    Hence we obtain  
        \begin{align}\label{eq:dif_ex}
            \lim_{n \to \infty} \left| \frac{\E_{\nu_{\theta^k \q}}\left[\lfloor n^2 t  \rfloor - M^{(0)}_{k}\left(\lfloor n^2 t  \rfloor\right)\right] - \lfloor n^2 t  \rfloor \bar{r}_{k}\left(\q\right)}{n} \right| = 0.
        \end{align}}
    {By using \eqref{eq:shift_q}, \eqref{ineq:MiMj_1}, \eqref{eq:dif_ex} and Lemma \ref{lem:rec_Mix}}, to show the weak convergence of \eqref{def:step_M} {under $\nu_{\q}$}, it is sufficient to show that the following process, 
        \begin{align}\label{def:step_r}
            t \mapsto \frac{1}{n}\sum_{m = 0}^{\lfloor n^2 t  \rfloor -1} \left( r\left( T^{m}{\eta}, {0} \right) - {\bar{r}_{k}\left(\q\right) } \right),
        \end{align}
    converges weakly to a Brownian motion under {$\mu_{\theta^k\q}$ conditional on $\left\{ \eta\left(0\right) = w \right\}$, for each $w = 0,1$}. {This} can be shown by the invariance principle for strongly mixing stationary sequences (cf. \cite[Theorem 3.1]{EK}). Therefore the {weak} convergence of \eqref{def:step_r} {under $\nu_\q$} has been shown.

    Now we compute the variance $G_{1}\left(\q\right)$ {with $\q \in \mathcal{Q}_{\mathrm{M}}$}. We observe that by Lemma \ref{lem:markov_eta}, the sum in \eqref{def:step_r} can be viewed as a functional of the ergodic Markov chain 
    $( (T^{m}\eta\left(x+z\right)$, $T^{m+1}\eta\left(x+z\right) )_{m \in \Z}$ on $\{ (0,0), (0,1), (1,0) \}$, where its transition matrix $R'$ and invariant measure $\pi' \in [0,1]^{3}$ are given by 
        \begin{align}
            R' = 
            \begin{pmatrix}
                1 - \frac{\rho\left(\theta\q\right)}{1 - \rho\left(\theta\q\right)} & \frac{\rho\left(\theta\q\right)}{1 - \rho\left(\theta\q\right)} & 0 \\
                0 & 0 & 1 \\
                1 - \frac{\rho\left(\theta\q\right)}{1 - \rho\left(\theta\q\right)} & \frac{\rho\left(\theta\q\right)}{1 - \rho\left(\theta\q\right)} & 0
            \end{pmatrix},
        \end{align}
    and 
        \begin{align}
            \pi' \left(\left(0,0\right)\right) = 1 - 2\rho\left(\theta\q\right), \quad \pi' \left(\left(0,1\right)\right) = \pi' \left(\left(1,0\right)\right) = 2\rho\left(\theta\q\right).
        \end{align}
    Since the explicit solution of the following Poisson equation, 
        \begin{align}
            \begin{pmatrix}
                1 \\ 0 \\ 0
            \end{pmatrix} = \left(\begin{pmatrix}
                1 & 0 & 0 \\
                0 & 1 & 0 \\
                0 & 0 & 1
            \end{pmatrix} - R'\right) \mathbf{f},
        \end{align}
    is given by 
        \begin{align}
            \mathbf{f} =  \begin{pmatrix}
                1  \\ - \left(1 - 2\rho\left(\theta\q\right)\right) \\ - 2\left(1 - 2\rho\left(\theta\q\right)\right) 
            \end{pmatrix},
        \end{align}
     from \cite[Theorem 1.2]{KLO12}, $G_{1}\left(\q\right)$ can be computed as 
        \begin{align}
            G_{1}\left(\q\right) &= \E_{\pi' }\left[ |\mathbf{f}|^2 \right] - \E_{\pi' }\left[ | R' \mathbf{f}|^2 \right] \\
            &= 4 \rho\left(\theta\q\right) \left(1 - \rho\left(\theta\q\right)\right) \left(1 - 2\rho\left(\theta\q\right)\right).
        \end{align}
    We note that thanks to \eqref{eq:shift_q}, the formula of $G_{k}\left(\q\right)$ for $k \in \N$ can be obtained by using $G_{k}\left(\q\right) = G_{1}\left(\theta^{k-1}\q\right)$.

\subsection{Convergence of (\ref{def:Lambda_M})}

    {We observe that $M^{i}_{k}\left(\eta, n\right)$ can be decomposed as
        \begin{align}\label{eq:dec_Mi}
            M^{i}_{k}\left(\eta, n\right) &= \sum_{j \in \Z} \mathbf{1}_{B_{i,j}} M^{\left(j\right)}_{k}\left(\eta, n\right),
        \end{align}
    where $B_{i,j}$ is defined as
        \begin{align}
            B_{i,j} &:= 
            \begin{dcases}
                \left\{  \zeta_{k}\left(\eta, J_{k}\left(\eta,1\right)\right) \ge i \right\} & i \ge 1, \ j = 1, \\
                \left\{  \sum_{z = 1}^{j} \zeta_{k}\left(\eta, J_{k}\left(\eta,z\right)\right) \ge i, \sum_{z = 1}^{j-1} \zeta_{k}\left(\eta, J_{k}\left(\eta,z\right)\right) < i \right\} & i \ge 1, \ 2 \le j, \\
                \left\{  \zeta_{k}\left(\eta, J_{k}\left(\eta,0\right)\right) \ge -i \right\} & i \le 0, \ j = 0, \\
                \left\{  \sum_{z = j}^{0} \zeta_{k}\left(\eta, J_{k}\left(\eta,z\right)\right) \ge -i, \sum_{z = j + 1}^{0} \zeta_{k}\left(\eta, J_{k}\left(\eta,z\right)\right) < -i \right\} & i \le 0, \ j \le - 1, \\
                \emptyset & \text{otherwise.}
            \end{dcases}
        \end{align}
    Here we note that since $\zeta_{k}\left(\eta, J_{k}\left(\eta, z\right)\right) \ge 1$ for any $z \in \Z$,  $B_{i,j} = \emptyset$ if $1 \le i \le j - 1$ or $j \le i \le 0$. Hence, \eqref{eq:dec_Mi} is a finite sum. By \eqref{eq:M_kl}, Remark \ref{lem:indep_skip} and \eqref{eq:dec_Mi}, we get 
        \begin{align}
            &\E_{\nu_{\q}}\left[ \exp\left( \lambda\left( n -   M^{i}_{k}\left(\eta, n\right) \right) \right) \right] \\
            &= \sum_{j \in \Z} \E_{\nu_{\q}}\left[ \mathbf{1}_{B_{i,j}} \exp\left( \lambda \sum_{m = 0}^{n - 1} r\left(T^{m}\Psi_{k}\left(\eta\right), J_{k}\left(j\right)\right) \right) \right] \\
            &\ge 
            \begin{dcases}
                \E_{\nu_{\q}}\left[ \mathbf{1}_{B_{i,1}} \exp\left( \lambda \sum_{m = 0}^{n - 1} r\left(T^{m}\Psi_{k}\left(\eta\right), J_{k}\left(1\right)\right) \right) \right] \ & \ i \ge 1, \\
                \E_{\nu_{\q}}\left[ \mathbf{1}_{B_{i,0}} \exp\left( \lambda \sum_{m = 0}^{n - 1} r\left(T^{m}\Psi_{k}\left(\eta\right), J_{k}\left(0\right)\right) \right) \right] \ & \ i \le 0.
            \end{dcases} 
        \end{align}
    By Lemma \ref{lem:rec_Mix} and similar computations used in Section \ref{subsec:IP}, we have
        \begin{align}
            &\E_{\nu_{\q}}\left[ \mathbf{1}_{B_{i,1}} \exp\left( \lambda \sum_{m = 1}^{n - 1} r\left(T^{m}\Psi_{k}\left(\eta\right), J_{k}\left(1\right)\right) \right) \right] \\
            &\ge \E_{\mu_{\theta^k\q}}\left[ \exp\left( \lambda \sum_{m = 1}^{n - 1} r\left(T^{m}\eta, 0\right) \right) \Bigg| T\eta\left(0\right) = 0 \right] \\
            & \quad \times \nu_{\q}\left(J_{k}\left(1\right) = 0, \zeta_{k}\left(0\right) \ge i \right) \nu_{\theta^{k}\q}\left( T\eta\left(0\right) = 0 \right), 
        \end{align}
    and 
        \begin{align}
            &\E_{\nu_{\q}}\left[ \mathbf{1}_{B_{i,0}} \exp\left( \lambda \sum_{m = 1}^{n - 1} r\left(T^{m}\Psi_{k}\left(\eta\right), J_{k}\left(0\right)\right) \right) \right] \\
            &\ge \E_{\mu_{\theta^k\q}}\left[ \exp\left( \lambda \sum_{m = 1}^{n - 1} r\left(T^{m}\eta, 0\right) \right) \Bigg| T\eta\left(0\right) = 0 \right] \\
            & \quad \times \nu_{\q}\left(J_{k}\left(0\right) = -1, \zeta_{k}\left(-1\right) \ge -i \right) \nu_{\theta^{k}\q}\left( T\eta\left(-1\right) = 0 \right).
        \end{align}
    Hence we obtain 
        \begin{align}
            &\varliminf_{n \to \infty} \frac{1}{n} \log\left(\E_{\nu_{\q}}\left[ \exp\left( \lambda\left( n -   M^{i}_{k}\left(\eta, n\right) \right) \right) \right]\right) \\ 
            &\ge \varliminf_{n \to \infty} \frac{1}{n} \log\left(\E_{\mu_{\theta^k\q}}\left[ \exp\left( \lambda \sum_{m = 0}^{n - 1} r\left(T^{m}\eta, 0\right) \right) \Bigg| \eta\left(0\right) = 0 \right]\right).
        \end{align}
    On the other hand, since 
        \begin{align}
            &\E_{\nu_{\q}}\left[ \mathbf{1}_{B_{i,j}} \exp\left( \lambda \sum_{m = 0}^{n - 1} r\left(T^{m}\Psi_{k}\left(\eta\right), J_{k}\left(j\right)\right) \right) \right] \\
            &= \sum_{x \in \Z} \nu_{\q}\left( B_{i,j} \cap \left\{ J_{k}\left(j\right) = x \right\} \right) \E_{\nu_{\theta^{k}\q}}\left[ \exp\left( \lambda \sum_{m = 0}^{n - 1} r\left(T^{m}\eta, x\right) \right) \right] \\
            &\le \sum_{x \in \Z} \frac{\nu_{\q}\left( B_{i,j} \cap \left\{ J_{k}\left(j\right) = x \right\} \right)}{\mu_{\theta^{k}\q}\left( s_{\infty}\left(0\right) = 0 \right)} \E_{\mu_{\theta^{k}\q}}\left[ \exp\left( \lambda \sum_{m = 0}^{n - 1} r\left(T^{m}\eta, x\right) \right) \right] \\
            &= \frac{\nu_{\q}\left( B_{i,j} \right)}{\mu_{\theta^{k}\q}\left( s_{\infty}\left(0\right) = 0 \right)} \E_{\mu_{\theta^{k}\q}}\left[ \exp\left( \lambda \sum_{m = 0}^{n - 1} r\left(T^{m}\eta, 0\right) \right) \right],
        \end{align}
    by the inequality $\varlimsup_{n \to \infty} n^{-1} \log\left(\sum_{i = 1}^{m} a^{i}_{n}\right) \le \max_{1 \le i \le m} \varlimsup_{n \to \infty} n^{-1} \log\left(a^{i}_{n}\right)$ for any $m \in \N$ and $\left(a^{i}_{n}\right)_{n \in \N} \subset \left(0,\infty\right)^{\N}$, $1 \le i \le m$, we get 
        \begin{align}
            &\varlimsup_{n \to \infty} \frac{1}{n} \log\left(\E_{\nu_{\q}}\left[ \exp\left( \lambda\left( n -   M^{i}_{k}\left(\eta, n\right) \right) \right) \right]\right) \\ 
            &\le \varlimsup_{n \to \infty} \frac{1}{n} \log\left(\E_{\mu_{\theta^k\q}}\left[ \exp\left( \lambda \sum_{m = 0}^{n - 1} r\left(T^{m}\eta, 0\right) \right)\right]\right) \\
            &\le \max_{w = 0,1} \varlimsup_{n \to \infty} \frac{1}{n} \log\left(\E_{\mu_{\theta^k\q}}\left[ \exp\left( \lambda \sum_{m = 0}^{n - 1} r\left(T^{m}\eta, 0\right) \right) \Bigg| \eta\left(0\right) = w \right]\right).
        \end{align}}
    
    {From the above, we see that} if the following limit,
            \begin{align}
             \lim_{n \to \infty} \frac{1}{n} \log\left(\E_{{\mu}_{\theta^{{k}} \q}}\left[ \exp\left( \lambda \sum_{m = 0}^{n-1} r\left(T^{m}\eta, 0\right) \right) \Bigg| \eta\left(0\right) = w \right]\right),
        \end{align}
    exists and independent of $w = 0,1$, then it coincides with $\Lambda^{M{,i}}_{\q,{k}}\left(\lambda\right)${, and $\Lambda^{M{,i}}_{\q,{k}}\left(\lambda\right)$ does not depend on $i$}. 
    By using $r\left(\eta, 0\right) = \left(1 - \eta\left(0\right)\right)\left(1 - T\eta\left(0\right)\right)$, 
    for any $w_0 \in \{0,1\}$, we have 
        \begin{align}
            &\E_{{\mu}_{\theta^{{k}} \q}}\left[ \exp\left( \lambda \sum_{m = 0}^{n-1} r\left(T^{m}\eta, 0\right) \right) \Bigg| \eta\left(0\right)  = w_{0} \right] \\
            &= \sum_{w_1,\dots,w_n} \prod_{i= 0}^{n-1}R_{w_{i}w_{i+1}} e^{\lambda \left(1-w_{i}\right) \left(1-w_{i+1}\right)} \\
            &= \left(\tilde{R}\left(\lambda\right)^{n}\right)_{w_{0}0} + \left(\tilde{R}\left(\lambda\right)^{n}\right)_{w_{0}1},
        \end{align}
    where $R_{ij}$ is defined in \eqref{def:tran_R}, and $\tilde{R}\left(\lambda\right)$ is given by 
        \begin{align}
            \tilde{R}\left(\lambda\right) := \begin{pmatrix}
                \frac{1 - 2\rho\left(\theta^{{k}} \q\right)}{1 - \rho\left(\theta^{{k}} \q\right)} e^{\lambda} &  \frac{\rho\left(\theta^{{k}} \q\right)}{1 - \rho\left(\theta^{{k}} \q\right)}  \\
                1 & 0
            \end{pmatrix}.
        \end{align}
    Hence, from \cite[Theorem 3.1.1]{DZ}, we have 
        \begin{align}
            \lim_{n \to \infty} \frac{1}{n} \log\left(\E_{{\mu}_{\theta^{{k}} \q}}\left[ \exp\left( \lambda \sum_{m = 0}^{n-1} r\left(T^{m}\eta, 0\right) \right) \Bigg| \eta\left(0\right) = w \right]\right) = \log\left(\mathrm{PF}\left(\lambda\right)\right),
        \end{align} 
    where $\mathrm{PF}(\lambda)$ is the Perron-Frobenius eigenvalue of $\tilde{R}$. 
    By a direct computation, we see that 
        \begin{align}
            \mathrm{PF}\left(\lambda\right) &=  \frac{1 - 2\rho\left(\theta^{{k}}\q\right)}{2\left(1 - \rho\left(\theta^{{k}}\q\right)\right)} \left(e^{\lambda} + \sqrt{e^{2\lambda} - 1 + \frac{1}{\left(1 - 2\rho\left(\theta^{{k}}\q\right)\right)^2}}  \right).
        \end{align}
    In particular, $\log \left( \mathrm{PF}\left(\lambda\right) \right)$ is a smooth convex function on $\R$. The convex conjugate of $\log \left( \mathrm{PF}\left(\lambda\right) \right)$ can be computed as
        \begin{align}
            I^{M}_{\q,{k}}\left(u\right) &= \sup_{\lambda \in \R} \left( \lambda u - \log \left( \mathrm{PF}\left(\lambda\right) \right) \right) \\
            &= 
            \begin{dcases}
                 \frac{u}{2} \log\left( \frac{4\rho\left(\theta^{{k}}\q\right)\left(1 - \rho\left(\theta^{{k}}\q\right)\right) u^2}{\left(1 - 2\rho\left(\theta^{{k}}\q\right)\right)^2 \left(1 - u^2\right)}  \right) \\
                 \quad - \frac{1}{2}\log\left(\frac{\rho\left(\theta^{{k}}\q\right)\left(1 + u\right)}{\left(1-\rho\left(\theta^{{k}}\q\right)\right)\left(1 - u\right)} \right) \ & \ 0 \le u < 1, \\
                 \log \left( \frac{2\left(1 - \rho\left(\theta^{{k}}\q\right)\right)}{1 - 2\rho\left(\theta^{{k}}\q\right)} \right) \ & \ u = 1, \\
                 \infty \ & \ \text{otherwise.}
            \end{dcases}
        \end{align}
    We note that the minimizer of $I^{M}_{\q,{k}}\left(u\right)$ is $1 - 2\rho\left(\theta^{{k}}\q\right)$, and from \eqref{eq:v_eff_1_r}, \eqref{eq:r_rho}, the value of minimizer coincides with $v^{\mathrm{eff}}_{1}\left(\theta^{{k-1}}\q\right)$. 

\section{Proof of Theorem \ref{thm:finite}}\label{sec:finite_proof}

    We fix $\q \in \mathcal{Q}$ satisfying the assumption of Theorem \ref{thm:finite} and define $k := \max\{ 1 \le h \le \ell - 1 \ ; \ q_{h} > 0 \}$. 

    First we claim that under ${\mu}_{\theta^{k}\q}$, $\left( \eta\left(x\right), W\left(x\right) \right)$ is an ergodic Markov chain in $x \in \Z$ on the state space, 
        \begin{align}
                    S_{\ell - k} := \left\{ \left(0,0\right), \left(0,1\right), \dots, \left(0,\ell - k - 1\right), \left(1,1\right), \left(1,2\right), \dots, \left(1,\ell - k\right)  \right\},
        \end{align}
    with transition matrix 
                \begin{align}
                    P_{\ell - k} = 
                    \begin{pmatrix}
                        P^{(1)}_{\ell - k} & P^{(2)}_{\ell - k} \\
                        P^{(3)}_{\ell - k} & P^{(4)}_{\ell - k}
                    \end{pmatrix},
                \end{align}
            where $P^{(i)}_{h}$, $i = 1, \dots 4$ are $h \times h$ matrices given by 
                \begin{align}
                    P^{(1)}_{1} =1- q_{\ell}, \quad P^{(2)}_{1} = q_{\ell}, \quad P^{(3)}_{1} = 1, \quad P^{(4)}_{1} = 0,
                \end{align}
            for $h = 1$, and
                \begin{align}
                    P^{(1)}_{h} &= 
                    \begin{pmatrix}
                        1 - q_{\ell} & 0 & \dots & 0 & 0 \\ 
                        1 & 0 & \dots & 0 & 0 \\
                        0 & 1 & \dots & 0 & 0 \\
                        \vdots & \vdots & \ddots & \vdots & \vdots \\
                        0 & 0 & \dots & 1 & 0
                    \end{pmatrix}, & P^{(2)}_{h}  &= \begin{pmatrix}
                        q_{\ell} & 0 & \dots  & 0 \\ 
                        0 & 0 & \dots  & 0 \\
                        \vdots & \vdots & \ddots & \vdots \\
                        0 & 0 & \dots & 0
                    \end{pmatrix}, \\
                    P^{(3)}_{h} &= 
                    \begin{pmatrix}
                        0 & \dots  & 0 \\ 
                        \vdots & \ddots & \vdots \\
                        0 & \dots  & 0 \\
                        0 & \dots & 1
                    \end{pmatrix}, & P^{(4)}_{h}  &= \begin{pmatrix}
                        0 & 1 & \dots  & 0 \\ 
                        0 & 0 & \ddots & 0 \\
                        0 & 0 & \dots & 1
                    \end{pmatrix},
                \end{align}
        for $h \ge 2$. 
        Actually, if we define $X_{n} = \left(x^{(1)}_{n},x^{(2)}_{n} \right)_{n \in \Z_{\ge 0}}$ is the Markov process on $S_{k'}$ with the above transition matrix, and recursively define stopping times as
            \begin{align}
                    \tau_{1} &:= \inf \left\{ m \in \N \ ; \  X_{m-1} = X_{m} = \left(0,0\right) \right\}, \\
                    \tau_{n + 1} &:= \inf \left\{ m \ge \tau_{n} + 1 \ ; \  X_{m-1} = X_{m} = \left(0,0\right) \right\},
                \end{align}
        then the distribution of $\left(X^{(1)}_{m}\right)_{\tau_{1} \le m \le \tau_{2} - 1}$ coincides with $\nu_{\theta^{k}\q} \circ \mathbf{e}^{(i)}$. In addition, $\left(X^{(1)}_{m}\right)_{\tau_{n} \le m \le \tau_{n+1} - 1}$ and $\left(X^{(1)}_{m}\right)_{\tau_{n'} \le m \le \tau_{n'} - 1}$ are independent if $n \neq n'$. Hence, from the construction of $\q$-statistics, $\left( \eta\left(x\right), W\left(x\right) \right)$, $x \in \Z$ is the desired ergodic Markov process under ${\mu}_{\theta^{k}\q}$. 
    
        On the other hand, since there are only $(\ell - k)$-solitons under $\nu_{\theta^{k}\q}$, from \eqref{eq:M_kl}, we have 
                \begin{align}
                    M^{(0)}_{k}\left(\eta,n\right) &= \sum_{m = 0}^{n-1} \left( 1 - r\left( T^{m}\Psi_{k}\left(\tilde{\eta}\right), J_{k}\left(\tilde{\eta}, 0\right) \right) \right) \\
                    &= \sum_{m = 0}^{n-1} \left( 1 - r\left(\Psi_{k}\left(\tilde{\eta}\right), J_{k}\left(\tilde{\eta}, 0\right) - (\ell - k)m \right) \right),
                \end{align}
            a.s. under $\nu_{\q}$. From the above, we see that  $M^{(0)}_{k}\left(\eta,n\right)$ is a functional of an ergodic Markov process. Therefore, by using a similar argument to Section \ref{sec:Markov_proof}, one can show that \eqref{def:step_M} converges weakly to a Brownian motion. Also, by using the relation $r(x) = \mathbf{1}_{\left\{W(x) = 0\right\}}\mathbf{1}_{\left\{W(x+1) = 0\right\}}$, for any $\mathbf{s}_{0} = (s^{(1)}_{0}, s^{(2)}_{0}) \in S_{\ell - k}$, we have 
                \begin{align}
                    &\E_{\nu_{\theta^{\ell}C_{k}\q}}\left[ \exp\left( \lambda \sum_{m = 0}^{n-1} r\left((\ell - k) m\right) \right) | \left(\eta\left(0\right), W\left(0\right)\right) = \mathbf{s}_{0} \right] \\
                    &= \sum_{\mathbf{s}_{1},\mathbf{s}_{\ell - k},\mathbf{s}_{\ell - k+1},\dots, \mathbf{s}_{(n-1)(\ell - k)}, \mathbf{s}_{(n-1)(\ell - k) + 1}} \\
                    & \quad \left( \prod_{i = 0}^{n-2} \left(P_{\ell - k}\right)_{\mathbf{s}_{i(\ell - k)}\mathbf{s}_{i(\ell - k)+1}} e^{\lambda \delta_{0}\left(s^{(2)}_{i(\ell - k)}\right)\delta_{0}\left(s^{(2)}_{i(\ell - k)+1}\right)} \left(P_{k'}\right)^{(\ell - k)-1}_{\mathbf{s}_{i(\ell - k) + 1}\mathbf{s}_{(i+1)(\ell - k)}} \right) \\
                    & \quad \quad \times \left(P_{\ell - k}\right)_{\mathbf{s}_{(n-1)k'}\mathbf{s}_{(n-1)(\ell - k)+1}} e^{\lambda \delta_{0}\left(s^{(2)}_{(n-1)(\ell - k)}\right)\delta_{0}\left(s^{(2)}_{(n-1)(\ell - k)+1}\right)} \\
                    &= \sum_{\mathbf{s}} \left(\tilde{P}_{\ell - k}\right)^{n-1}_{\mathbf{s}_{0}\mathbf{s}} \sum_{\mathbf{s}'}\left(P_{\ell - k}\right)_{\mathbf{s}\mathbf{s}'} e^{\lambda \delta_{0}\left(s^{(2)}\right)\delta_{0}\left((s')^{(2)}\right)},
                \end{align}
            where 
                \begin{align}
                    \tilde{P}_{\ell - k}\left(\lambda\right) := 
                    \begin{pmatrix}
                        \tilde{P}^{(1)}_{\ell - k}\left(\lambda\right) & \tilde{P}^{(2)}_{\ell - k}\left(\lambda\right) \\
                        \tilde{P}^{(3)}_{\ell - k} & \tilde{P}^{(4)}_{\ell - k}
                    \end{pmatrix},
                \end{align}
            and $\tilde{P}^{(i)}_{h}$, $i = 1, \dots, 4$ are given by 
                \begin{align}
                    \tilde{P}^{(1)}_{1} = e^{\lambda}\left(1- q_{k}\right), \quad \tilde{P}^{(2)}_{1} = q_{k}, \quad \tilde{P}^{(3)}_{1} = 1, \quad \tilde{P}^{(4)}_{1} = 0,
                \end{align}
            for $h = 1$, and
                \begin{align}
                    \tilde{P}^{(1)}_{h}\left(\lambda\right) &:= 
                    \begin{pmatrix}
                        e^{\lambda} \left(1 - q_{\ell}\right)^{h} & 0 & \dots & 0 \\ 
                        \left(1 - q_{\ell}\right)^{h-1} & 0 & \dots & 0 \\
                        \vdots & \vdots & \ddots & \vdots \\
                        1 - q_{\ell} & 0 & \dots & 0
                    \end{pmatrix}, \\ 
                    \tilde{P}^{(2)}_{h}\left(\lambda\right)  &:= \begin{pmatrix}
                        e^{\lambda} \left(1 - q_{\ell}\right)^{h-1} q_{\ell} & \dots & e^{\lambda} \left(1 - q_{\ell}\right)q_{\ell} & q_{\ell} \\ 
                        \left(1 - q_{\ell}\right)^{h-2} q_{\ell} & \dots  &  q_{\ell} & 0 \\
                        \vdots & \udots  & \udots & \vdots \\
                        q_{\ell} & 0 & \dots & 0
                    \end{pmatrix}, \\
                    \tilde{P}^{(3)}_{h} &:= 
                    \begin{pmatrix}
                        0 & \dots  & 1 \\ 
                        \vdots & \udots & \vdots \\
                        1 & \dots  & 0 
                    \end{pmatrix}, \quad \tilde{P}^{(4)}_{h}  := \begin{pmatrix}
                        0 &  \dots  & 0 \\ 
                        \vdots & \ddots & \vdots \\
                        0 &\dots & 0
                    \end{pmatrix},
                \end{align}
            for $h \ge 2$.
            From \cite[Theorem 3.1.1 (e)]{DZ} and the ergodicity of the Markov chain defined above, we have 
                \begin{align}
                    \Lambda^{M}_{\q,\ell}\left(\lambda\right) &= 
                    \lim_{n \to \infty} \frac{1}{n} \log\left(\E_{\nu_{\theta^{k}\q}}\left[ \exp\left( \lambda \sum_{m = 0}^{n-1} r\left((\ell - k) m\right) \right) | \left(\eta\left(0\right), W\left(0\right)\right) = \mathbf{s}_{0} \right] \right) \\ &= \log \left( \widetilde{\mathrm{PF}}\left(\lambda\right) \right),
                \end{align}
            where $\widetilde{\mathrm{PF}}(\lambda)$ is the Perron-Frobenius eigenvalue of $\tilde{P}_{k'}\left(\lambda\right)$. Since 
                \begin{align}
                    \operatorname{det}\left(\tilde{P}_{\ell - k}\left(\lambda\right) - x I_{2(\ell - k)}\right) = \operatorname{det}\left(x I_{\ell - k}\left(x I_{\ell - k} -\tilde{P}^{(1)}_{\ell - k}\left(\lambda\right) \right)  - \tilde{P}^{(2)}_{\ell - k}\tilde{P}^{(3)}_{\ell - k} \right),
                \end{align}
            where $I_{h}$ is the $h\times h$ identity matrix, by direct computation we see that 
                \begin{align}
                    \widetilde{\mathrm{PF}}\left(\lambda\right) =  \frac{\left(1 - q_{\ell}\right)^{\ell - k}e^{\lambda} }{2} + \sqrt{\frac{\left(1 - q_{\ell}\right)^{2(\ell - k)}e^{2\lambda}}{4} + q_{\ell}}.
                \end{align}
            Hence ${G}_{k}\left({\q}\right)$ can be computed as 
                \begin{align}
                    {G}_{k}\left(\q\right) &= \frac{d^2 \log \left( \widetilde{\mathrm{PF}}\left(\lambda\right) \right)}{d\lambda^2} |_{\lambda = 0} \\
                    &= \frac{4q_\ell}{\left(1 - q_{\ell}\right)^{2(\ell - k)}} \left(1 + \frac{4q_\ell}{\left(1 - q_{\ell}\right)^{2(\ell - k)}}\right)^{-\frac{3}{2}}.
                \end{align}

\section{Proof of Theorem \ref{thm:IPLDP}}\label{sec:final}

    We recall that if ${\mu}$ is a {space-homogeneous} Bernoulli product measure  or two-sided Markov distribution supported on $\Omega$, then there exists $\q \in \mathcal{Q}_{\mathrm{M}}$ such that ${\mu} ={\mu}_{\q}$ and $K(\q) = 1$. In addition, {if $\q \in \mathcal{Q}_{{\mathrm{M}}}$, then thanks to Lemma \ref{lem:expbound_ex},} we have the exponential bound of ${s_{\infty}\left(1\right)}$ {under $\nu_{\q}$}. 
        
    By combining Theorems \ref{thm:lln_lp}, \ref{thm:main}, \ref{thm:Markov}, Proposition \ref{prop:main} and Lemma \ref{lem:expbound_ex}, we have Theorem \ref{thm:IPLDP}.
    
\section*{Acknowledgment}

The authors thank to Kenkichi Tsunoda for suggesting to consider the large deviation principle of the tagged soliton. 
The work of MS has been supported by JSPS KAKENHI GRANT No. 19H0179, 23K22414, 24K21515 and 24K00528.
The work of HS has been supported by JSPS KAKENHI GRANT No. 24KJ1037, 24K16936 and 24K00528.
SO appreciates the hospitality of the University of Tokyo.

\appendix

\section{Proof of Lemma \ref{lem:dif_xi}}\label{app:lem_1}

            For notational simplicity, we only consider the case $n = 1$. We can use the same proof presented below for any $n \in \N$.

            We will quote some formulae and results from \cite{S}. First, we recall that the carrier with capacity $\ell \in \N$, which is a variant of the carrier process, is defined as 
                \begin{align}
                        W_{\ell}\left(s_{\infty}\left(i\right)\right) & := 0, \\
                        W_{\ell}\left(x\right) - W_{\ell}\left(x - 1\right) 
                        &= 
                        \begin{dcases}
                            1 \ & \ \text{if } \eta\left(x\right) = 1, W_{\ell}\left(x - 1\right) < \ell \\
                            -1 \ & \ \text{if } \eta\left(x\right) = 0, W_{\ell}\left(x - 1\right) > 0, \\
                            0 \ & \ \text{otherwise}.
                        \end{dcases}
                    \end{align}
            We note that from the construction of $\mathcal{W}_{k}, k \in \N$, for any $\ell \in \N$ and $x \in \Z$ we have the relation 
            \begin{align}\label{eq:cap_seat}
                W_{\ell}\left(x\right) = \sum_{k = 1}^{\ell} \mathcal{W}_{k}\left(x\right).
            \end{align}
            Next, from Remark \ref{rem:seat_soliton} and \cite[Lemma 4.2]{S}, we see that for any $\gamma \in \Gamma_{k}$ , $k \in \N$, $X\left(\gamma\right)$ is either a record or a $(\ell,\sigma)$-seat with $\ell \ge k$ and $\sigma \in \{0,1\}$. In particular, for any $\gamma \in \Gamma_{k}$, there exists some $i \in \Z$ such that 
                \begin{align}\label{eq:X_sk_i}
                    X\left(\gamma\right) = s_{k}(i). 
                \end{align}
            From \eqref{eq:cap_seat}, \eqref{eq:X_sk_i} and \cite[Lemma 4.2]{S}, we see that for any $\gamma \in \Gamma_{k}$, $k \in \N$ and  $1\le \ell \le k$, 
                \begin{align}
                    W_{\ell}\left(X\left(\gamma\right)\right) = 
                    \begin{dcases}
                        \ell \ & \ \text{ if } \eta\left(X\left(\gamma\right)\right) = 1, \\
                        0 \ & \ \text{ if } \eta\left(X\left(\gamma\right)\right) = 0.
                    \end{dcases}
                \end{align}
            
            From the proof of \cite[Theorem 4.{5}]{S}, for any $\ell \in \N$ and $x \in \Z$, we get 
                \begin{align}
                    \xi_{\ell}\left( T\eta, x \right) - \xi_{\ell}\left( \eta, x\right) = W_{\ell}\left(T\eta, x\right) +  W_{\ell}\left(\eta, x\right) + o_{\ell}\left(\eta\right). 
                \end{align}
            If the $i$-th $k$-soliton is free, then $X^{(i)}_{k}$ is a record, and $X^{(i)}_{k}\left(1\right) = T_{1}\left( \gamma^{(i)}_{k} \right) - 1$. We observe that from the TS algorihtm, if a $\ell$-soliton $\gamma$ is contained in $\left(H_{1}\left(\gamma^{(i)}_{k}\right), T_{k}\left(\gamma^{(i)}_{k}\right)\right)$, we have either $\gamma \subset [H_{1}\left(\gamma^{(i)}_{k}\right), T_{1}\left(\gamma^{(i)}_{k}\right))$ or $\gamma \subset \left(T_{1}\left(\gamma^{(i)}_{k}\right), T_{k}\left(\gamma^{(i)}_{k}\right)\right)$. From this observation, Remark \ref{rem:seat_soliton} and \cite[Lemma 4.2]{S}, for any $1\le \ell \le k$ we have 
                \begin{align}
                    W_{\ell}\left(\eta, X^{(i)}_{k}\left(1\right) \right) = \sum_{h = 1}^{\ell} \mathcal{W}_{h}\left(\eta, T_{1}\left(\gamma^{(i)}_{k}\right) - 1 \right) = \sum_{h = 1}^{\ell} \mathcal{W}_{h}\left(\eta, H_{k}\left(\gamma^{(i)}_{k}\right) \right)  = \ell.
                \end{align}
            In addition, since $T\eta\left(X^{(i)}_{k}\left(1\right)\right) = 0$, for any $1\le \ell \le k$ we have
                \begin{align}
                    W_{\ell}\left(T\eta, X^{(i)}_{k}\left(1\right) \right) = 0.
                \end{align}
            On the other hand, if the $i$-th $k$-soliton is not free, then $X^{(i)}_{k}\left(1\right) = X^{(i)}_{k}$. In addition, $\eta\left(X^{(i)}_{k}\right) = 1 - T \eta\left(X^{(i)}_{k}\right) = 0$. Thus for any $1\le \ell \le k$ we get 
                \begin{align}
                    W_{\ell}\left(\eta, X^{(i)}_{k}\left(1\right) \right) = 0, \quad W_{\ell}\left(T\eta, X^{(i)}_{k}\left(1\right) \right) = \ell.
                \end{align}
            From the above, for any $1\le \ell \le k$ we have 
                \begin{align}\label{eq:dif_xi_1}
                    \xi_{\ell}\left( T\eta, X^{(i)}_{k}\left(1\right) \right) - \xi_{\ell}\left( \eta, X^{(i)}_{k}\left(1\right) \right) = \ell + o_{\ell}.
                \end{align}
            Now we assume that the $i$-th $k$-soliton is not free at time $0$. Then we have $X^{(i)}_{k}\left(1\right) = X^{(i)}_{k}$,
            and thus from \eqref{eq:dif_xi_1} we obtain \eqref{eq:dif_xi} for this case. Next we assume that the $i$-th $k$-soliton is free at time $0$. In this case we obtain 
                \begin{align}
                    &\xi_{\ell}\left( \eta, X^{(i)}_{k}\left(1\right) \right) - \xi_{\ell}\left( \eta, X^{(i)}_{k} \right) \\ 
                    &= \sum_{y \in \left[X^{(i)}_{k} + 1, X^{(i)}_{k}\left(1\right) \right]} \sum_{h \in \N} \left( \eta^{\uparrow}_{\ell + h}\left(y\right) + \eta^{\downarrow}_{\ell + h}\left(y\right) \right) \\
                    &=  \sum_{y \in \left[X^{(i)}_{k} + 1, X^{(i)}_{k}\left(1\right) \right] \cap \gamma^{(i)}_{k}} \sum_{h = \ell + 1}^{k} \left( \eta^{\uparrow}_{h}\left(y\right) + \eta^{\downarrow}_{h}\left(y\right) \right) \\
                    & \ + \sum_{y \in \left[X^{(i)}_{k} + 1, X^{(i)}_{k}\left(1\right) \right] \cap \left(\gamma^{(i)}_{k}\right)^{c}}\sum_{h = \ell + 1}^{k-1} \left( \eta^{\uparrow}_{h}\left(y\right) + \eta^{\downarrow}_{h}\left(y\right) \right),
                \end{align}
            where we use the fact that in the interval $[H_{1}\left(\gamma^{(i)}_{k}\right), T_{k}\left(\gamma^{(i)}_{k}\right))$, there are only $(h,\sigma)$-seats with $h \le k$, and all $(k,\sigma)$-seats are elements of $\gamma^{(i)}_{k}$. 
            For the first term, we get 
                \begin{align}
                    \sum_{y \in \left[X^{(i)}_{k} + 1, X^{(i)}_{k}\left(1\right) \right] \cap \gamma^{(i)}_{k}} \sum_{h = \ell + 1}^{k}\left( \eta^{\uparrow}_{\ell + h}\left(y\right) + \eta^{\downarrow}_{\ell + h}\left(y\right) \right) = \sum_{y \in  \gamma^{(i)}_{k}} \sum_{h = \ell + 1}^{k}\eta^{\uparrow}_{\ell + h}\left(y\right) = k - \ell.
                \end{align}
            For the second term, we observe that if $\left[X^{(i)}_{k} + 1, X^{(i)}_{k}\left(1\right) \right] \cap \left(\gamma^{(i)}_{k}\right)^{c}$ is not empty, then each element is a component of some $h$-soliton $\gamma$ with $h < k$, and $\gamma \subset \left[X^{(i)}_{k} + 1, X^{(i)}_{k}\left(1\right) \right]$. In addition, a $h$-soliton is composed by one of each $(h',\sigma)$-seats for $1 \le h' \le h$ and $\sigma \in \{\uparrow, \downarrow\}$. Hence for any $1 \le h \le k - 1$, we have 
                \begin{align}
                    \sum_{y \in \left[X^{(i)}_{k} + 1, X^{(i)}_{k}\left(1\right) \right] \cap \left(\gamma^{(i)}_{k}\right)^{c}} \eta^{\uparrow}_{h}\left(y\right) = \sum_{y \in \left[X^{(i)}_{k} + 1, X^{(i)}_{k}\left(1\right) \right] \cap \left(\gamma^{(i)}_{k}\right)^{c}} \eta^{\downarrow}_{h}\left(y\right),
                \end{align}
            and 
                \begin{align}
                    N^{(i)}_{k,h}\left(1\right) = \sum_{y \in \left[X^{(i)}_{k} + 1, X^{(i)}_{k}\left(1\right) \right] \cap \left(\gamma^{(i)}_{k}\right)^{c}} \left(\eta^{\uparrow}_{h}\left(y\right) - \eta^{\uparrow}_{h+1}\left(y\right) \right).
                \end{align}
            Thus we get
                \begin{align}
                    &\sum_{y \in \left[X^{(i)}_{k} + 1, X^{(i)}_{k}\left(1\right) \right] \cap \left(\gamma^{(i)}_{k}\right)^{c}}\sum_{h = \ell + 1}^{k-1} \left( \eta^{\uparrow}_{h}\left(y\right) + \eta^{\downarrow}_{h}\left(y\right) \right) \\
                    &= 2\sum_{y \in \left[X^{(i)}_{k} + 1, X^{(i)}_{k}\left(1\right) \right] \cap \left(\gamma^{(i)}_{k}\right)^{c}}\sum_{h = \ell + 1}^{k-1} \eta^{\uparrow}_{ h}\left(y\right) \\
                    &= 2\sum_{y \in \left[X^{(i)}_{k} + 1, X^{(i)}_{k}\left(1\right) \right] \cap \left(\gamma^{(i)}_{k}\right)^{c}}  \sum_{h = \ell + 1}^{k-1} \left(h - \ell\right)\left(\eta^{\uparrow}_{h}\left(y\right) - \eta^{\uparrow}_{h+1}\left(y\right) \right) \\
                    &= 2 \sum_{h = \ell + 1}^{k-1} \left(h - \ell\right) N^{(i)}_{k,h}\left(1\right).
                \end{align}
            From the above, we have 
                \begin{align}
                \xi_{\ell}\left( \eta, X^{(i)}_{k}\left(1\right) \right) - \xi_{\ell}\left( \eta, X^{(i)}_{k} \right) 
                &= k - \ell + 2 \sum_{h = \ell + 1}^{k - 1} \left( h - \ell \right) N^{(i)}_{k,h}\left(1\right),
                \end{align}
            and thus from \eqref{eq:dif_xi_1} we obtain \eqref{eq:dif_xi} when the $i$-th $k$-soliton is free.

\section{Computations omitted in Section \ref{sec:general}}

    {In this section, first we prove Lemmas \ref{lem:expbound_ex_0} and \ref{lem:expbound_ex}. In Section \ref{app:Lp_X0}, we will show that if $s_{\infty}\left(1\right)$ has exponential integrability, then $X^{i}_{k}(0)$ has the finite $p$-th moment for any $p \ge 1$. }
    Then, we will derive \eqref{eq:system_r}, \eqref{eq:eff_lam} {and show Propositions \ref{prop:char_velo} and \ref{prop:main},} whose proofs were omitted in Section \ref{sec:general}. 

\subsection{Proof of Lemma \ref{lem:expbound_ex_0}}\label{app:expbound_ex_0}

    {First we consider the case $k = 1$. By \eqref{eq:ex_size}, if $\eta \in \Omega_{0}$, then we get 
        \begin{align}
            s_{\infty}\left(\Psi_{1}\left(\eta\right), 1\right) &= 1 + 2 \sum_{\ell = 1}^{\infty} \sum_{j = 0}^{s_{\infty}\left(\Psi_{\ell}\left(\Psi_{1}\left(\eta\right)\right), 1\right) - 1} \ell \zeta_{\ell}\left(\Psi_{1}\left(\eta\right), j \right) \\
            &= 1 + 2 \sum_{\ell =  2}^{\infty} \sum_{j = 0}^{s_{\infty}\left(\Psi_{\ell}\left(\eta\right), 1\right) - 1} \left(\ell - 1 \right) \zeta_{\ell}\left(\eta, j \right),
        \end{align}
    and thus we obtain 
        \begin{align}
            s_{\infty}\left(\eta, 1\right) - s_{\infty}\left(\Psi_{1}\left(\eta\right), 1\right) &= 2 \sum_{\ell = 1}^{\infty} \sum_{j = 0}^{s_{\infty}\left(\Psi_{\ell}\left(\eta\right), 1\right) - 1} \zeta_{\ell}\left(\eta, j \right) \\
            &\le s_{\infty}\left(\Psi_{1}\left(\eta\right), 1\right) + 2\sum_{j = 0}^{s_{\infty}\left(\Psi_{1}\left(\eta\right), 1\right) - 1} \zeta_{1}\left(\eta, j \right).
        \end{align}
    From the above inequality, \eqref{eq:shift_q} and Remark \ref{lem:indep_skip}, for any $\lambda' > 0$, we have 
        \begin{align}
            &\E_{\nu_{\q}}\left[ e^{\lambda s_{\infty}\left(\eta, 1\right)} \right] \\
            &\le \E_{\nu_{\q}}\left[ \exp\left(2\lambda' s_{\infty}\left(\Psi_{1}\left(\eta\right), 1\right) + 2 \lambda'\sum_{j = 0}^{s_{\infty}\left(\Psi_{1}\left(\eta\right), 1\right) - 1} \zeta_{1}\left(\eta, j \right) \right)\right] \\
            &= \E_{\nu_{\q}}\left[ e^{\left( u_{\q,1}\left(\lambda'\right) + 2\lambda'\right) s_{\infty}\left(\Psi_{1}\left(\eta\right), 1\right)} \right] \\
            &= \E_{\nu_{\theta\q}}\left[ e^{\left( u_{\q,1}\left(\lambda'\right) + 2\lambda'\right) s_{\infty}\left(\eta, 1\right)} \right],
        \end{align}
    where $u_{\q,k}$ is defined in \eqref{def:uqk}. Hence, if $\lambda' > 0$ satisfies $u_{\q,1}\left(\lambda'\right) + 2\lambda' < \lambda$, then $\E_{\nu_{\q}} \left[ e^{\lambda' {s_{\infty}\left(1\right)}} \right] < \infty$. 

    For general $k \in \N$, one can show the claim of this lemma by repeating the above computation $k$ times, so we omit the proof. }
\subsection{Proof of Lemma \ref{lem:expbound_ex}}\label{app:expbound_ex}
    
    First we consider the case $\q \in \mathcal{Q}_{\mathrm{M}}$. 
    From \cite[Lemma 3.7]{FG}, if we write $\tilde{\nu}_\q$ the distribution of $\mathbf{e}^{(0)}$ on $\mathcal{E}$ under $\nu_\q$, then {$\E_{\nu_{\q}}\left[e^{\lambda s_{\infty}\left(1\right)}\right] = \E_{\tilde{\nu}}\left[e^{\lambda |\mathbf{e}|} \right]$, and} the probability $\tilde{\nu}_\q(\mathbf{e})$, $\mathbf{e} \in \mathcal{E}$ is 
        \begin{align}
            \tilde{\nu}_\q(\mathbf{e}) &= \nu_\q\left( \eta(1) = 0 |  \eta(0) = 0 \right) \prod_{k \in \N} \left(a\left(\q\right)b\left(\q\right)^{k-1}\right)^{\zeta_{k}\left(\mathbf{e}\right)} \\
            &= \nu_\q\left( \eta(1) = 0 |  \eta(0) = 0 \right) \left(a'\left(\q\right)\right)^{\sum_{k \in \N}\zeta_{k}\left(\mathbf{e}\right)} b\left(\q\right)^{\frac{\left|\mathbf{e}\right| - 1}{2}},
        \end{align}
    where $a'\left(\q\right) := a\left(\q\right)/b\left(\q\right)$ and $\zeta_{k}\left(\mathbf{e}\right)$ is the total number of $k$-solitons in $\mathbf{e}$. We observe that $\sum_{k \in \N}\zeta_{k}\left(\mathbf{e}\right)$ is equal to the number of $1 \le x \le |\mathbf{e}|$ such that $\mathbf{e}\left(x\right) = 1$, $\mathbf{e}\left(x+1\right) = 0$ and it is known that 
        \begin{align}
            &\left| \left\{ \mathbf{e} \in \mathcal{E}\left(m\right) \ ; \ \left| \left\{ 1\le x \le 2m+1 \ ; \ \mathbf{e}\left(x\right) = 1, \mathbf{e}\left(x+1\right) = 0  \right\} \right| = z  \right\} \right| \\
            &= \frac{1}{m}
            \begin{pmatrix}
                m \\ z
            \end{pmatrix}
            \begin{pmatrix}
                m \\ z - 1
            \end{pmatrix},
        \end{align}
    for any $m \in \N$, where the right-hand side is called the Narayana numbers. 
    Hence, we get 
        \begin{align}
            &\E_{\tilde{\nu}}\left[e^{\lambda |\mathbf{e}|} \right] \nu_{\q}\left( \eta(1) = 0 |  \eta(0) = 0 \right)^{-1} \\
            &= 1+  \sum_{m = 1}^{\infty} \sum_{z = 1}^{m} \frac{1}{m}
            \begin{pmatrix}
                m \\ z
            \end{pmatrix}
            \begin{pmatrix}
                m \\ z - 1
            \end{pmatrix} \left(a'\left(\q\right)\right)^{z} \left(e^{2\lambda}b\left(\q\right)\right)^{m} \\
            &= 1 + \frac{a\left(\q\right)}{1 - e^{2\lambda}\left(a\left(\q\right) + b\left(\q\right)\right) + \sqrt{\left(1 - e^{2\lambda}\left(a\left(\q\right) + b\left(\q\right)\right)\right)^2 - 4e^{2\lambda}a\left(\q\right)b\left(\q\right)}},
        \end{align}
    where we use the fact that the generating function of the Narayana numbers $F\left(a,b\right)$ is given by 
        \begin{align}
            F\left(a,b\right) := \frac{1 - b\left(1 + a\right) - \sqrt{\left(1 - b\left(1 + a\right)\right)^2 - 4ab^2} }{2 b}.
        \end{align}
    From the above, for sufficiently small $\lambda > 0$, we have  $\E_{\tilde{\nu}}\left[e^{\lambda |\mathbf{e}|} \right] < \infty$. 

    {Next we consider the case $\q \in \mathcal{Q}_{\mathrm{AM}}$ and $K\left(\q\right) \ge 2$. Since $\theta^{K\left(\q\right)} \q \in \mathcal{Q}_{\mathrm{M}}$, there exists $\lambda > 0$ such that  
        \begin{align}
            \E_{\nu_{\theta^{K\left(\q\right)-1} \q}}\left[e^{\lambda s_{\infty}\left(1\right)}\right] < \infty. 
        \end{align}
    Hence by Lemma \ref{lem:expbound_ex_0}, there exists some $\lambda' > 0$ such that $\E_{\nu_{\q}}\left[e^{\lambda' s_{\infty}\left(1\right)}\right] < \infty$. Therefore Lemma \ref{lem:expbound_ex} is proved. }
    
\subsection{\texorpdfstring{$L^p$}{Lp} bound for \texorpdfstring{$X^{i}_{k}(0)$}{X(0)}}\label{app:Lp_X0}

    {We assume that $\E_{\nu_{\q}}\left[e^{\lambda s_{\infty}\left(1\right)}\right] < \infty$ with some $\lambda > 0$. We will show that for any $k \in \N$, $i \in \Z$ and $p \ge 1$,
        \begin{align}\label{ineq:lp_0} 
            \E_{\nu_\q}\left[ \left| X^{{i}}_{k}\left(0\right) \right|^{p} \right] < \infty.        
        \end{align}
    Before proving this, we note that by Lemma \ref{lem:expbound_ex}, \eqref{ineq:lp_0} and the Schwarz inequality, we get 
        \begin{align}
            \E_{\mu_\q}\left[ \left| X^{{i}}_{k}\left(0\right) \right|^{p} \right]^{2} \le \frac{1}{\bar{s}_{\infty}\left(\q\right)} \E_{\nu_\q}\left[ \left| s_{\infty}\left(1\right) \right|^{2} \right]\E_{\nu_\q}\left[ \left| X^{{i}}_{k}\left(0\right) \right|^{2p} \right] < \infty.
        \end{align}
    In the following we only consider the case $i \in \N$, and the case $i \in \Z_{\le 0}$ can be shown by using the same strategy. 
    We recall that $J_{k}\left(\eta,i\right)$, $i \in \Z$ is defined in \eqref{def:J_stop}. From the definition of $J_{k}\left(\eta,i\right)$ and the following inequality $s_{k}\left(\eta, x\right) \le s_{\infty}\left(\eta, x\right)$ for any $x \in \Z_{\ge 0}$, we get 
        \begin{align}
            0 < X^{{i}}_{k}\left(\eta,0\right) \le X^{\left(i\right)}_{k}\left(\eta,0\right) = s_{k}\left(\eta, J_{k}\left(\eta,i\right) \right) \le s_{\infty}\left(\eta, J_{k}\left(\eta,i\right) \right) \le \sum_{j = 0}^{J_{k}\left(\eta,i\right) - 1} \left| \mathbf{e}^{(j)}\right|.
        \end{align}
    Since $J_{k}\left(\eta,j + 1\right) - J_{k}\left(\eta,j\right)$, $j \ge 1$ and $J_{k}\left(\eta,1\right)$ are i.i.d. geometric random variables with mean $q_{k}^{-1}$, we have
        \begin{align}
            \sum_{x \in \N} x^{{2p}} \nu_{\q}\left( J_{k}\left(i\right) = x \right)^{\frac{1}{2}} < \infty.
        \end{align}
    Then by using Remark \ref{lem:ex_iid} and the Schwarz inequality, we obtain 
        \begin{align}
            \E_{\nu_\q}\left[ \left| X^{{i}}_{k}\left(0\right) \right|^{p} \right] 
            &\le \E_{\nu_\q}\left[ \left| s_{\infty}\left(J_{k}\left(i\right) \right) \right|^{p} \right] \\
            &= \sum_{x \ge i} \E_{\nu_\q}\left[ \left| s_{\infty}\left(x \right) \right|^{p} \mathbf{1}_{\left\{ J_{k}\left(i\right) = x \right\}} \right] \\
            &\le \sum_{x \ge i} \E_{\nu_\q}\left[ \left| s_{\infty}\left(x \right) \right|^{2p} \right]^{\frac{1}{2}} \nu_{\q}\left( J_{k}\left(i\right) = x \right)^{\frac{1}{2}} \\
            &\le \sum_{x \ge i} \E_{\nu_\q}\left[ \left| \sum_{j = 0}^{x-1} \left|\mathbf{e}^{(j)}\right| \right|^{2p} \right]^{\frac{1}{2}} \nu_{\q}\left( J_{k}\left(i\right) = x \right)^{\frac{1}{2}} \\
            &\le \E_{\nu_\q}\left[ \left|\mathbf{e}^{(0)} \right|^{2p} \right]^{\frac{1}{2}} \sum_{x \ge i} x^{{2p}} \nu_{\q}\left( J_{k}\left(i\right) = x \right)^{\frac{1}{2}} \\
            &< \infty.
        \end{align}
    Hence we have \eqref{ineq:lp_0}.}

\subsection{Derivation of (\ref{eq:system_r})}\label{app:rec}

    First we observe that { from \eqref{def:inv_Palm},
        \begin{align}
            \E_{{\mu}_\q}\left[ r\left(0\right) \right] = \frac{1}{\bar{s}_{\infty}\left(\q\right)}.
        \end{align}}
    Hence we have 
        \begin{align}
            \bar{r}_{k}\left(\q\right) = \frac{1}{\bar{s}_{\infty}\left(\theta^{k}\q\right)}.
        \end{align}
    Now we fix an excursion $\mathbf{e} \in \mathcal{E}$. Note that $\mathbf{e}$ can be regarded as an element of $\Omega$, by considering $\eta = \eta(\mathbf{e})$ as 
        \begin{align}
            \eta(x) = \begin{dcases}
                \mathbf{e}\left(x + 1\right) \ & \ 0 \le x \le |\mathbf{e}| - 1, \\
                0 \ & \ \text{otherwise}.
            \end{dcases}
        \end{align}
    Then, we can apply $\Psi_{\ell}$ to $\mathbf{e}$, and we will write $\Psi_{\ell}\left(\e\right)$ instead of $\Psi_{\ell}\left(\eta(\mathbf{e})\right)$. {The length of an excursion $\left|\mathbf{e}\right|$ is given by}
        \begin{align}
            \left|\mathbf{e}\right| &= 1 + \sum_{\ell = 1}^{\infty}  \sum_{x = 1}^{\left|\mathbf{e}\right| - 1} \left( \eta^{\uparrow}_{\ell}\left(x\right) + \eta^{\downarrow}_{\ell}\left(x\right) \right) \\
            &= 1 + 2 \sum_{\ell = 1}^{\infty} \sum_{j = 0}^{\left|  \Psi_{\ell}\left(\e\right)\right| - 1} \ell \zeta_{\ell}\left(j\right) \label{eq:ex_size},
        \end{align}
    where at the last line we use \eqref{eq:seat_psi} to derive
        \begin{align}
            \xi_{\ell}\left(\eta(\mathbf{e}), \left|\mathbf{e}\right| - 1 \right) 
            &= 1 + \sum_{h = \ell + 1}^{\infty} \sum_{x = 1}^{\left|\mathbf{e}\right| - 1} \left( \eta^{\uparrow}_{h}\left(x\right) + \eta^{\downarrow}_{h}\left(x\right) \right) \\
            &= \left| \Psi_{\ell}\left(\e\right)\right|.
        \end{align}
    By using the above, \eqref{eq:shift_q} and Remark \ref{lem:indep_skip}, we have 
        \begin{align}
            \frac{1}{\bar{r}_{k}\left(\q\right)} 
            &= \E_{\nu_\theta^{k}\q}\left[\left|\mathbf{e}\right|\right] = 1 + 2 \sum_{\ell = 1}^{\infty} \ell \E_{\nu_\theta^{k}\q}\left[\left| \Psi_{\ell}\left(\e\right) \right|\right] \a_{\ell}\left(\theta^{k}\q\right) \\
            &= 1 + 2 \sum_{\ell = 1}^{\infty} \left( \ell  - k\right)\frac{\a_{\ell}\left(\theta^{k}\q\right)}{\bar{r}_{\ell}\left(\q\right)}.
        \end{align}
    Hence we have \eqref{eq:system_r}.

\subsection{Proof of Proposition \ref{prop:char_velo}}\label{app:prop_1}

    First we derive \eqref{eq:v_eff}
    From \eqref{eq:lem_shift_NM} and \eqref{eq:N_kl}, we have 
        \begin{align}
            \frac{1}{n}Y^{(i)}_{k}\left(\eta, n\right)
            = \frac{k}{n}Y^{(i)}_{1}\left(\Psi_{k-1}\left(\eta\right), n\right)  + \frac{2}{n} \sum_{\ell = 1}^{k - 1} \ell \sum_{j = X^{(i)}_{k - \ell}\left(\Psi_{\ell}\left(\tilde{\eta}\right), 0 \right)  + 1}^{X^{(i)}_{k - \ell}\left(\Psi_{\ell}\left(\tilde{\eta}\right), n \right) } \zeta_{\ell}\left(\tilde{\eta},j\right).
        \end{align}
    From \eqref{eq:LLN_q}, by taking $n \to \infty$ we have 
        \begin{align}
            \frac{1}{n}Y^{(i)}_{k}\left(\eta, n\right) = v^{\mathrm{eff}}_{k}\left(\q\right) \quad  \nu_\q\text{-a.s.}
        \end{align}
    and 
        \begin{align}
            \frac{1}{n} Y^{(i)}_{1}\left( \Psi_{k-1}\left(\eta\right), n \right) = v^{\mathrm{eff}}_{1}\left(\theta^{k - 1}\q\right) \quad  \nu_\q\text{-a.s.}
        \end{align}
    In addition, since $X^{(i)}_{k - \ell}\left(\Psi_{\ell}\left(\tilde{\eta}\right), n \right)$ is $\sigma\left(\zeta_{h} \ ; \  h \ge \ell + 1 \right)$-m'ble for any $1 \le \ell \le  k - 1$ and $n \in \Z_{\ge 0}$, by \eqref{eq:LLN_q} and Remark \ref{lem:indep_skip}, we have 
        \begin{align}
            \frac{1}{n} \sum_{j = X^{(i)}_{k - \ell}\left(\Psi_{\ell}\left(\tilde{\eta}\right), 0 \right)  + 1}^{X^{(i)}_{k - \ell}\left(\Psi_{\ell}\left(\tilde{\eta}\right), n \right) } \zeta_{\ell}\left(\eta,j\right) = \a_{\ell}\left(\q\right) v^{\mathrm{eff}}_{k-\ell}\left(\theta^{\ell}\q\right)  \quad \nu_\q\text{-a.s.}
        \end{align}
    From the above, we have \eqref{eq:v_eff}. 

    Next, we show \eqref{eq:v_eff_1_r}. From \eqref{eq:lem_shift_NM} and \eqref{eq:M_kl}, we have 
        \begin{align}
            \frac{1}{n} Y^{(i)}_{1}\left( \Psi_{k-1}\left(\eta\right), n \right) = \frac{1}{n} \sum_{m = 0}^{n-1} r\left( T^{m}\Psi_{k}\left(\tilde{\eta}\right), J_{k}\left(\tilde{\eta}, i\right) \right). 
        \end{align}
    Since $Y^{(i)}_{1}\left( \Psi_{k-1}\left(\eta\right), n \right) $ converges to $v^{\mathrm{eff}}_{1}\left(\theta^{k - 1}\q\right)$, by Remark \ref{lem:indep_skip}, we see that if 
        \begin{align}\label{eq:erg_r}
            \lim_{n \to \infty} \frac{1}{n} \sum_{m = 0}^{n-1} r\left( T^{m}{\eta}, x \right) = \bar{r}_{k}\left(\q\right),
        \end{align}
    {$\nu_{\theta^{k}\q}\text{-a.s.}$} 
    for any $x \in \Z$, then we obtain \eqref{eq:v_eff_1_r}. 
    To show \eqref{eq:erg_r}, we observe that by $T$-invariance of {$\q$-statistics} and the ergodic theorem, we see that $n^{-1} \sum_{m = 0}^{n-1} r\left( T^{m}{\eta}, x \right)$ converges a.s. to $\E_{{\mu_{\theta^{k}\q}}}[r\left( T^{m}{\eta}, x \right) | \mathcal{I} ]$, where $\mathcal{I}$ is the set of invariant sets of $T$. On the other hand, since ${\mu_{\theta^{k}\q}}$ is shift-ergodic and the limit $\E_{{\mu_{\theta^{k}\q}}}[r\left( T^{m}{\eta}, x \right) | \mathcal{I} ]$ is shift-invariant, we see that $\E_{{\mu_{\theta^{k}\q}}}[r\left( T^{m}{\eta}, x \right) | \mathcal{I} ]$ is a.s. constant. Hence we have the limit \eqref{eq:erg_r} ${\mu_{\theta^{k}\q}}$-a.s., and this implies that \eqref{eq:erg_r} also holds $\nu_{{\theta^{k}}\q}\text{-a.s.}$ Thus \eqref{eq:v_eff_1_r} is proved.

\subsection{Proof of (\ref{eq:eff_lam})}\label{app:eff_lam}

    From \eqref{eq:Lambda_Y}, the derivative of $\Lambda^{Y}_{\q,k}$ with $\lambda = 0$ is given by 
        \begin{align}
            \frac{d \Lambda^{Y}_{\q,k}}{d \lambda}\left(0\right) = \frac{d U_{\q,k}}{d\lambda}\left(0\right) \frac{d \Lambda^{M}_{\q,k}}{d \lambda}\left(0\right),
        \end{align}
    where $U_{\q,k}\left(\lambda\right)$ is defined in \eqref{def:U}. 
    First we check that the expression \eqref{eq:eff_lam} is the same as \cite[(1.12)]{FNRW}. 
    We observe that $\frac{d U_{\q,k}}{d\lambda}\left(0\right)$ satisfies the following system,
        \begin{align}\label{app:U_l}
            \frac{d U_{\q,k}}{d\lambda}\left(0\right) =  k + 2\sum_{\ell = 1}^{k - 1} \left(k - \ell\right) \frac{d U_{\q,\ell}}{d\lambda}\left(0\right). 
        \end{align}
    On the other hand, from \eqref{eq:system_r}, \eqref{eq:v_eff_1_r} and \eqref{eq:lem_shift_NM}, we have 
        \begin{align}
            \frac{d \Lambda^{M}_{\q,k}}{d \lambda}\left(0\right) 
            &= \frac{d \Lambda^{Y}_{\theta^{k-1}\q,1}}{d \lambda}\left(0\right) \\
            &= \bar{r}_{k}\left(\q\right). 
        \end{align}
    Then by combining \eqref{eq:system_r} and \eqref{app:U_l}, we see that $\left( \frac{d U_{\q,k}}{d\lambda}\left(0\right), \bar{r}_{k}^{-1}\left(\q\right) \right)$ coincide with the quantities $\left( s_k, w_k \right)$ in \cite[(1.12)]{FNRW}, respectively, and that \eqref{eq:eff_lam} and \cite[(1.12)]{FNRW} are the same. 

    To show \eqref{eq:eff_lam}, it is sufficient to prove that $\frac{d U_{\q,k}}{d\lambda}\left(0\right) = v^{\mathrm{eff}}_{k}\left(C_k \q\right)$. Since $M^{(i)}_{k}\left( \ \cdot \ \right) = 0$ a.s. under $\nu_{C_{k}\q}$, we have 
        \begin{align}
            v^{\mathrm{eff}}_{k}\left(C_k \q\right) = \frac{d U_{C_k\q,k}}{d\lambda}\left(0\right).
        \end{align}
    On the other hand, from \eqref{app:U_l}, we get
        \begin{align}
            \frac{d U_{C_k\q,k}}{d\lambda}\left(0\right) = \frac{d U_{\q,k}}{d\lambda}\left(0\right). 
        \end{align}
    Hence $\frac{d U_{\q,k}}{d\lambda}\left(0\right) = v^{\mathrm{eff}}_{k}\left(C_k \q\right)$, and thus we have  \eqref{eq:eff_lam}.    

\subsection{Proof of Proposition \ref{prop:main}}\label{app:propmain}
{
\subsubsection{Proof of (\ref{item:prop1})}

    We fix $\mathbf{T} > 0$, $k \in \N$ and $i \in \Z$. We denote by $Z^{i}_{n,k}\left(\eta, \ \cdot \ \right)$ the scaled process defined in \eqref{def:step_Y}, and denote by $B_{k}\left(\ \cdot \ \right)$ the centered Brownian motion with variance $D_{k}\left(\q\right)$. We define a scaled process $\tilde{Z}_{n,k}\left(\eta, \ \cdot \ \right)$ as 
        \begin{align}
            \tilde{Z}_{n,k}\left(\eta,t\right) 
            &:= \frac{v^{\mathrm{eff}}_{k}\left(\mathbf{q}\right)}{n v^{\mathrm{eff}}_{1}\left(\theta^{k-1}\q\right)}\left(  M^{(0)}_{k}\left(\eta, \left\lfloor n^2 t \right\rfloor \right) - \E_{\nu_{\q}}\left[ M^{(0)}_{k}\left(\left\lfloor n^2 t \right\rfloor \right) \right] \right) \\
            & \ + \frac{2}{n}\sum_{h = 1}^{k - 1} \frac{v^{\mathrm{eff}}_{h}\left(\mathbf{q}\right)}{v^{\mathrm{eff}}_{1}\left(\theta^{h-1}\q\right)} \sum_{j = n + 1}^{ \max\{ \lfloor v^{\mathrm{eff}}_{k-\ell}\left(\theta^{\ell}\q\right) n^2 t   \rfloor, n \}} \left( \zeta_{h}\left(j\right) - \a_{h}\left(\q\right) \right).
        \end{align}
    From Lemmas \ref{lem:rep_W}, \ref{lem:ran_det} and \eqref{ineq:MiMj}, we see that for any $\delta > 0$, 
        \begin{align}
            \lim_{n \to \infty}\nu_{\q}\left(\sup_{0 \le t  \le \mathbf{T}} \left| Z^{i}_{n,k}\left(t\right) - \tilde{Z}_{n,k}\left(t\right) \right| > \delta \right) = 0,
        \end{align}
    i.e., $\tilde{Z}_{n,k}\left(\eta, \ \cdot \ \right)$ also converges weakly to $B_{k}\left(\ \cdot \ \right)$ under $\nu_{\q}$. In particular, for any $N \in \N$, $\lambda_{i} \in \R$, $1 \le i \le N$, $0 \le t_1 <  \dots < t_N \le \mathbf{T}$, we get 
        \begin{align}\label{lim:weak_B_nu}
            \lim_{n \to \infty} \E_{\nu_{\q}}\left[ \exp\left( \i \sum_{i = 1}^{N} \lambda_{i}  \tilde{Z}_{n,k}\left(t_{i}\right) \right) \right] = \E\left[ \exp\left( \i \sum_{i = 1}^{N} \lambda_{i}  B_{k}\left(t_{i}\right) \right) \right].
        \end{align}
    On the other hand, from Lemma \ref{lem:sp_shift}, for any measurable set $B \subset D\left([0,\mathbf{T}]\right)^{2}$, we get 
        \begin{align}
            \mu_{\q}\left( \left\{ \left( Z^{i}_{n,k}, \tilde{Z}_{n,k} \right) \in B \right\} \right) = \frac{\E_{\nu_\q}\left[ \left| \mathbf{e}^{(0)} \right| \mathbf{1}_{\left\{\left( Z^{\left(i\right)}_{n,k}, \tilde{Z}_{n,k} \right) \in B\right\}} \right]}{\bar{s}_{\infty}\left(\q\right)}.
        \end{align}
    Hence, to show the weak convergence of $\left(Z^{i}_{n,k}\left(\ \cdot \ \right)\right)_{n \in \N}$ under $\mu_{\q}$, it is sufficient to prove the weak convergence $\left(\tilde{Z}_{n,k}\left(\ \cdot \ \right)\right)_{n \in \N}$ under $\mu_{\q}$. Since the tightness of $\left(\tilde{Z}_{n,k}\left(\ \cdot \ \right)\right)_{n \in \N}$ under $\mu_{\q}$ is clear, it is sufficient to show that for any $N \in \N$, $\lambda_{i} \in \R$, $1 \le i \le N$, $0 \le t_1 <  \dots < t_N \le \mathbf{T}$, 
        \begin{align}
            \lim_{n \to \infty} \E_{\mu_{\q}}\left[ \exp\left( \i \sum_{i = 1}^{N} \lambda_{i}  \tilde{Z}_{n,k}\left(t_{i}\right) \right) \right] = \E\left[ \exp\left( \i \sum_{i = 1}^{N} \lambda_{i}  B_{k}\left(t_{i}\right) \right) \right] \label{lim:weak_B}.
        \end{align}
    For notational simplicity, we only consider the case $N = 1$. The same proof is possible in general cases. We observe that for any $m \in \N$, the event $\{ \left| \mathbf{e}^{(0)} \right| \le 2m + 1 \}$ is $\sigma\left( \zeta_{\ell}\left(j\right) ; \ell \in \N, 0 \le j \le m \right)$-m'ble. Thus from Remark \ref{lem:indep_skip}, we have 
        \begin{align}
            \E_{\mu_{\q}}\left[ \exp\left( \i  \lambda_{1}  \tilde{Z}_{n,k}\left(t_{1}\right) \right) \right]
            &= \frac{\E_{\nu_\q}\left[ \mathbf{1}_{\{\left| \mathbf{e}^{(0)} \right| \le 2n + 1\}} \left| \mathbf{e}^{(0)} \right| \right] \E_{\nu_\q}\left[ \exp\left( \i \lambda_{1} \tilde{Z}_{n,k}\left(t_{1}\right) \right) \right]}{\bar{s}_{\infty}\left(\q\right)} \\
            & \quad + \frac{\E_{\nu_\q}\left[ \mathbf{1}_{\{\left| \mathbf{e}^{(0)} \right| \ge 2n + 2\}} \left| \mathbf{e}^{(0)} \right| \exp\left( \i \lambda_{1} \tilde{Z}_{n,k}\left(t_{1}\right) \right) \right] }{\bar{s}_{\infty}\left(\q\right)}.
        \end{align}
    Since $\E_{\nu_\q}\left[\left| \mathbf{e}^{(0)} \right| \right] < \infty$, by \eqref{lim:weak_B_nu}, we obtain \eqref{lim:weak_B}. 

\subsubsection{Proof of (\ref{item:prop2})}

    By Lemma \ref{lem:sp_shift}, \eqref{eq:expmoment_YM} and the H\"{o}lder inequality, for any $p > 1$ and $\lambda \in \R$ with $p \lambda < \delta_{\q,k}$, we have 
        \begin{align}
            &\E_{\mu_{\q}}\left[ \exp\left(\lambda Y^{i}_{k}\left(n\right)\right)  \right] \\
            &= \frac{1}{s_{\infty}\left(\q\right)} \E_{\nu_{\q}}\left[ s_{\infty}\left(1\right)  \exp\left(\lambda Y^{i}_{k}\left(n\right)\right) \right] \\
            &\le \frac{1}{s_{\infty}\left(\q\right)}\E_{\nu_{\q}}\left[ s_{\infty}\left(1\right)^{\frac{p}{p-1}} \right]^{\frac{p-1}{p}} \E_{\nu_{\q}}\left[  \exp\left(p \lambda Y^{i}_{k}\left(n\right)\right) \right]^{\frac{1}{p}} \\
            &= \frac{1}{s_{\infty}\left(\q\right)}\E_{\nu_{\q}}\left[ s_{\infty}\left(1\right)^{\frac{p}{p-1}} \right]^{\frac{p-1}{p}} \E_{\nu_{\q}}\left[  \exp\left( U_{\q,k}\left( p \lambda \right) \left(n - M^{i}_{k}\left(n\right)\right) \right) \right]^{\frac{1}{p}}.
        \end{align}
    Hence we have 
        \begin{align}
            \varlimsup_{n \to \infty} \frac{1}{n} \log\left(\E_{\mu_{\q}}\left[ \exp\left(\lambda Y^{i}_{k}\left(n\right)\right)  \right] \right) \le \frac{\Lambda^{M,i}_{\q,k}\left(U_{\q,k}\left( p \lambda \right)\right)}{p}.
        \end{align}
    By taking the limit $p \downarrow 1$, for any $\lambda < \delta_{\q,k}$, we get 
        \begin{align}
            \varlimsup_{n \to \infty} \frac{1}{n} \log\left(\E_{\mu_{\q}}\left[ \exp\left(\lambda Y^{i}_{k}\left(n\right)\right)  \right] \right) \le \Lambda^{M,i}_{\q,k}\left(U_{\q,k}\left( \lambda \right)\right) = \Lambda^{Y,i}_{\q,k}\left(\lambda\right).
        \end{align}
    On the other hand, since 
        \begin{align}
            \E_{\mu_{\q}}\left[ \exp\left(\lambda Y^{i}_{k}\left(n\right)\right)  \right]
            \ge \frac{1}{s_{\infty}\left(\q\right)} \E_{\nu_{\q}}\left[ \exp\left(\lambda Y^{i}_{k}\left(n\right)\right) \right],
        \end{align}
    we obtain 
        \begin{align}
            \varliminf_{n \to \infty} \frac{1}{n} \log\left(\E_{\mu_{\q}}\left[ \exp\left(\lambda Y^{i}_{k}\left(n\right)\right) \right] \right) \ge 
            \begin{dcases}
                \Lambda^{Y,i}_{\q,k}\left(\lambda\right) \ & \ \lambda < \delta_{\q,k}, \\
                \infty \ & \ \lambda \ge \delta_{\q,k}.
            \end{dcases}
        \end{align}
    Thus for any $\lambda \in \R$, the limit $\lim_{n \to \infty} \frac{1}{n} \log\left(\E_{\mu_{\q}}\left[ \exp\left(\lambda Y^{i}_{k}\left(n\right)\right) \right] \right)$ exists in $\R \cup \{\infty\}$, and coincides with $\Lambda^{Y,i}_{\q,k}\left(\lambda\right)$. 
    Therefore by the G\"{a}rtner-Ellis theorem, under $\mu_{\q}$, $\left(Y^{i}_{k}\left(n\right) / n\right)_{n \in \N}$ satisfies the LDP with the good rate function $I^{Y,i}_{\q,k}$. 

\subsubsection{Proof of (\ref{item:prop3})}

    By using the H\"{o}lder inequality, for any $p' \ge 1$, we get 
        \begin{align}
            &\E_{\mu_{\q}}\left[ \left| \frac{Y^{i}_{k}\left(n\right)}{n} - v^{\mathrm{eff}}_{k}\left(\q\right)\right|^{p'}  \right] \\ 
            &\le \frac{1}{\bar{s}_{\infty}\left(\q\right)}\E_{\nu_{\q}}\left[ s_{\infty}\left(1\right)^{p} \right]^{\frac{1}{p}} \E_{\nu_{\q}}\left[ \left| \frac{Y^{i}_{k}\left(n\right)}{n} - v^{\mathrm{eff}}_{k}\left(\q\right)\right|^{\frac{p p'}{p-1}}  \right]^{\frac{p-1}{p}}.
        \end{align}
    Thus we have
        \begin{align}
            \lim_{n \to \infty} \E_{\mu_{\q}}\left[ \left| \frac{Y^{i}_{k}\left(n\right)}{n} - v^{\mathrm{eff}}_{k}\left(\q\right)\right|^{p'}  \right] = 0.
        \end{align}
}

\end{document}